\documentclass[10pt,oneside]{amsart}
\usepackage[a4paper, total={5.5in, 8in}]{geometry}

\usepackage{mathtools}
\usepackage{todonotes}
\usepackage{soul}
\usepackage{easyReview}
\usepackage{derivative}
\usepackage[dvipsnames]{xcolor} 
\usepackage[urlcolor=purple,linkcolor=blue, citecolor=Emerald, colorlinks=true, linktocpage=true]{hyperref}
\usepackage{tikz}
\usepackage{graphicx,amssymb,eucal,mathrsfs}
\usepackage{latexsym,amsmath,epsfig,epic,eepic}
\usepackage{amscd}
\usepackage{amsthm}
\usepackage{mathabx}
\usepackage{indentfirst}
\usepackage{epsf,graphicx,pgf}
\usepackage{pstricks,mathrsfs}
\usepackage{mathtools}
\usepackage{tikz-cd}
\usepackage{comment}
\usepackage{stmaryrd}
\usepackage[foot]{amsaddr}
\usepackage{graphicx} % Required for inserting images

\newtheorem{thmA}{Theorem}

\newtheorem{prop}{Proposition}[section]
\newtheorem*{prop*}{Proposition}
\newtheorem{lemma}[prop]{Lemma}

\newtheorem*{coro*}{Corollary}
\newtheorem*{ques*}{Question}

\theoremstyle{definition}
\newtheorem{rem}[prop]{Remark}

\newcommand{\Ham}{\text{Ham}}

\usetikzlibrary{calc,intersections, decorations.markings,decorations.pathreplacing}
\title{Braid Stability and a Floer theory of Braid Isotopies}
\author{Nicolas Grunder}
\email{nicolas.grunder@unine.ch}
\begin{document}
\begin{abstract}
   We build a Floer theory on the space of braids of periodic orbits of a Hamiltonian flow on a closed symplectic surface. The differential and continuation maps are defined by counting Floer isotopies of braids: tuples of Floer cylinders with pairwise disjoint graphs.
    With this new perspective, we prove a quantitative braid stability result that shows the persistence of braids under possibly large Hamiltonian perturbations. As an application, we show that for every $\alpha\geq 0$ there is a sequence of Hamiltonian diffeomorphisms $\phi_k$ on the two-torus $T^2$ such that $$d_{\mathrm{H}}(\operatorname{Ent}_{\leq \alpha}(T^2,\omega),\phi_k)\to \infty \quad (k \to \infty),$$ where $\operatorname{Ent}_{\leq \alpha}(T^2,\omega)\subset \Ham(T^2,\omega)$ denotes the set of Hamiltonian diffeomorphisms with topological entropy at most $\alpha$ and $d_{\mathrm{H}}$ the Hofer metric. We prove this result by studying only contractible periodic orbits, whereas analogous higher-genus statements were previously obtained using non-contractible orbits.
\end{abstract}
\maketitle
\section{Introduction}\label{section: intro}
Let $(S,\omega)$ be a closed symplectic surface of genus $g\geq 0$. For every smooth Hamiltonian $H:S^1\times S\to \mathbb R$ the equation $$\omega(X_{H_t},\cdot) = -dH_t$$ uniquely defines the Hamiltonian vector field $X_{H_t}$ whose flow we denote by $\phi_H^t$. Such time-1 maps $\phi_H^1$ are called Hamiltonian diffeomorphisms, and their collection forms the group of Hamiltonian diffeomorphisms denoted by $\Ham(S,\omega)$. On $\Ham(S,\omega)$, there is a non-degenerate bi-invariant metric called the \emph{Hofer metric}, which is defined as follows. For a Hamiltonian $H$ set $$E(H) = \int_0^1\left(\max_{x\in S}H(t,x)-\min_{x\in S}H(t,x)\right)\,dt.$$ Then the Hofer metric is defined by $$d_{\mathrm{H}}(\phi,\psi) := \inf_{H,G}E(H-G),$$ where the infimum runs over all $H$ and $G$ such that $\phi_H^1 = \phi$ and $\phi_G^1 = \psi$. For every smooth loop $x:S^1\to S$ and a capping $C$ given by a homotopy from $x$ to a fixed reference loop in its homotopy class, we define $$\mathcal A_H([x,C]) = \int_0^1H(t,x(t))\,dt-\int_C\omega.$$ In accordance with standard convention, we let $C$ be a disk in the case when $x$ is contractible.
\subsection{The Braided Floer Space}
We denote the set of one-periodic orbits of the flow $\phi_H^t$ by $\mathcal P(H)$. For $x=\{x_1,\dots,x_k\}\subset \mathcal P(H)$ one can define a \emph{(pure closed) surface braid} $$B(H,x) = \{(x_1(t),t),\dots,(x_k(t),t)\}_{t\in S^1},$$ where the loops $(x_i(t),t)$ are called \emph{strands}. The isotopy class of this braid will be denoted by $\mathcal B(H,x)$. We define the \emph{braided Floer space in the class $\beta$} by $$\operatorname{BCF}_\beta(H)= \langle \{x_1,\dots,x_k\}\subset \mathcal P(H)\,|\,\mathcal B(H,\{x_1,\dots,x_k\})= \beta\rangle_{\mathbb Z_2}.$$ In this paper, we build a Floer theory for this space and study the stability of braids of orbits $B(H,\{x_1,\dots,x_k\})$ under Hamiltonian perturbations $H+h$ of $H$, where $E(h)$ is possibly large.
  \begin{figure}[h!]
\begin{tikzpicture}
\draw[thick, name path = S] (-10,0) circle (1.3);
\draw[thick, dashed] (-8.7,0) arc (0:180:1.3 and 0.5);
\draw[thick ] (-8.7,0) arc (0:-180:1.3 and 0.5);
\path[thin, name path=a1] (-10,0.9)--(-1,0.9);
\path[thin, name path=b1] (-10,0.3)--(-1,0.3);
\path[thin, name path=c1] (-10,-0.3)--(-1,-0.3);
\path[thin, name path=d1] (-10,-0.9)--(-1,-0.9);
\path[name intersections={of=S and a1, by=A1}];
\path[name intersections={of=S and b1, by=B1}];
\path[name intersections={of=S and c1, by=C1}];
\path[name intersections={of=S and d1, by=D1}];
%%%%%%%%%%%%%%%%%%%%%%%%%%%%%%%%%%%%%%%%%%%%%%%
\draw[thick] (A1)-- (-8,0.9);
\draw[thick] (B1)-- (-8,0.3);
\draw[thick] (C1)-- (-8,-0.3);
\draw[thick] (D1)-- (-8,-0.9);
%%%%%%%%%%%%%%%%%%%%%%%%%%%%%%%%%%%%%%%%%%%%%%%
\draw[thick] (-7,0.3)-- ($(-7,0.3)!0.40!(-8,0.9)
$);
\draw[thick] (-8,0.9) -- ($(-8,0.9)!0.40!(-7,0.3)
$);
\draw[thick] (-7,0.9)-- (-8,0.3);
\draw[thick] (-7,-0.9)-- ($(-7,-0.9)!0.4!(-8,-0.3)$);
\draw[thick] (-8,-0.3)-- ($ (-8,-0.3)!0.4!(-7,-0.9)$);
\draw[thick] (-7,-0.3)-- (-8,-0.9);
%%%%%%%%%%%%%%%%%%%%%%%%%%%%%%%%%%%%%%%%%%%%%%%
\draw[thick] (-7,0.3)-- (-6,0.9);
\draw[thick] (-7,0.9)-- ($(-7,0.9)!0.4!(-6,0.3)$);
\draw[thick] (-6,0.3)-- ($(-6,0.3)!0.4!(-7,0.9)$);
\draw[thick] (-7,-0.3)-- ($(-7,-0.3)!0.4!(-6,-0.9)$);
\draw[thick] (-6,-0.9)-- ($(-6,-0.9)!0.4!(-7,-0.3)$);
\draw[thick] (-7,-0.9)-- (-6,-0.3);
%%%%%%%%%%%%%%%%%%%%%%%%%%%%%%%%%%%%%%%%%%%%%%%
\draw[thick] (-5,0.3)-- (-6,0.3);
\draw[thick] (-5,0.9)-- (-6,0.9);
\draw[thick] (-5,-0.9)-- (-6,-0.9);
\draw[thick] (-5,-0.3)-- (-6,-0.3);
%%%%%%%%%%%%%%%%%%%%%%%%%%%%%%%%%%%%%%%%%%%%%%%
\draw[thick] (-5,0.9)-- (-4,0.9);
\draw[thick] (-5,0.3)-- ($(-5,0.3)!0.4!(-4,-0.3)$);
\draw[thick] (-4,-0.3)-- ($(-4,-0.3)!0.4!(-5,0.3)$);
\draw[thick] (-5,-0.3)-- (-4,0.3);
\draw[thick] (-5,-0.9)-- (-4,-0.9);
%%%%%%%%%%%%%%%%%%%%%%%%%%%%%%%%%%%%%%%%%%%%%%%
\draw[thick] (-4,0.9)-- (-3,0.9);
\draw[thick] (-4,0.3)-- ($(-4,0.3)!0.4!(-3,-0.3)$);
\draw[thick] (-3,-0.3)-- ($(-3,-0.3)!0.4!(-4,0.3)$);
\draw[thick] (-4,-0.3)-- (-3,0.3);
\draw[thick] (-4,-0.9)-- (-3,-0.9);
%%%%%%%%%%%%%%%%%%%%%%%%%%%%%%%%%%%%%%%%%%%%%%%
\draw[thick] (-3,0.9)-- (-2,0.9);
\draw[thick] (-3,0.3)-- (-2,0.3);
\draw[thick] (-3,-0.3)-- (-2,-0.3);
\draw[thick] (-3,-0.9)-- (-2,-0.9);
%%%%%%%%%%%%%%%%%%%%%%%%%%%%%%%%%%%%%%%%%%%%%%%
\draw[thick,dashed] (-1,0.9)-- (-2,0.9);
\draw[thick,dashed] (-1,0.3)-- (-2,0.3);
\draw[thick,dashed] (-1,-0.3)-- (-2,-0.3);
\draw[thick,dashed] (-1,-0.9)-- (-2,-0.9);
%%%%%%%%%%%%%%%%%%%%%%%%%%%%%%%%%%%%%%%%%%%%%%%
\draw[thick,->] (-7.5,-1.75) -- (-3.5,-1.75) node[midway, below]{$S^1$};
\end{tikzpicture}
\caption{A (pure closed) surface braid on $S^2$ with four strands.}
\end{figure}
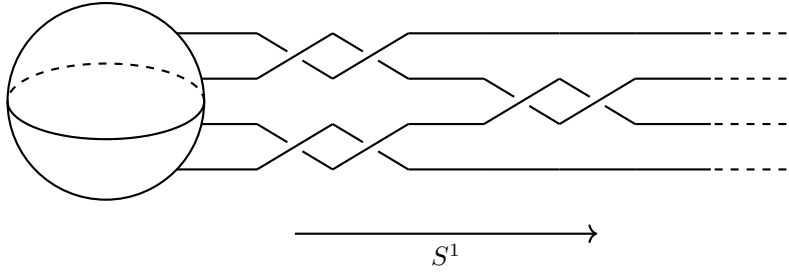
 \subsection{An Application: High-Entropy Braids Stable under Large Perturbations}
 From a topological perspective, it is very natural to study flows of vector fields through the braids traced by fixed points $\{z_1,\dots,z_k\}$ of the time-1 map of the flow. There is a push map $$\text{Push}: \mathrm{PB}_k(S,z_1,\dots,z_k)\to \operatorname{MCG}(S,z_1,\dots,z_k)$$ from the pure braid group at $z_1,\dots,z_k$ to the mapping class group of the punctured surface $S\setminus \{z_1,\dots,z_k\}$ (see \cite{BM12}). At the level of conjugacy classes, we obtain a map from closed braids to conjugacy classes of mapping classes, where $\mathcal B(H,\{z_1,\dots,z_k\})$ is mapped to the conjugacy class of $ [\phi_H^1|_{S\setminus \{z_1,\dots,z_k\}} ]$. Thus,  properties determined by the conjugacy class of $[\phi^1_H|_{S\setminus \{z_1,\dots,z_k\}}]$ can be inferred from the existence of the corresponding braid. One such property is the minimal \emph{topological entropy} in the conjugacy class (see Section~\ref{section: braids and entropy} for a definition), which measures dynamical complexity. To estimate the topological entropy, we use a notion of entropy at the level of braids. For a non-degenerate Hamiltonian $H$ this estimate takes the form $$h(\mathcal B(H,\{x_1,\dots,x_k\}))\leq h_{\mathrm{\operatorname{top}}}(\phi_H^1).$$ This motivates the study of braid stability under Hamiltonian perturbations. To the best of the author's knowledge, Alves and Meiwes first proved braid stability in \cite{AM24} for closed surfaces and the disk. Hutchings generalized braid stability to area-preserving surface diffeomorphisms of surfaces and braids which arise from periodic points rather than fixed points \cite{Hu25}. Alves and Meiwes \cite{AM24} proved that under small Hamiltonian perturbations, braids persist, thereby showing that topological entropy is lower-semicontinuous with respect to the Hofer metric. To quantify the perturbations for which braid stability is guaranteed, they had to consider the action of all periodic orbits. The machinery built in this work allows more refined estimates: braid stability is essentially quantified by the energy of Floer isotopies that end at the braid. For action estimates, this sharply restricts the set of action values we need to consider. As a proof of concept that this is in fact an improvement on the theory, we prove the following theorem by only doing action estimates of contractible orbits.
\begin{thmA}\label{thm:entropy}
    For every $\alpha\geq 0$ there is a sequence of Hamiltonian diffeomorphisms $\phi_k\in \operatorname{Ham}(T^2,\omega)$ of the torus such that $$d_{\mathrm{H}}(\operatorname{Ent}_{\leq \alpha}(T^2,\omega),\phi_k)\to \infty \quad(k\to \infty),$$ where $\operatorname{Ent}_{\leq \alpha}(T^2,\omega)$ denotes the set of Hamiltonian diffeomorphisms with topological entropy at most $\alpha$.
\end{thmA} The braid stability result proven in this work is essential to this theorem; we apply a new type of action estimate that relies on the energy of Floer isotopies rather than on the energy of Floer cylinders alone. Importantly, we prove this result by only considering contractible periodic orbits. For surfaces of genus $\geq 2$, the analogous statement was proved by Chor and Meiwes \cite{CM23}. Their work builds on that of Polterovich and Shelukhin \cite{PS16}, who used persistence modules in a Floer theory of non-contractible orbits to show that there are Hamiltonian diffeomorphisms of a closed symplectic surface of genus $g\geq 4$ that lie arbitrarily far, in the Hofer metric, from the set of \emph{autonomous} Hamiltonian diffeomorphisms. Also using non-contractible orbits, Chor \cite{Ch22} extended this result to $ g\geq 2$ and Khanevsky \cite{Kha22} to $g\geq 1$. Since autonomous Hamiltonian diffeomorphisms of surfaces have zero entropy, Theorem \ref{thm:entropy} extends the last result. That an analogous statement to that in \cite{CM23} holds for the case $g=1$ is not surprising. The novelty, however, is that we need only consider contractible orbits, making the case $g=0$, where little is known, more approachable. Next, we outline the machinery that is used in the proof of Theorem~\ref{thm:entropy}: we prove a quantitative version of braid stability by counting Floer isotopies.
\subsection{Quantitative Braid Stability and Floer Isotopies}
To study the stability of the space $\operatorname{BCF}_{\beta}(H)$, we define continuation maps between $\operatorname{BCF}_\beta(H)$ and $\operatorname{BCF}_\beta(G)$ by counting \emph{Floer isotopies}. These Floer isotopies are given by Floer cylinders $u_1,\dots,u_k$ that connect the periodic orbits of the strands. Each cylinder satisfies the Floer equation $$\partial_su_i+J_{s,t}(\partial_tu_i-X_{F_{s,t}}\circ u_i)=0,$$ for a Hamiltonian $F: \mathbb R\times S^1\times S\to \mathbb R$ and a generic family of almost complex structures $(J_{s,t})_{(s,t)\in \mathbb R\times S^1}$, and the cylinders additionally satisfy the following: for all $(s,t)\in \mathbb R\times S^1$ and $i \neq j$ $$u_i(s,t) \neq u_j(s,t).$$ In contrast to standard Floer theory, where the relevant moduli space of Floer cylinders of index $1$ is a union of intervals, the moduli space of isotopies in which each curve has index $1$ is a union of hypercubes that arise as a product of moduli spaces of Floer cylinders. 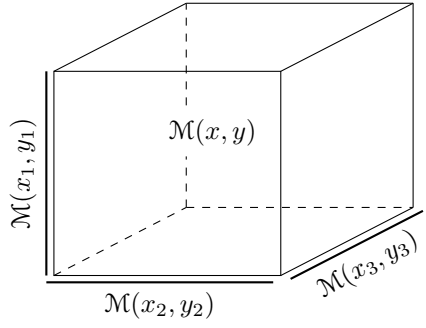
\begin{figure}[h!]
\begin{tikzpicture}[yscale = 0.9]
    \node[] at (0.6,0.6) {${\mathcal M}(x,y)$};
    \draw[thin] (-1.5,-1.5)--(1.5,-1.5)  --(1.5,1.5)--(-1.5,1.5)--cycle;
    \draw[thick](-1.6,-1.6)-- (1.4,-1.6) node[midway,sloped,below]{${\mathcal M}(x_2,y_2)$};
    \draw[thick] (-1.6,-1.5)--(-1.6,1.5) node[midway,sloped, above] {${\mathcal M}(x_1,y_1)$};
    \draw[thin] (-1.5,1.5)--(0.25,2.5);
    \draw[thin, dashed] (-1.5,-1.5)--(0.25,-0.5);
    \draw[thin, dashed](0.25,-0.5)--(0.25,0.25); \draw[thin, dashed](0.25,1)--(0.25,2.5);
    \draw[thin,dashed](0.25,-0.5)--(3.25,-0.5);
     \draw[thin] (1.5,-1.5)--(3.25,-0.5);
    \draw[thick] (1.63,-1.55)--(3.38,-0.55) node[midway,sloped,below]{${\mathcal M}(x_3,y_3)$};
    \draw[thin] (3.25,-0.5)--(3.25,2.5);
    \draw[thin] (1.5,1.5)--(3.25,2.5);
    \draw[thin] (3.25,2.5)--(0.25,2.5);
\end{tikzpicture}
\caption{A cube as a moduli space.}
\end{figure} 

These continuation maps preserve the \emph{Conley--Zehnder index $\mu = \mu_{CZ}$} of each strand. This index is a symplectic twisting number associated to a periodic orbit $x\in \mathcal P(H)$ and takes values in $\mathbb Z$ for $g \geq 1$ and in $\mathbb Z_4$ for $g = 0$.  Since continuation maps preserve the index of every strand, the maps preserve the \emph{framing} $$f_x:x_i\mapsto \mu(x_i)$$ of the underlying sets $\{x_1,\dots,x_k\}\subset \mathcal P(H)$ of the braids of orbits. A \emph{(closed pure) $\mathbb Z$-framed surface braid} or \emph{(closed pure) $\mathbb Z_4$-framed surface braid} is a (closed pure) surface braid $b$ with strands $\{s_1,\dots,s_k\}$ equipped with a map, called \emph{framing}, $$f:\{s_1,\dots,s_k\}\to \mathbb Z\quad \text{ or }\quad f:\{s_1,\dots,s_k\}\to \mathbb Z_4.$$ When it is clear from the context, we drop the prefix $\mathbb Z$ or $\mathbb Z_4$. The framing-preserving isotopies induce an equivalence relation on framed surface braids; the equivalence classes are called \emph{(closed pure) framed surface braid classes}, denoted by the tuple $(\beta,f)$. For every framed braid class $(\beta,f)$ where the framing is $\mathbb Z$-valued when $g \geq 1$ and $\mathbb Z_4$-valued when $g=0$ we define $$\operatorname{BCF}_{\beta,f}(H) = \langle x\subset \mathcal P(H)\,|\,\mathcal B(H,x) = \beta,f_x = f\rangle_{\mathbb Z_2}.$$ 
For every non-degenerate Hamiltonian $H$ and a set $x= \{x_1,\cdots,x_k\}\subset \mathcal P(H)$, we will define the \emph{stability radius $\mathcal S(H,x)\in \mathbb R\cup\{\infty\}$ of the braid $B(H,x)$}  (see below for the definition when the framing $f_x$ is injective and Section~\ref{section braid stability} for the general definition) such that the following theorem holds.
\begin{thmA}[Quantitative Braid Stability]\label{thm: braid stability}
    Let $H$ be a non-degenerate Hamiltonian on a closed symplectic surface $(S,\omega)$ and let $x\subset \mathcal P(H)$. Let $\beta = \mathcal B(H,x)$ and $f=f_x$ be the associated braid class and framing. For every non-degenerate Hamiltonian $G$ such that $E(H-G)<\mathcal S(H,x)$ we have  $$\dim \operatorname{BCF}_{\beta,f}(G)\geq 1.$$
\end{thmA}
The essential improvement over the results in \cite{AM24} is that we can estimate the stability radius $\mathcal S(H,x)$ by the action differences of orbits that appear in $\beta$-braids, as we explain next for an injective framing $f$. For a smooth family of $\omega$-compatible almost complex structures $(J_t)_{t\in S^1}$ we consider Floer isotopies consisting of cylinders $u=(u_1,\dots, u_k)$ that solve the Floer equation $$\partial_su_i+J_t(\partial_tu_i-X_{H_t}\circ u_i) = 0.$$  For a smooth cylinder $u:\mathbb R \times S^1\to S$ the energy is defined by  $$E(u) = \int_\mathbb R\int _0^1 \omega(\partial_su,J_t\partial_su)\,dtds.$$ We define the energy of a Floer isotopy by $$E(u) =\sum_{i = 1}^kE(u_i).$$ When $E(u)$ is finite, the cylinders $u_i$ converge to periodic orbits $\lim_{s\to -\infty}u_i(s,\cdot) = x_i$ and $\lim_{s\to +\infty}u_i(s,\cdot) = y_i$ of $H$. In this case, one can define the index $\mu(u_i)$ of $u_i$ by $\mu([x_i,C_{x_i}])-\mu([y_i,C_{x_i}\#u_i])$, where $C_{x_i}$ is an auxiliary choice of capping of $x_i$ on which the difference does not depend. The maximal index of a Floer isotopy is defined by $$\mu_{\operatorname{max}}(u) = \max_i \mu(u_i).$$ The central quantity in defining the stability radius is the minimal energy of non-trivial Floer isotopies with one end in $x$ for which the Conley--Zehnder index drops by at most $1$ for each strand  $$\mathcal E(H,J,x) = \min \{E(u)\,|\, u \text{ is a Floer isotopy with one end given by $x$ and $\mu_{\operatorname{max}}(u) =1$}\}.$$ This generalizes the minimal energy of Floer cylinders connecting single orbits of index difference one used in similar previous work (see for example \cite{Kha22}). Moreover, we define the \emph{interior minimal energy of Floer cylinders of $x$} by $$\mathcal E^{\operatorname{int}}(H,J,x) = \min \{E(u)\, |\, u \in \mathcal M(x_i,x_j),\; \mu([x_i,C])-\mu([x_j,C\#u]) = 1\},$$ where $C$ is any choice of capping of $x_i$. The stability radius for an injective framing is now defined by \begin{align}\label{eq: stability radius injective framing} \mathcal S(H,x) = \frac 1 k \sup_J\min \{\mathcal E(H,J,x),\mathcal E^{\operatorname{int}}(H,J,x)\},\end{align} where the supremum runs over all loops of $\omega$-compatible almost complex structures such that $(H,J)$ is regular and $k = |x|$ is the number of strands. When $f$ is not an injective framing, the minimum in the definition of $\mathcal S (H, x)$ includes an additional term $\mathcal E^2 (H, J, x)$ that is sensitive to configurations of two Floer cylinders with ends in $x$. The general definition is given in Section~\ref{section braid stability}.

\subsection{Braided Floer Homology}
Although this is not relevant for our approach to proving braid stability, for non-degenerate Hamiltonians and a generic almost complex structure $J$ as auxiliary data, we define a differential $\partial$ on $\operatorname{BCF}_\beta(H)$ that satisfies $\partial^2 =0$. This differential counts Floer isotopies given by Floer cylinders $u_1,\dots,u_k$ which all drop the Conley--Zehnder index by $1$. We call the associated homology \emph{braided Floer homology in the class $\beta$} and denote it by $\operatorname{BHF}_\beta(H,J)$. To avoid Novikov coefficients, in the case $S=T^2$ we assume that all the strands of $\beta$ are contractible.
We note that with the Floer-theoretic setup of $\operatorname{BCF}_\beta(H)$, we recover a refined version of the Arnold conjecture for surfaces, which already follows from the work of Connery-Grigg \cite{CG24}: \\ \begin{align*}\begin{split}&\text{The flow of any non-degenerate time-1 periodic Hamiltonian of a closed symplectic surface} \\ &\text{of genus $g$ has at least $2g+2$ contractible orbits of period $1$ that form the trivial braid.}\end{split}\end{align*}\\
Indeed, this is an immediate consequence of the following result about braided Floer homology.\begin{thmA}\label{thm: braided floer homology}
    Let $(S,\omega)$ be a closed symplectic surface of genus $g$ and let $\beta_0$ be the closed braid with $2g+2$ constant strands. For every non-degenerate $H$ and generic almost complex structure $J$, we have $$\dim \operatorname{BHF}_{\beta_0}(H,J)\geq 1.$$
\end{thmA}

Importantly, the homology theory constructed in this paper differs from the braid Floer homology defined by van den Berg, Ghrist, Vandervorst, and Wójcik in \cite{BGVW14}. In \cite{BGVW14}, a homology theory is defined that is independent of the Hamiltonian $H$, and the differential essentially counts braid isotopies where the index drops by $1$ in a single orbit. Our homology theory depends on the Hamiltonian $H$ and counts tuples of Floer cylinders pairwise connecting strands in braids of periodic orbits, where the index drops by $1$ in every strand.
Next, we outline what we think are the essential new perspectives and ideas in this work.
\subsection{Elements of the proofs}
All proofs build on standard Hamiltonian Floer theory with additional topological and combinatorial input. In this work, the differential $\partial$ and continuation maps $\Phi$ of braided Floer spaces are defined by counting tuples of Floer cylinders $(u_1,\dots ,u_k)$ all of the same index such that \begin{align}\label{eq: intro disjointness property} u_i(s,t) \neq u_j(s,t)\end{align} for all $(s,t)\in \mathbb R \times S^1$ and $i \neq j$. This property can be expressed by a topological intersection number of the graphs of the cylinders. 
\subsubsection{The Intersection Number}
We consider a purely topological intersection number of cylinders $u_i$, which is a priori only defined when the cylinders converge to distinct loops in the limit $\lim_{s\to \pm \infty}u_i(s,t) \neq \lim_{s \to \pm \infty}u_j(s,t)$ for $i \neq j$. The intersection number $\iota(u_i,u_j)$ is the algebraic count of intersections of the graphs $\hat u_i$ and $\hat u_j$ of $u_i$ and $u_j$. Since $u_i$ and $u_j$ are Floer cylinders, their graphs are pseudoholomorphic (for a suitable almost complex structure), and the intersection number is non-negative. This means that (\ref{eq: intro disjointness property}) holds if and only if $\iota(u_i,u_j) = 0$ for all $i\neq j$. In many proofs, we need to study configurations of two broken trajectories $u_i\#v_i$ for $i = 1,2$ of Floer cylinders that break at the same orbit $y$. 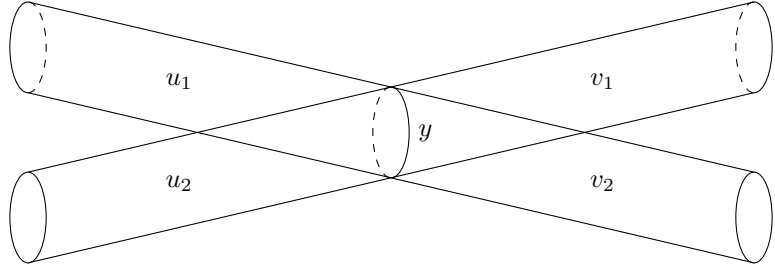
\begin{figure}[h!]
\begin{tikzpicture}[xscale = 0.8,yscale = 0.75]
    \draw[thin] (6,-1.5) ellipse (0.3 and 0.8);
    \path[thin] (6,-0.7) arc (-90:90:0.3 and -0.8) ;
    \draw[thin] (-6,-1.5) ellipse (0.3 and 0.8);
    \path[thin] (-6,-0.7) arc (90:-90:-0.3 and 0.8) ;
    \draw[thin,dashed] (-6,0.7) arc (-90:90:0.3 and 0.8);
    \draw[thin] (-6,2.3) arc (90:-90:-0.3 and 0.8);
    \draw[thin] (6,0.7) arc (-90:90:0.3 and 0.8) ;
    \draw[thin,dashed] (6,2.3) arc (90:-90:-0.3 and 0.8);
    \draw[thin] (0,-0.8) arc (-90:90:0.3 and 0.8) node[midway, right]{$y$};
    \draw[thin, dashed] (0,0.8) arc (90:-90:-0.3 and 0.8);
    \draw[thin] (-6,2.3) -- (6,-0.7);
    \draw[thin] (-6,0.7) -- (6,-2.3);
    \draw[thin] (-6,-2.3) -- (6,0.7);
    \draw[thin] (-6,-0.7) -- (6,2.3);
    \node[] at (-3.5,0.9) {$u_1$};
    \node[] at (-3.5,-0.9) {$u_2$};
    \node[] at (3.5,0.9) {$v_1$};
    \node[] at (3.5,-0.9) {$v_2$};
\end{tikzpicture}
\caption{Two cylinders breaking at a single orbit $y$.}

\end{figure} In this situation, the intersection number equals the count of intersections away from $y$, plus an asymptotic term $\lambda(u_1\#v_1,u_2\#v_2)$ accounting for intersections near $y$: $$\iota(u_1\#v_1,u_2\#v_2) = |\hat u_1\cap \hat u_2|+ |\hat v_1\cap\hat  v_2| + \lambda (u_1\#v_1,u_2\#v_2).$$ This asymptotic term is defined purely topologically: near a breaking orbit $y$  we associate to the ends of cylinders $\{u_i\}_{i = 1,2}$ and $\{v_i\}_{i=1,2}$ two two-component links $L^{u}$ and $L^{v}$ in 
$S^3$ for the positive and negative ends of $\{u_i\}_{i = {1,2}}$ and $\{v_i\}_{i = 1,2}$. The asymptotic term is the difference of the linking numbers of these links: $\lambda(u_1\#v_1,u_2\#v_2) = \ell(L^v)-\ell(L^u)$. This perspective is essential to the proofs of the main theorems in this paper. \begin{rem}
   The asymptotic linking number of curves near a limiting orbit appears in the definition of the ECH-index, where it shows up as asymptotic terms in the relative intersection number of pseudoholomorphic curves with cylindrical ends (see the definition of $Q_\tau$ in \cite{Hu09}). 
\end{rem}
\subsubsection{The Moduli Spaces}
For a Hamiltonian $H:\mathbb R \times S^1 \times S \to \mathbb R$ that is locally constant outside of a compact set, a smooth family of $\omega$-compatible almost complex structures $(J_{s,t})_{(s,t)\in \mathbb R\times S^1}$, and a positive integer $k$, we consider the space of Floer isotopies consisting of $k$ Floer cylinders $$\mathcal M(H,J,k) = \left \{(u_1,\dots, u_k)\,\middle |\,\begin{aligned} 
\partial_s u_i+ J_{s,t}(\partial_tu_i-X_{H_{s,t}}\circ u_i) = 0, \\ E(u_i)<\infty,\,\iota(u_i,u_j) = 0 \text{ for } i\neq j\end{aligned}\right \}.$$ Using positivity of intersections, the condition $\iota(u_i,u_j)=0$ is equivalent to $u_i(s,t) \neq u_j(s,t)$ for all $(s,t)\in \mathbb R\times S^1$, while the condition $E(u_i)<\infty$ is a standard assumption and ensures that the Floer cylinders converge to periodic orbits for $s\to\pm \infty$. The space $\mathcal M(H,J,1)$ is simply the classical space of Floer cylinders, and $$\mathcal M(H,J,k)\subset \mathcal M(H,J,1) \times\cdots \times \mathcal M(H,J,1),$$ where $\mathcal M(H,J,k)$ is given by the union of the connected components on which $\sum_{i\neq j}\iota(u_i,u_j)=0$.   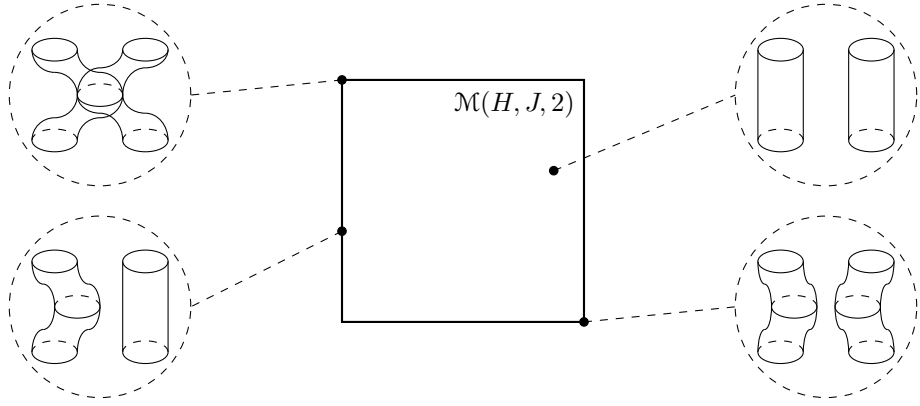
\begin{figure}[h!]
    \centering
   \begin{tikzpicture}[scale = 0.8]
       \draw[thick] (-2,-2)--(-2,2)--(2,2)--(2,-2)--cycle;
       \node[] at (0.85,1.6) {$\mathcal M(H,J,2)$};
       \filldraw[black] (-2,2) circle (2pt);
       \filldraw[black] (2,-2) circle (2pt);
       \filldraw[black] (1.5,0.5) circle (2pt);
       \filldraw[black] (-2,-0.5) circle (2pt);
       
       \begin{scope}[shift = {(-6,1.75)},scale = 0.75]
       \coordinate[] (P1) at (2,0);
            \draw[dashed] (0,0) ellipse (2 and 2);
            \draw[] (-1,1) ellipse (0.5 and 0.25);
            \draw[] (1,1) ellipse (0.5 and 0.25);
            \draw[] (-0.5,-1) arc (-180:0:-0.5 and 0.25);
            \draw[dashed] (-0.5,-1) arc (180:0:-0.5 and 0.25);
            \draw[] (1.5,-1) arc (-180:0:-0.5 and 0.25);
            \draw[dashed] (1.5,-1) arc (180:0:-0.5 and 0.25);
            \draw[] (0.5,0) arc (-180:0:-0.5 and 0.25);
            \draw[dashed] (0.5,0) arc (180:0:-0.5 and 0.25);
            
            %%%%%%%%%%%%%%%
            \draw[thin](-1.5,1) arc (0:-90:-0.5 and 0.5);
            \draw[thin](-1,0.5) arc (90:0:0.5 and 0.5);
            \draw[thin](-0.5,1) arc (0:-90:-0.5 and 0.5);
            \draw[thin](0,0.5) arc (90:0:0.5 and 0.5);
            %%%%%%%%%%%%%%%%%%%%
            \draw[thin](-1.5,-1) arc (0:-90:-0.5 and -0.5);
            \draw[thin](-1,-0.5) arc (90:0:0.5 and -0.5);
            \draw[thin](-0.5,-1) arc (0:-90:-0.5 and -0.5);
            \draw[thin](0,-0.5) arc (90:0:0.5 and -0.5);
            %%%%%%%%%%%%%%%%%
            \draw[thin](1.5,1) arc (0:-90:0.5 and 0.4);
            \draw[thin](1,0.6) arc (90:0:-0.5 and 0.6);
            \draw[thin](0.5,1) arc (0:-90:0.5 and 0.4);
            \draw[thin](0,0.6) arc (90:0:-0.5 and 0.6);
            %%%%%%%%%%%%%%%%%
            \draw[](1.5,-1) arc (0:-90:0.5 and -0.6);
            \draw[](1,-0.4) arc (90:0:-0.5 and -0.4);
            \draw[](0.5,-1) arc (0:-90:0.5 and -0.6);
            \draw[](0,-0.4) arc (90:0:-0.5 and -0.4);
       \end{scope}
       \begin{scope}[shift = {(-6,-1.75)},scale = 0.75]
             \coordinate[] (P2) at (2,0);
            \draw[dashed] (0,0) ellipse (2 and 2);
            \draw[] (-1,1) ellipse (0.5 and 0.25);
            \draw[] (1,1) ellipse (0.5 and 0.25);
            \draw[] (-0.5,-1) arc (-180:0:-0.5 and 0.25);
            \draw[dashed] (-0.5,-1) arc (180:0:-0.5 and 0.25);
            \draw[] (1.5,-1) arc (-180:0:-0.5 and 0.25);
            \draw[dashed] (1.5,-1) arc (180:0:-0.5 and 0.25);
            \draw[] (0.,0) arc (-180:0:-0.5 and 0.25);
            \draw[dashed] (0.,0) arc (180:0:-0.5 and 0.25);
            
            %%%%%%%%%%%%%%%
            \draw[thin](-1.5,1) arc (0:-90:-0.25 and 0.5);
            \draw[thin](-1.25,0.5) arc (90:0:0.25 and 0.5);
            \draw[thin](-0.5,1) arc (0:-90:-0.25 and 0.5);
            \draw[thin](-0.25,0.5) arc (90:0:0.25 and 0.5);
            %%%%%%%%%%%%%%%%%%%%
            \draw[thin](-1.5,-1) arc (0:-90:-0.25 and -0.5);
            \draw[thin](-1.25,-0.5) arc (90:0:0.25 and -0.5);
            \draw[thin](-0.5,-1) arc (0:-90:-0.25 and -0.5);
            \draw[thin](-0.25,-0.5) arc (90:0:0.25 and -0.5);
            %%%%%%%%%%%%%%%%%
            \draw[] (0.5,1)--(0.5,-1);
            \draw[] (1.5,1)--(1.5,-1);
       \end{scope}
          \begin{scope}[shift = {(6,-1.75)},scale = 0.75]
                \coordinate[] (P3) at (-2,0);
            \draw[dashed] (0,0) ellipse (2 and 2);
            \draw[] (-1,1) ellipse (0.5 and 0.25);
            \draw[] (1,1) ellipse (0.5 and 0.25);
            \draw[] (-0.5,-1) arc (-180:0:-0.5 and 0.25);
            \draw[dashed] (-0.5,-1) arc (180:0:-0.5 and 0.25);
            \draw[] (1.5,-1) arc (-180:0:-0.5 and 0.25);
            \draw[dashed] (1.5,-1) arc (180:0:-0.5 and 0.25);
            \draw[] (-0.2,0) arc (-180:0:-0.5 and 0.25);
            \draw[dashed] (-0.2,0) arc (180:0:-0.5 and 0.25);
            \draw[] (1.2,0) arc (-180:0:-0.5 and 0.25);
            \draw[dashed] (1.2,0) arc (180:0:-0.5 and 0.25);
            
            %%%%%%%%%%%%%%%
            \draw[thin](-1.5,1) arc (0:-90:-0.15 and 0.5);
            \draw[thin](-1.35,0.5) arc (90:0:0.15 and 0.5);
            \draw[thin](-0.5,1) arc (0:-90:-0.15 and 0.5);
            \draw[thin](-0.35,0.5) arc (90:0:0.15 and 0.5);
            %%%%%%%%%%%%%%%%%%%%
            \draw[thin](-1.5,-1) arc (0:-90:-0.15 and -0.5);
            \draw[thin](-1.35,-0.5) arc (90:0:0.15 and -0.5);
            \draw[thin](-0.5,-1) arc (0:-90:-0.15 and -0.5);
            \draw[thin](-0.35,-0.5) arc (90:0:0.15 and -0.5);
             %%%%%%%%%%%%%%%
            \draw[thin](1.5,1) arc (0:-90:0.15 and 0.5);
            \draw[thin](1.35,0.5) arc (90:0:-0.15 and 0.5);
            \draw[thin](0.5,1) arc (0:-90:0.15 and 0.5);
            \draw[thin](0.35,0.5) arc (90:0:-0.15 and 0.5);
            %%%%%%%%%%%%%%%%%%%%
            \draw[thin](1.5,-1) arc (0:-90:0.15 and -0.5);
            \draw[thin](1.35,-0.5) arc (90:0:-0.15 and -0.5);
            \draw[thin](0.5,-1) arc (0:-90:0.15 and -0.5);
            \draw[thin](0.35,-0.5) arc (90:0:-0.15 and -0.5);
       \end{scope}
         \begin{scope}[shift = {(6,1.75)},scale = 0.75]
               \coordinate[] (P4) at (-2,0);
            \draw[dashed] (0,0) ellipse (2 and 2);
            \draw[] (-1,1) ellipse (0.5 and 0.25);
            \draw[] (1,1) ellipse (0.5 and 0.25);
            \draw[] (-0.5,-1) arc (-180:0:-0.5 and 0.25);
            \draw[dashed] (-0.5,-1) arc (180:0:-0.5 and 0.25);
            \draw[] (1.5,-1) arc (-180:0:-0.5 and 0.25);
            \draw[dashed] (1.5,-1) arc (180:0:-0.5 and 0.25);
            \draw[] (-1.5,1)--(-1.5,-1);
            \draw[] (1.5,1)--(1.5,-1);
            \draw[] (0.5,1)--(0.5,-1);
            \draw[] (-0.5,1)--(-0.5,-1);
       \end{scope}
       \draw[dashed] (-2,2)-- (P1);
        \draw[dashed] (-2,-0.5)-- (P2);
        \draw[dashed] (1.5,0.5)--(P4);
        \draw[dashed] (2,-2)--(P3);
   \end{tikzpicture}
    \caption{Cylinders at different points in the moduli space $\mathcal M(H,J,2)$.}
    \label{fig: moduli space square}
   
\end{figure}Here, a path in the product space gives rise to a homotopy of cylinders relative to limiting orbits, so the intersection number is locally constant. To define and analyze the differential and continuation maps, we restrict the moduli space to cylinders with the same index. Here, the most relevant part of the moduli space is a product of one-dimensional intervals forming a hypercube. Passing to a face of this hypercube means one cylinder breaks into a broken trajectory, with the extreme case at the corner points, where all cylinders are broken (see Figure~\ref{fig: moduli space square}).
\subsubsection{The Symmetry}
Understanding the composition of continuation maps and the differential is central to Floer theory. These compositions always count broken Floer isotopies that consist of tuples of Floer cylinders that all break at distinct orbits. Such broken Floer isotopies appear as corner points of the moduli space we defined above. So to understand the composition of maps, we need to count these corner points. A simple but deep observation in classical Floer theory is that the number of boundary components of a compact  $1$-dimensional manifold is even. In our case, this is replaced by the observation that the number of corner points of hypercubes is even. Unfortunately, the moduli space is not sensitive to configurations of breaking orbits at corner points. Still, we want to count only those where all breaking orbits are distinct and hence correspond to honest broken isotopies. 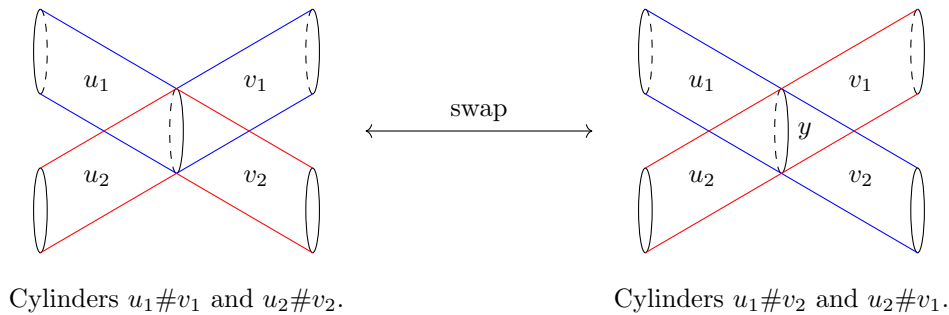
\begin{figure}[h!]
\begin{tikzpicture}
    \begin{scope}[shift = {(-4,0)},xscale = 0.3, yscale = 0.7]
    \draw[thin] (6,-1.5) ellipse (0.3 and 0.8);
    \path[thin] (6,-0.7) arc (-90:90:0.3 and -0.8) ;
    \draw[thin] (-6,-1.5) ellipse (0.3 and 0.8);
    \path[thin] (-6,-0.7) arc (90:-90:-0.3 and 0.8) ;
    \draw[thin,dashed] (-6,0.7) arc (-90:90:0.3 and 0.8);
    \draw[thin] (-6,2.3) arc (90:-90:-0.3 and 0.8);
    \draw[thin] (6,0.7) arc (-90:90:0.3 and 0.8) ;
    \draw[thin,dashed] (6,2.3) arc (90:-90:-0.3 and 0.8);
    \draw[thin] (0,-0.8) arc (-90:90:0.3 and 0.8);
    \draw[thin, dashed] (0,0.8) arc (90:-90:-0.3 and 0.8);
    \draw[thin,blue] (-6,2.3) -- (0,0.8);
    \draw[thin,blue] (-6,0.7) -- (0,-0.8);
    \draw[thin,blue] (6,2.3) -- (0,0.8);
    \draw[thin,blue] (6,0.7) -- (0,-0.8);
    \draw[thin,red] (0,0.8) -- (6,-0.7);
    \draw[thin,red] (0,-0.8) -- (6,-2.3);
    \draw[thin,red] (-6,-2.3) -- (0,-0.8);
    \draw[thin,red] (-6,-0.7) -- (0,0.8);
    \node[] at (-3.5,0.9) {$u_1$};
    \node[] at (-3.5,-0.9) {$u_2$};
    \node[] at (3.5,0.9) {$v_1$};
    \node[] at (3.5,-0.9) {$v_2$};
    \node[] at (0,-3.2) {Cylinders $u_1\#v_1$ and $u_2\#v_2$.};
    \end{scope}
    \draw[<->] (-1.5,0)--(1.5,0) node[midway, above]{\text{swap}};
    \begin{scope}[shift = {(4,0)},xscale = 0.3, yscale = 0.7]
    \draw[thin] (6,-1.5) ellipse (0.3 and 0.8);
    \path[thin] (6,-0.7) arc (-90:90:0.3 and -0.8) ;
    \draw[thin] (-6,-1.5) ellipse (0.3 and 0.8);
    \path[thin] (-6,-0.7) arc (90:-90:-0.3 and 0.8) ;
    \draw[thin,dashed] (-6,0.7) arc (-90:90:0.3 and 0.8);
    \draw[thin] (-6,2.3) arc (90:-90:-0.3 and 0.8);
    \draw[thin] (6,0.7) arc (-90:90:0.3 and 0.8);
    \draw[thin,dashed] (6,2.3) arc (90:-90:-0.3 and 0.8);
    \draw[thin] (0,-0.8) arc (-90:90:0.3 and 0.8) node[midway, right]{$y$};
    \draw[thin, dashed] (0,0.8) arc (90:-90:-0.3 and 0.8);
    \draw[thin,blue] (-6,2.3) -- (6,-0.7);
    \draw[thin,blue] (-6,0.7) -- (6,-2.3);
    \draw[thin,red] (-6,-2.3) -- (6,0.7);
    \draw[thin,red] (-6,-0.7) -- (6,2.3);
    \node[] at (-3.5,0.9) {$u_1$};
    \node[] at (-3.5,-0.9) {$u_2$};
    \node[] at (3.5,0.9) {$v_1$};
    \node[] at (3.5,-0.9) {$v_2$};
    \node[] at (0,-3.2) {Cylinders $u_1\#v_2$ and $u_2\#v_1$.};
    \end{scope}
\end{tikzpicture}\caption{The swapping operation.}\label{fig: swapping0}
\end{figure} We approach this problem by using the symmetry of broken trajectories with the same breaking orbits and show that they pair up and do not affect the count modulo $2$. Given broken trajectories $(u_1,v_1)$ and $(u_2,v_2)$ breaking at an orbit $y$, we use that $(u_1,v_2)$ and $(u_2,v_1)$ are two other broken trajectories and that this swapping operation does not change the intersection number of the broken trajectory. Hence, they remain in the moduli space. At its core, this argument relies on the asymptotic expression for the intersection number and on the symmetry of the linking number. This symmetry already played a role in Hutchings' proof of braid stability \cite{Hu25}.

\subsection*{Organization of the paper:} In Section~\ref{section.intersections}, we define an intersection number for cylinders and study its behavior under Floer-theoretic gluing. In Section~\ref{section: braided floer homology}, we define a differential on $\operatorname{BCF}_\beta(H,J)$ and define continuation maps between braided Floer complexes. We then proceed to prove Theorem~\ref{thm: braided floer homology}. In Section~\ref{section braid stability}, we prove Theorem~\ref{thm: braid stability} by defining \emph{bounded continuation maps}. In Section~\ref{section: application}, we prove Theorem~\ref{thm:entropy} by constructing eggbeater maps on $T^2$ with high-entropy braids of periodic contractible orbits and a large stability radius.
Finally, in the Appendix we carry out a Maslov index calculation which we use in Section~\ref{section: application}.
\subsection*{Acknowledgement:}I thank Peter Feller for many useful discussions and support throughout the project, both mathematically and emotionally. I thank Felix Schlenk, Michael Hutchings, Léo Mousseau, Beomjun Sohn, Dustin Connery--Grigg, Francesco Morabito, Baptiste Serraille, Paul Biran, and especially Johannes Hauber for their interest in the project, their comments, and patiently listening to my attempts to explain it. Felix Schlenk and Johannes Hauber also generously read preprints of this paper, pointed out various mistakes, and made valuable remarks, for which I am very thankful.

\subsection*{AI Acknowledgement: } The author used Claude to identify typos and verify calculations in Section~\ref{subsection: eggbeater}.
\newpage
\section{Some Preparation: Intersections}
 \label{section.intersections}
In this section, we define an intersection number of cylinders $u,v:\mathbb R\times S^1\to S$ that algebraically counts points $(s,t)\in \mathbb R\times S^1$ such that $u(s,t) = v(s,t)$. To do this, we pass to graphs of cylinders in a 4-dimensional almost complex manifold. We then make use of the well-known intersection theory of pseudoholomorphic cylinders in such manifolds to conclude that solutions to Floer equations intersect only positively. A priori, the intersection number we use is purely topological and is only defined when the Floer cylinders converge to distinct orbits in the limit. We do, however, discuss the behavior when the limiting orbits agree, specifically when they arise as breaking orbits of broken trajectories (see Figure~\ref{fig: cylinders in degenerate breaking position}). In this case, an additional asymptotic term appears, equal to the linking number of the knots given by slices of the Floer cylinders near the breaking orbit (see Lemma~\ref{intersection number and knots} and Figure~\ref{fig: local behavior at breaking orbit}).

First, we discuss the intersection number from a purely topological standpoint. We study cylinders $u:[-1,1]\times S^1\to S$, where $S$ is a closed surface. For two such cylinders $u_1,u_2$, we would like to define an intersection number that counts, algebraically, the number of points $(s,t)$ such that $u_1(s,t)= u_2(s,t)$. To do so, we consider the graphs of $u$ parametrized by $$\hat u:[-1,1]\times S^1\to [-1,1]\times S^1\times S,\quad \hat u(s,t) = (s,t,u(s,t))$$and consider the intersection of $\hat u_1,\hat u_2$ in $M = [-1,1]\times S^1\times S$.  For simplicity of notation, we denote by $\hat u$ both the graph of $u$ and its parameterization. A pair of cylinders $u_1,u_2$ is called \emph{admissible} if for every $t\in S^1$ we have $$u_1(-1,t) \neq u_2(-1,t)\quad \text{and}\quad u_1(1,t)\neq u_2(1,t).$$ For admissible $u_1$ and $u_2$ such that $\hat u_1$ and $\hat u_2$ can be put in transverse position via a homotopy that is constant at the boundary, the intersection number is defined by  $$\iota(u_1,u_2) = \sum_{p\in \hat u_1\cap \hat u_2} \epsilon_p,$$ where the orientation of the intersection determines $\epsilon_p\in \{\pm 1\}$. This intersection number is symmetric and invariant under homotopies constant on the boundary. We remark that there is a purely algebraic definition of this intersection number using Poincaré--Lefschetz duality $$D:H_2(M,A) \to H^2(M,B),$$ where $A \cup B = \partial M$ is a partition of the boundary of $M$ into two manifolds which overlap only at their common boundary $A\cap B = \partial A= \partial B$.  For admissible cylinders $u_1, u_2$, the intersection number is given as follows. Choose a partition $A\cup B$ of $\partial( [-1,1]\times S^1\times S) = \{\pm 1\}\times S^1\times S$ such that $\hat u_1(\pm 1,t) \in A$ and $\hat u_2(\pm 1,t)\in B$ for all $t\in S^1$. For admissible cylinders $u_1$ and $u_2$ we have $$\iota(u_1,u_2) = D([\hat u_1])([\hat u_2]).$$ 

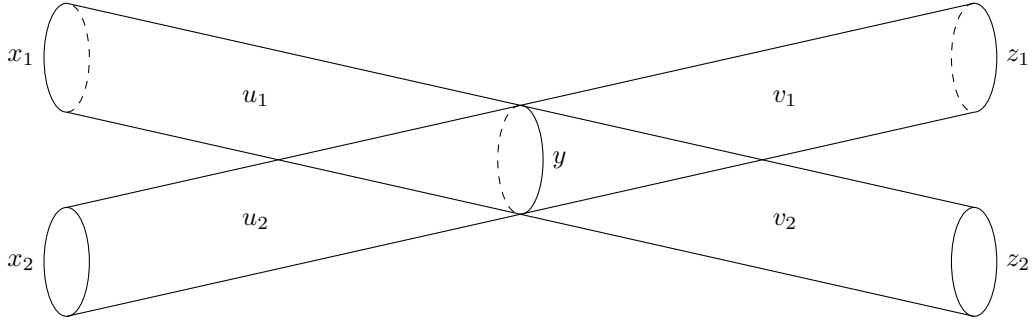
\begin{figure}[h!]
\begin{tikzpicture}[yscale = 0.8]
    \draw[thin] (6,-1.5) ellipse (0.3 and 0.8);
    \path[thin] (6,-0.7) arc (-90:90:0.3 and -0.8)  node[midway, right]{$z_2$};
    \draw[thin] (-6,-1.5) ellipse (0.3 and 0.8);
    \path[thin] (-6,-0.7) arc (90:-90:-0.3 and 0.8)  node[midway, left]{$x_2$};
    \draw[thin,dashed] (-6,0.7) arc (-90:90:0.3 and 0.8);
    \draw[thin] (-6,2.3) arc (90:-90:-0.3 and 0.8) node[midway, left]{$x_1$};
    \draw[thin] (6,0.7) arc (-90:90:0.3 and 0.8) node[midway, right] {$z_1$};
    \draw[thin,dashed] (6,2.3) arc (90:-90:-0.3 and 0.8);
    \draw[thin] (0,-0.8) arc (-90:90:0.3 and 0.8) node[midway, right]{$y$};
    \draw[thin, dashed] (0,0.8) arc (90:-90:-0.3 and 0.8);
    \draw[thin] (-6,2.3) -- (6,-0.7);
    \draw[thin] (-6,0.7) -- (6,-2.3);
    \draw[thin] (-6,-2.3) -- (6,0.7);
    \draw[thin] (-6,-0.7) -- (6,2.3);
    \node[] at (-3.5,0.9) {$u_1$};
    \node[] at (-3.5,-0.9) {$u_2$};
    \node[] at (3.5,0.9) {$v_1$};
    \node[] at (3.5,-0.9) {$v_2$};
\end{tikzpicture}
\caption{Cylinders in degenerate breaking position.}
\label{fig: cylinders in degenerate breaking position}
\end{figure}
\subsection{Intersections of Cylinders in Degenerate Breaking Position}  \label{Subsection: Intersections of Cylinders in Degenerate Breaking Position}
We now discuss the intersection number of cylinders that arise as follows. Let $u_1,u_2:[-1,0]\times S^1\to S$ and $v_1,v_2:[0,1]\times S^1\to S$  such that for all $t\in S^1$ $$y(t) = u_1(0,t) = u_2(0,t) = v_1(0,t)= v_2(0,t)$$ and consider $$u_i\#v_i: [-1,1]\times S^1\to S,\quad (s,t) \mapsto \begin{cases}
    u_i(s,t),& s\in [-1,0]\\ v_i(s,t) ,& s\in [0,1]
\end{cases}.$$ Further assume that $u_1\#v_1$ and $u_2\#v_2$ are admissible and that away from $\hat y$, $\widehat{u_1\#v_1}$ and $\widehat{u_2\#v_2}$ are transverse and intersect only finitely many times. Cylinders that arise in this form are \emph{in degenerate breaking position} with  \emph{breaking orbit} $y$. That is, two admissible cylinders $w_1, w_2$ are in degenerate breaking position if their graphs $\hat w_1,\hat w_2$ intersect only finitely many times away from $0\times S^1$, where they agree. Now suppose that $w_1,w_2$ are cylinders in degenerate breaking position with breaking orbit $y$. We next discuss the local behavior of the cylinders near the breaking orbit. Consider \begin{align*}\hat y: S^1\to  [-1,1]\times S^1\times S,\quad  t\mapsto (0,t,y(t))\end{align*} and choose a small tubular neighborhood $U$ of $\hat y$ that is homeomorphic to $(-\varepsilon,\varepsilon)\times S^1\times D^2$ where $D^2$ is the two-dimensional disk. Since the $\hat w_i$ are the graphs of $w_i$ we have that the slice of the cylinder $\hat w_i$ at $s = \delta$ for  $\delta\in (-\varepsilon,\varepsilon)$ lies in $\{\delta\}\times S^1\times D^2$, that is $$\hat w_{i}\big |_{\{\delta\}\times S^1}\subset \{\delta\}\times S^1\times D^2.$$ The neighborhood $U$ can be embedded in $ (-\varepsilon,\varepsilon)\times S^3$ by embedding the filled torus $S^1\times D^2$ as the standard filled torus in $S^3$ $$\iota: (-\varepsilon,\varepsilon) \times S^1\times D^2 \hookrightarrow (-\varepsilon,\varepsilon)\times S^3.$$ Now, let us assume that $\varepsilon>0$ is small enough so that $w_1(s,t)\neq w_2(s,t)$ for every $(s,t) \in (-\varepsilon,\varepsilon)\times S^1$ where $s \neq 0$. This is possible because the pair $(w_1,w_2)$ is in degenerate breaking position. In this case, the cylinders $\hat w_i|_{(-\varepsilon,\varepsilon)\times S^1}$ define isotopies of knots in $S^3$ by $$K^\delta_i: S^1 \to \{\delta\}\times S^3,\quad K_i^\delta =  \iota \circ \hat w_i\big|_{\{\delta\}\times S^1}.$$  For every $\delta\in (-\varepsilon,\varepsilon)\setminus\{0\}$ the knots $K_1^\delta$ and $K_2^\delta$ are disjoint and the linking number $\ell(K_1^\delta,K_2^\delta)$ is well-defined. Since in $((-\varepsilon,\varepsilon)\setminus \{0\})\times S^1$ the cylinders $\hat w_1$ and $\hat w_2$ define an isotopy of $K_1^\delta\dot \cup K_2^\delta$ as links, the linking number $\ell(K_1^\delta,K_2^\delta)$ is constant on each interval $(-\varepsilon,0)$ and $(0,\varepsilon)$. 
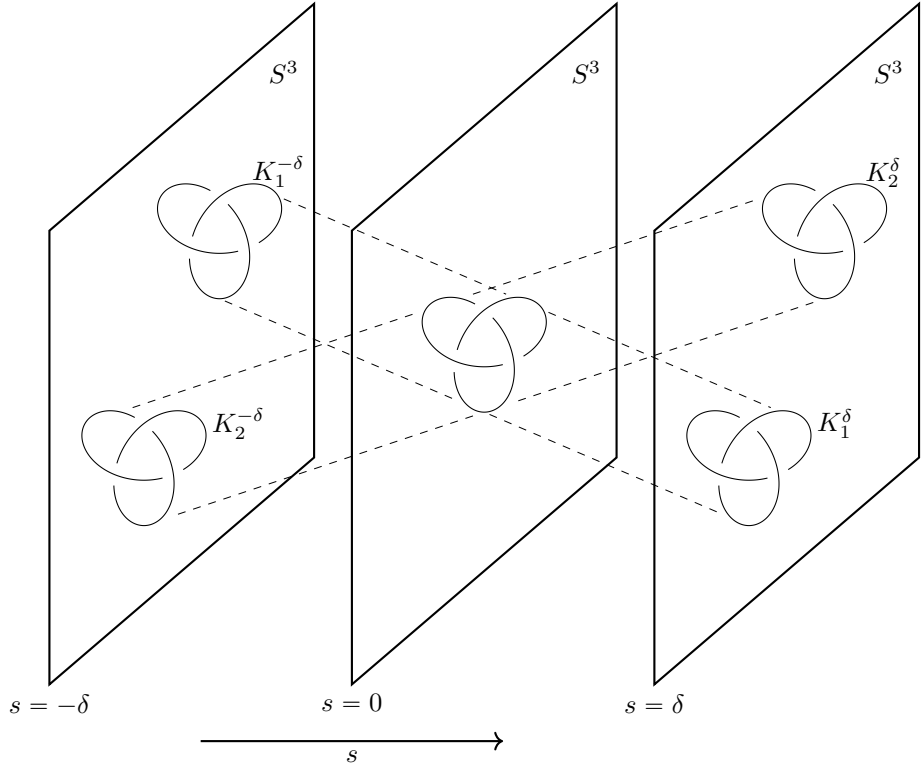
\begin{figure}[!h]
    \begin{tikzpicture}
    %planes
    \draw[thick] (0,-3) node[below] {$ s = 0$}--(0,3)--(3.5,6)--(3.5,0)--cycle;
    
     \draw[thick] (-4,-3) node[below] {$ s = -\delta$}--(-4,3)--(-0.5,6)--(-0.5,0)--cycle;
      \draw[thick] (4,-3)  node[below] {$ s = \delta$}--(4,3)--(7.5,6)--(7.5,0)--cycle;
      \draw[thick,->](-2,-3.75)--(2,-3.75) node[midway,below]{$s$};
      \node[] at (3.1,5.1) {$S^3$};
      \node[] at (-0.9,5.1) {$S^3$};
      \node[] at (7.1,5.1) {$S^3$};
      %trefoils
    \begin{scope}[shift = {(1.75,1.5)},scale = 0.3]
\draw[samples=150,smooth,domain=21.522:129.522,variable=\t]
  plot ({sin(\t)+2*sin(2*\t)}, {cos(\t)-2*cos(2*\t)});
\draw[samples=150,smooth,domain=141.522:249.522,variable=\t]
  plot ({sin(\t)+2*sin(2*\t)}, {cos(\t)-2*cos(2*\t)});
\draw[samples=150,smooth,domain=261.522:369.522,variable=\t]
  plot ({sin(\t)+2*sin(2*\t)}, {cos(\t)-2*cos(2*\t)});
  \coordinate (K31) at (30:3);
  \coordinate (K32) at (150:3) ;
  \coordinate (K33) at (270:3);
  \end{scope}
  %%%%%%%%%%%%%%%%%%%%%%%%%%%%%%
      \begin{scope}[shift = {(-1.75,3)},scale = 0.3]
\draw[samples=150,smooth,domain=21.522:129.522,variable=\t]
  plot ({sin(\t)+2*sin(2*\t)}, {cos(\t)-2*cos(2*\t)});
\draw[samples=150,smooth,domain=141.522:249.522,variable=\t]
  plot ({sin(\t)+2*sin(2*\t)}, {cos(\t)-2*cos(2*\t)});
\draw[samples=150,smooth,domain=261.522:369.522,variable=\t]
  plot ({sin(\t)+2*sin(2*\t)}, {cos(\t)-2*cos(2*\t)});
  \coordinate (K11) at (30:3);
  \coordinate (K12) at (150:3) ;
  \coordinate (K13) at (270:3);

  \end{scope}
  %%%%%%%%%%%%%%%%%%%%%%%%%%%%%%
      \begin{scope}[shift = {(-2.75,0)},scale = 0.3]
\draw[samples=150,smooth,domain=21.522:129.522,variable=\t]
  plot ({sin(\t)+2*sin(2*\t)}, {cos(\t)-2*cos(2*\t)});
\draw[samples=150,smooth,domain=141.522:249.522,variable=\t]
  plot ({sin(\t)+2*sin(2*\t)}, {cos(\t)-2*cos(2*\t)});
\draw[samples=150,smooth,domain=261.522:369.522,variable=\t]
  plot ({sin(\t)+2*sin(2*\t)}, {cos(\t)-2*cos(2*\t)});
  \coordinate (K21) at (30:3);
  \coordinate (K22) at (150:3) ;
  \coordinate (K23) at (270:3);
  \end{scope}
   %%%%%%%%%%%%%%%%%%%%%%%%%%%%%%
      \begin{scope}[shift = {(6.25,3)},scale = 0.3]
\draw[samples=150,smooth,domain=21.522:129.522,variable=\t]
  plot ({sin(\t)+2*sin(2*\t)}, {cos(\t)-2*cos(2*\t)});
\draw[samples=150,smooth,domain=141.522:249.522,variable=\t]
  plot ({sin(\t)+2*sin(2*\t)}, {cos(\t)-2*cos(2*\t)});
\draw[samples=150,smooth,domain=261.522:369.522,variable=\t]
  plot ({sin(\t)+2*sin(2*\t)}, {cos(\t)-2*cos(2*\t)});
  \coordinate (K41) at (30:3);
  \coordinate (K42) at (150:3) ;
  \coordinate (K43) at (270:3);
  \end{scope}
    %%%%%%%%%%%%%%%%%%%%%%%%%%%%%%
      \begin{scope}[shift = {(5.25,0)},scale = 0.3]
\draw[samples=150,smooth,domain=21.522:129.522,variable=\t]
  plot ({sin(\t)+2*sin(2*\t)}, {cos(\t)-2*cos(2*\t)});
\draw[samples=150,smooth,domain=141.522:249.522,variable=\t]
  plot ({sin(\t)+2*sin(2*\t)}, {cos(\t)-2*cos(2*\t)});
\draw[samples=150,smooth,domain=261.522:369.522,variable=\t]
  plot ({sin(\t)+2*sin(2*\t)}, {cos(\t)-2*cos(2*\t)});
  \coordinate (K51) at (30:3);
  \coordinate (K52) at (150:3) ;
  \coordinate (K53) at (270:3);
  \end{scope}
%dotted lines
\draw[very thin, dashed] ($(K11)!0.02!(K31)$)--($(K11)!0.86!(K31)$);
\draw[very thin, dashed] ($(K31)!0.02!(K51)$)--($(K31)!0.86!(K51)$);
\draw[very thin, dashed] ($(K13)!0.02!(K33)$)--($(K13)!0.9!(K33)$);
\draw[very thin, dashed] ($(K33)!0.02!(K53)$)--($(K33)!0.9!(K53)$);
%%%%%%%%%%%%%%%%%%%%%%%%%
\draw[very thin, dashed] ($(K22)!0.14!(K32)$)--($(K22)!0.98!(K32)$);
\draw[very thin, dashed] ($(K32)!0.14!(K42)$)--($(K32)!0.98!(K42)$);
\draw[very thin, dashed] ($(K23)!0.1!(K33)$)--($(K23)!0.98!(K33)$);
\draw[very thin, dashed] ($(K33)!0.1!(K43)$)--($(K33)!0.98!(K43)$);
%knot names
\node[above] at (K11) {$K_1^{-\delta}$};
\node[right] at (K21) {$K_2^{-\delta}$};
\node[above] at (K41) {$K_2^{\delta}$};
\node[right] at (K51) {$K_1^{\delta}$};
    \end{tikzpicture}
    \caption{Local behavior at breaking orbit.}
    \label{fig: local behavior at breaking orbit}
\end{figure}
Finally, we find a well-defined quantity that describes the behavior of the cylinders at the breaking orbit; namely, we define \begin{align}\label{eq: asymptoric term}
\lambda(w_1,w_2) = \ell(K_1^\delta,K_2^\delta)-\ell(K_1^{-\delta},K_2^{-\delta}),\end{align} where $\delta >0$ is small enough. \begin{rem}\label{rem: ECH-index}
    In the setting of ECH, Hutchings defines an intersection number $Q_\tau$ (see \cite[Definition 2.12]{Hu09}) of pseudoholomorphic curves with cylindrical ends limiting in Reeb orbits $\gamma$ of a contact boundary manifold $(M,\xi)$ in a similar way as we do here. The intersection number equals the interior intersections plus the linking number of the limiting braids. Hutchings definition depends on the choice $\tau$ of trivializations of $\gamma^*\xi$, which gives rise to an identification of the tubular neighborhood of $\gamma$ with the standard filled torus in $S^3$. We consider a simpler case where we only consider the difference of linking numbers, which does not depend on the embedding and hence the trivialization $\tau$. In our case, the symplectic trivialization of $\gamma^*\xi$ translates to a symplectic trivialization $\tau$ of $y^*TS$ and the following interesting estimate holds \cite[Lemma 6.9]{HU02}$$\ell_\tau(K_1^{-\delta},K_2^{-\delta}) \leq \left \lfloor{ \mu_\tau(y)\over 2}\right \rfloor\quad\text{and}\quad \ell_\tau(K_1^{\delta},K_2^{\delta}) \geq \left \lceil{\mu_\tau(y)\over 2}\right \rceil,$$ where $\mu_\tau(y)$ is the Conley--Zehnder index of the orbit $y$ associated with the symplectic trivialization $\tau$ of $y^*TS$. (Note that $K_i^{-\delta}$ corresponds to positive ends of cylinders and $K_i^{\delta}$ to negative ends.)
\end{rem}
\begin{rem} In a first version, the author wondered how the asymptotic term $\lambda$ relates to the asymptotic term $\Omega^\Phi(u,v)$ in the generalized intersection number $[u]\ast [v]$ of pseudoholomorphic curves defined by Siefring \cite[Section 4.1.2]{Si11}. Michael Hutchings kindly pointed out the following, which resolves the question. The term $\Omega^\Phi(u,v)$ is given by the upper bound on the linking number as in Remark \ref{rem: ECH-index} (see also \cite[Lemma 6.9]{HU02}). 
The generalized intersection number is given by $\iota(\tilde u,\tilde v) + \Omega^\Phi(u,v)$, where $\iota(\tilde u,\tilde v)$ is the algebraic intersection number of a certain perturbation of $u$ and $v$ and  
$\Omega^\Phi(u,v)$ is defined so that $[u]\ast[v]$ bounds the number of intersections of $u$ and $v$.
\end{rem}
\begin{lemma} \label{intersection number and knots}
    Let $w_1,w_2$ be a pair of cylinders in degenerate breaking position. Then the intersection number of $w_1,w_2$ is given by $$\iota(w_1,w_2 ) = \lambda(w_1,w_2)+\sum_{p\in \hat w_1\cap \hat w_2\setminus\hat y}\epsilon_{p}.$$
\end{lemma}
\begin{proof}
    We will homotope the cylinders in a neighborhood of the breaking orbit into transverse position and then see that in this neighborhood the number $\lambda(w_1,w_2)$ counts exactly the intersections with sign that have appeared. We use the following formulation of the linking number of two knots in $S^3$. Consider $S^3$ as the boundary of the $4$-ball $B^4$ and let $S_1,S_2$ be embedded surfaces in $B^4$ such that $\partial S_i = K_i$. Then the algebraic intersection number of the homotopy classes of $S_1,S_2$ relative to the boundary equals the linking number $\ell(K_1,K_2)$. Let $\{- \delta\}\times S^3$ bound $B^4$ for $\delta\in (0,\varepsilon)$ and choose surfaces $S_1,S_2$ that bound $K_1^{-\delta}$ and $K_2^{-\delta}$. Homotope $\hat w_1$ and $\hat w_2$ in a small tubular neighborhood $U$ of $\hat y$ such that $U\cong (-{\delta\over 2},{\delta\over 2})\times S^1\times D^2$ such that the resulting cylinders intersect transversely, and such that the homotopy is constant on $\partial U$. Glue $ [-\delta,\delta]\times S^3$ to $B^4$ by identifying $\partial  B^4$ with $ \{- \delta\}\times S^3$ and note that the cylinders $\hat w_1|_{[-\delta,\delta]\times S^1}$ and $\hat w_2|_{[-\delta,\delta]\times S^1}$ extend the surfaces $S_1$ and $S_2$ to $\Sigma_1$ and $\Sigma_2$. By construction, $\Sigma_1$ and $\Sigma_2$ are surfaces in a $4$-ball that bound the knots $K_1^\delta$ and $K_2^\delta$ on the boundary. The difference of the linking number $$\ell(K_1^\delta,K_2^\delta)-\ell(K_1^{-\delta},K_2^{-\delta})
    $$ is now given by $$\sum_{z\in \Sigma_1\cap \Sigma_2}\epsilon_z-\sum_{z\in S_1\cap S_2}\epsilon_z  = \sum_{z\in (\Sigma_1\setminus S_1)\cap (\Sigma_2\setminus S_2)}\epsilon_z,$$ where $\epsilon_z\in \{\pm 1\}$ depends on the orientation of the intersections.  By construction, this exactly counts the intersections with sign that appear after a local homotopy of $\hat w_1$ and $\hat w_2$ near the breaking orbit.
\end{proof}
The number $\lambda(w_1,w_2)$ is the local intersection number of $\hat w_1$ and $\hat w_2$ near the breaking orbit $\hat y$. That is, after a homotopy in a small neighborhood of $\hat y$ that makes the cylinders transverse, $\lambda(w_1,w_2)$ equals the number of intersections with signs in this neighborhood.\\

We now apply this to the breaking of Floer cylinders in Floer theory.
Suppose that we have a homotopy of Hamiltonians $H:\mathbb R\times S^1\times S\to \mathbb R$ such that $H(s,t,x)$ is independent of $s$ for $|s|$ large enough. We want to consider pseudoholomorphic cylinders $u:S^1\times \mathbb R\to S$ that satisfy the equation $$\partial_su + J_{s,t}(\partial _tu - X_{H_{s,t}}\circ u) = 0$$ and such that as $s\to \pm\infty$ we have $u(s,t)\to x^\pm(t)$, which consequently are periodic orbits of the Hamiltonian flow induced by $H_{\pm\infty}(t,x) = \lim _{s\to \pm \infty}H(s,t,x)$. Fix an orientation-preserving diffeomorphism $\varphi: (-1,1)\to \mathbb R$. The cylinders $u\circ\varphi$ extend to a cylinder $\tilde u:[-1,1]\times S^1 \to S$.  In this case, two cylinders $\tilde u_1$ and $\tilde u_2$ are admissible if and only if the positive limiting orbits of $u_1$ and $u_2$ are pointwise distinct, and likewise the negative limits. We analogously define $\hat u$ to be the parameterization of the graph of $\tilde u$ $$\hat u:[-1,1]\times S^1\to [-1,1]\times S^1\times S, \quad (s,t)\mapsto (s,t,\tilde u(s,t))$$ and we say that $u_1,u_2$ are admissible if $\tilde u_1,\tilde u_2$ are so. Moreover, for admissible $u_1,u_2$ we define $$\iota(u_1,u_2) = \iota(\tilde u_1,\tilde u_2).$$ \begin{prop}\label{prop.positivityofintersection}
    Let $u_1,u_2$ be admissible cylinders satisfying the Floer equation $$\partial_su_i+J_{s,t}(\partial_tu_i-X_{H_{s,t}}\circ u_i) = 0.$$ Then $\iota(u_1,u_2) \geq 0$ and if $\iota(u_1,u_2) = 0$ the maps $\hat u_1$ and $\hat u_2$ do not intersect.
\end{prop}
\begin{proof}We use the Gromov trick and give an almost complex structure on $\mathbb R_s \times S^1_t\times S$ such that the graphs $\hat u_i$ are unperturbed holomorphic cylinders. Namely, define $$\tilde J(\partial_s) = \partial_t+X_{H_{s,t}},\quad \tilde J(\partial_t) = -\partial_s-J_{s,t}X_{H_{s,t}},\quad \tilde J|_{TS} = J_{s,t},$$ where we identify $X_{H_{s,t}}(z)$ with its image under the differential of the inclusion $$\iota_{s,t}:S\hookrightarrow\mathbb R\times S^1\times S,\quad z\mapsto (s,t,z).$$
We then have $$\partial_s\hat u_i +\tilde J\partial_t \hat u_i = 0.$$ The result now follows from standard theory of intersection of pseudoholomorphic curves in almost complex 4-manifolds (see \cite{MDS04}).
\end{proof}
\subsection{Gluing and Intersection Numbers}\label{section.gluing}
We now study how the intersection number behaves under gluing of pseudoholomorphic cylinders. We briefly recall the standard gluing construction for Floer cylinders (see \cite[Chapter 9]{AD14}). More concretely, fix a Hamiltonian $H:\mathbb R\times S^1 \times S\to \mathbb R$ and a smooth family of almost complex structures $\{J_{s,t}\}_{s\in \mathbb R,t\in S^1}$ and let $\mathcal M(x,z)$ denote the space of solutions $u:\mathbb R\times S^1\to S$ to $$\partial_su+ J_{s,t}(\partial_tu-X_{H_{s,t}}\circ u) = 0$$ such that $\lim_{s\to\infty}u(s,t) = z(t)$ and $\lim_{s\to -\infty}u(s,t) = x(t)$. For a generic choice of $H_{s,t}$ and $J_{s,t}$, the space $\mathcal M(x,z)$ is a finite-dimensional manifold carrying the $C^\infty_{\text{loc}}$ topology. In our context, a gluing map is a continuous map $$\psi:(\rho_0,\infty)\to \mathcal M(x,z)$$ such that there exist $u\in \mathcal M(x,y)$ and $v\in \mathcal M(y,z)$ and smooth functions $f^\pm:\mathbb R \to \mathbb R$ with $\lim_{\rho\to \infty}\psi_\rho(s+f^-(\rho),t) = u(s,t)$ and $\lim_{\rho\to \infty}\psi_\rho(s+f^+(\rho),t) = v(s,t)$ where we write $\psi_\rho$ for $\psi(\rho)$. We then say that the curves $\psi_\rho$ converge to the broken trajectory $(u,v)$ with breaking orbit $y$. For us, the map $\psi$, if it exists, is always constructed as follows. Given $u\in \mathcal M(x,y)$ and $v\in \mathcal M(y,z)$, we define a \emph{pre-gluing map} $w_\rho: \mathbb R \times S^1\to S$, where the \emph{gluing parameter} $\rho>0$ is large enough. This pre-gluing map approximates the broken trajectory $(u,v)$ from $x$ to $z$ which breaks at the orbit $y$.

First, let $\beta^+:\mathbb R\to [0,1]$ be smooth and increasing such that $$\beta^+(s) =\begin{cases}
   0,& s\leq \varepsilon\\   1,& s\geq 1
\end{cases}$$ and let $\beta^-(s) = \beta^+(-s)$. Then we define $$w_\rho(s,t) = 
\begin{cases}
u(s+\rho,t)& s\leq-1\\
\exp_{y(t)}\left(\beta^-(s)\exp_{y(t)}^{-1}(u(s+\rho,t))+\beta^+(s)\exp_{y(t)}^{-1}(v(s-\rho,t))\right)&s\in [-1,1]\\
v(s-\rho,t)&s\geq1
\end{cases},$$ where $\rho>\rho_0$ and $\rho_0$ is chosen large enough that the exponential maps are well defined. 
\begin{figure}[h!]
    \begin{tikzpicture}
        \draw[] (-5.5,0) ellipse (0.2 and 0.5);
        \draw[] (5.5,0.5) arc (-90:90: 0.2 and -0.5);
        \draw[dashed] (5.5,0.5) arc (-90:90: -0.2 and -0.5);
        \draw[] (-5.5,0.5) arc (40:110: -3.75 and -1.5)coordinate (M1);
        \draw[] (-5.5,-0.5) arc (40:110: -3.75 and -1.5) ;
        \draw[] (5.5,0.5) arc (40:110: 3.75 and -1.5)coordinate (M2);
        \draw[] (5.5,-0.5) arc (40:110: 3.75 and -1.5) ;
        \draw[] ($(M2)-(0,0.5)$) ellipse (0.2 and 0.5);
        \draw[]  (M1) arc (-90:90: 0.2 and -0.5);
        \draw[dashed]  (M1) arc (-90:90: -0.2 and -0.5);
        \path[] (-5.5,0.5) arc (40:135: -3.75 and -1.5) coordinate (M) ;
        \draw[] (M) arc (-90:90: 0.2 and -0.5);
        \draw[dashed] (M) arc (-90:90: -0.2 and -0.5);
        \draw[]  ($(M)+(0.5,0)$) arc (-90:90: 0.2 and -0.5);
        \draw[dashed]  ($(M)+(0.5,0)$) arc (-90:90: -0.2 and -0.5);
        \draw[]  ($(M)+(-0.5,-0.5)$)  ellipse (0.2 and 0.5);
        \draw[] ($(M)-(0.5,0)$) --  ($(M)+(0.5,0)$);
        \draw[] ($(M)-(0.5,1)$) --  ($(M)+(0.5,-1)$);
        \draw[dashed] (M1)--($(M)-(0.5,0)$);
        \draw[dashed] ($(M1)-(0,1)$)--($(M)-(0.5,1)$);
        \draw[dashed] (M2)--($(M)+(0.5,0)$);
        \draw[dashed] ($(M2)-(0,1)$)--($(M)-(-0.5,1)$);
        \draw[decorate,decoration={brace,amplitude=8pt}] ($(M2)-(0,1.3)$)-- ($(M1)-(0,1.3)$) ;
        \node[] at (0,-1.8) {$\text{exp}_{y(t)}(X(s,t))$};
        \node[above] at (M) {$y$};
        \node[above] at (-5.5,0.5) {$x$};
        \node[above] at (5.5,0.5) {$z$};
        \node[above left] at (M1) {$u(\rho-1,\cdot)$};
        \node[above right] at (M2) {$v(1-\rho,\cdot)$};
        %\draw[->] (0,-2.5)--(0,-3.5) node[midway, right] {$\text{as }\rho \to \infty$};
        %\draw[] (-5.5,-4.25) node[above]{$x$} arc (40:134: -3.75 and -1.5) coordinate (M') node[above] {$y$};
        %\draw[] (-5.5,-5.25) arc (40:134: -3.75 and -1.5);
        %\draw[] (5.5,-4.25) node[above] {$z$} arc (40:134: 3.77 and -1.5);
        %\draw[] (5.5,-5.25) arc (40:134.5: 3.77 and -1.5);
        %\draw[] (-5.5,-4.75) ellipse (0.2 and 0.5);
        %\draw[]  (M') arc (-90:90: 0.2 and -0.5);
        %\draw[dashed]  (M') arc (-90:90: -0.2 and -0.5);
        %\draw[]  (5.5,-4.25) arc (-90:90: 0.2 and -0.5);
        %\draw[dashed]  (5.5,-4.25) arc (-90:90: -0.2 and -0.5);
    \end{tikzpicture}
    \caption{Pre-gluing of cylinders.}
\end{figure}
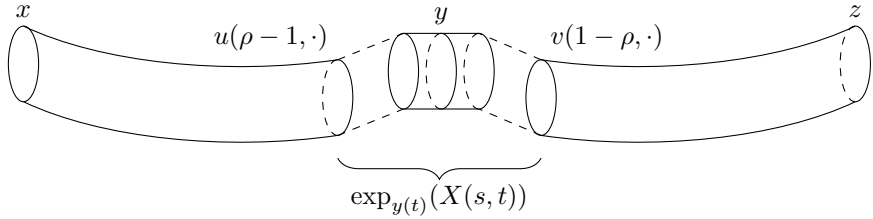 $ $\\
The map $\psi$ is now given by \begin{align}\label{eq: geodesic approximation}
\psi_\rho(s,t) = \exp_{w_\rho(s,t)}(\xi_\rho(s,t)),\end{align} where $\xi_\rho({s,t})$ is a well-chosen continuous vector field with $\xi_\rho(s,t) \to 0$ as $|s|\to \infty$ and such that $ \exp_{w_\rho(s,t)}(\xi_\rho(s,t))$ is indeed an element of $\mathcal M(x,z)$ (see, for example, \cite[Chapter~9]{AD14}). We now analyze the behavior of intersection numbers under standard Floer gluing.
\begin{lemma}\label{homotopy to breaking}
    Let $\varphi:(-\infty,0)\to \mathbb R$ be a diffeomorphism such that $\varphi\,|\,_{(-\infty,-\varepsilon]}(x) = x$. We identify a broken trajectory $(u,v)$ with the map $$u\#v:\mathbb R \times S^1\to S,\quad (s,t) \mapsto \begin{cases}
        u(\varphi(s),t),& s<0\\
        y(t),& s = 0\\
        v(-\varphi(-s),t)& s >0
    \end{cases},$$ where $\lim_{s\to \infty}u(s,t) = \lim_{s\to -\infty}v(s,t) = y(t)$. Then $u\#v$ and $\psi_\rho$ are homotopic relative to the limiting orbits for every $\rho$ large enough. In particular, $\widehat{u\#v}$ and $\hat \psi_\rho$ are homotopic for every $\rho$ large enough.
\end{lemma}\begin{proof}
   By (\ref{eq: geodesic approximation}), $\psi_\rho$ is homotopic to $w_\rho$ for $\rho$ large enough. For every $\delta\in (0,1]$ define $$\tilde w_\rho^\delta(s,t) = \begin{cases}
        u(\varphi(s)+\rho,t), & s\leq -\delta\\
        \exp_{y(t)}\left(\beta^-({s\over \delta})\exp_{y(t)}^{-1}(u(\varphi(s)+\rho,t))+\beta^+({s\over \delta})\exp_{y(t)}^{-1}(v(-\varphi(-s)-\rho,t))\right)&s\in [-\delta,\delta]\\
v(-\varphi(-s)-\rho,t)&s\geq\delta \end{cases} $$ which is defined for any $\delta\in (0,1]$ if $\rho$ is large enough. Since $\varphi\big |_{(-\infty,-\varepsilon]} = \text{id}_{(-\infty,-\varepsilon]}$ (where $\varepsilon$ is from the definition of $\beta^\pm$) we have $\tilde w^1_{\rho}=w_\rho$ for $\rho$ large enough. Now, for $\delta\to 0$ the map $\tilde w^\delta_{\rho}$ converges to $(u,v)$ in the $C^0$-topology.  Hence, for every $\rho$ we obtain a homotopy from $(u,v)$ to $\tilde w_{\rho}^{1}$. Since $w_\rho$ is homotopic to $\psi_\rho = \exp_{w_\rho}(\xi_\rho)$, we conclude that $\psi_\rho$ is homotopic to $u\#v$ for $\rho$ large enough.
\end{proof} A direct consequence is the following lemma.\begin{lemma}[Stability of Intersection]\label{stability of intersections}
     Let $u_i\in \mathcal M(x_i,y_i)$ and $v_i\in \mathcal M(y_i,z_i)$ for $i =1,2$ be pairs of admissible curves, that is $$x_1(t) \neq x_2(t),\quad y_1(t) \neq y_2(t),\quad z_1(t)\neq z_2(t)$$ for all $t\in S^1$. Let $\psi^{(i)}:(\rho_0,\infty)\to \mathcal M(x_i,z_i)$ for $i=1,2$ be gluing maps of $u_i,v_i$. Then $$\iota(u_1,u_2) + \iota(v_1,v_2) = \iota(\psi^{(1)}(\rho),\psi^{(2)}(\rho))$$ for every $\rho >0$ large enough.
\end{lemma}\begin{proof}
    By Lemma~\ref{homotopy to breaking} the cylinders $u_i\#v_i$ and $\psi^{(i)}(\rho)$ are homotopic relative to the boundary when $\rho$ is large enough and hence $$\iota(u_1\#v_1,u_2\#v_2) = \iota(\psi^{(1)}(\rho),\psi^{(2)}(\rho)).$$ Furthermore, $$\iota(u_1,u_2) + \iota(v_1,v_2) = \iota(u_1\#v_1,u_2\#v_2)$$ since orbits satisfy $y_1(t) \neq y_2(t)$ for all $t\in S^1$ so that the count of intersection points split.
\end{proof}
We will later encounter pairs of broken trajectories in a degenerate breaking position, as defined in Section~\ref{Subsection: Intersections of Cylinders in Degenerate Breaking Position}, that break at the same periodic orbit $y$; that is, we have two cylinders $u_1\#v_1$ and $u_2\#v_2$ that break at $y$. To use the symmetry of this configuration, we ``swap'' the latter two cylinders, giving another configuration in a degenerate breaking position (see Figure~\ref{fig: swapping}). The following lemma states that this operation does not change the cylinders' intersection number and is essential to the proofs of the main theorems. For the proof of the lemma we use the expression $$\iota(u_1\#v_1,u_2\#v_2) = \lambda(u_1\#v_1,u_2\#v_2) +\sum_{z\in \widehat{u_1\#v_1}\cap \widehat{u_2\#v_2}\setminus \hat y} \epsilon_z$$ as in Lemma~\ref{intersection number and knots}. We then use the description of $\lambda$ by linking numbers  $$\lambda(u_1\#v_1,u_2\#v_2) =\ell (K_1^{\delta},K_2^\delta)- \ell (K^{-\delta}_1,K^{-\delta}_2)$$ and use the fact that the linking number is symmetric.  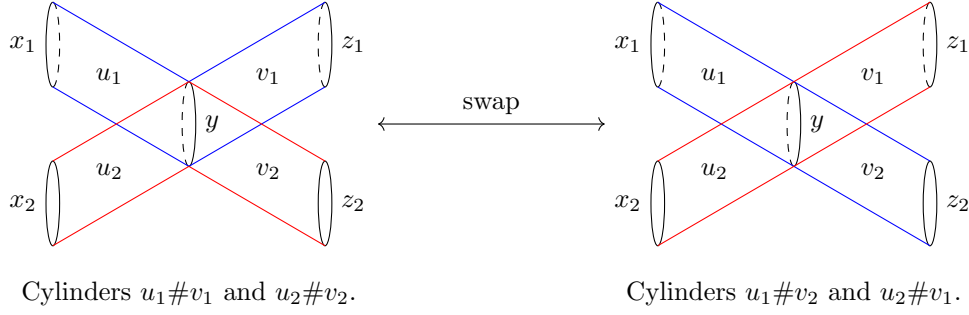
\begin{figure}[h!]
\begin{tikzpicture}
    \begin{scope}[shift = {(-4,0)},xscale = 0.3, yscale = 0.7]
    \draw[thin] (6,-1.5) ellipse (0.3 and 0.8);
    \path[thin] (6,-0.7) arc (-90:90:0.3 and -0.8)  node[midway, right]{$z_2$};
    \draw[thin] (-6,-1.5) ellipse (0.3 and 0.8);
    \path[thin] (-6,-0.7) arc (90:-90:-0.3 and 0.8)  node[midway, left]{$x_2$};
    \draw[thin,dashed] (-6,0.7) arc (-90:90:0.3 and 0.8);
    \draw[thin] (-6,2.3) arc (90:-90:-0.3 and 0.8) node[midway, left]{$x_1$};
    \draw[thin] (6,0.7) arc (-90:90:0.3 and 0.8) node[midway, right] {$z_1$};
    \draw[thin,dashed] (6,2.3) arc (90:-90:-0.3 and 0.8);
    \draw[thin] (0,-0.8) arc (-90:90:0.3 and 0.8) node[midway, right]{$y$};
    \draw[thin, dashed] (0,0.8) arc (90:-90:-0.3 and 0.8);
    \draw[thin,blue] (-6,2.3) -- (0,0.8);
    \draw[thin,blue] (-6,0.7) -- (0,-0.8);
    \draw[thin,blue] (6,2.3) -- (0,0.8);
    \draw[thin,blue] (6,0.7) -- (0,-0.8);
    \draw[thin,red] (0,0.8) -- (6,-0.7);
    \draw[thin,red] (0,-0.8) -- (6,-2.3);
    \draw[thin,red] (-6,-2.3) -- (0,-0.8);
    \draw[thin,red] (-6,-0.7) -- (0,0.8);
    \node[] at (-3.5,0.9) {$u_1$};
    \node[] at (-3.5,-0.9) {$u_2$};
    \node[] at (3.5,0.9) {$v_1$};
    \node[] at (3.5,-0.9) {$v_2$};
    \node[] at (0,-3.2) {Cylinders $u_1\#v_1$ and $u_2\#v_2$.};
    \end{scope}
    \draw[<->] (-1.5,0)--(1.5,0) node[midway, above]{\text{swap}};
    \begin{scope}[shift = {(4,0)},xscale = 0.3, yscale = 0.7]
    \draw[thin] (6,-1.5) ellipse (0.3 and 0.8);
    \path[thin] (6,-0.7) arc (-90:90:0.3 and -0.8)  node[midway, right]{$z_2$};
    \draw[thin] (-6,-1.5) ellipse (0.3 and 0.8);
    \path[thin] (-6,-0.7) arc (90:-90:-0.3 and 0.8)  node[midway, left]{$x_2$};
    \draw[thin,dashed] (-6,0.7) arc (-90:90:0.3 and 0.8);
    \draw[thin] (-6,2.3) arc (90:-90:-0.3 and 0.8) node[midway, left]{$x_1$};
    \draw[thin] (6,0.7) arc (-90:90:0.3 and 0.8) node[midway, right] {$z_1$};
    \draw[thin,dashed] (6,2.3) arc (90:-90:-0.3 and 0.8);
    \draw[thin] (0,-0.8) arc (-90:90:0.3 and 0.8) node[midway, right]{$y$};
    \draw[thin, dashed] (0,0.8) arc (90:-90:-0.3 and 0.8);
    \draw[thin,blue] (-6,2.3) -- (6,-0.7);
    \draw[thin,blue] (-6,0.7) -- (6,-2.3);
    \draw[thin,red] (-6,-2.3) -- (6,0.7);
    \draw[thin,red] (-6,-0.7) -- (6,2.3);
    \node[] at (-3.5,0.9) {$u_1$};
    \node[] at (-3.5,-0.9) {$u_2$};
    \node[] at (3.5,0.9) {$v_1$};
    \node[] at (3.5,-0.9) {$v_2$};
    \node[] at (0,-3.2) {Cylinders $u_1\#v_2$ and $u_2\#v_1$.};
    \end{scope}
\end{tikzpicture}\caption{The swapping operation.}\label{fig: swapping}
\end{figure}
\begin{lemma} [Swapping Lemma]\label{swapping lemma}
    Let $u_i\in  \mathcal M(x_i,y)$ and $v_i\in \mathcal M(y,z_i)$ for $i=1,\dots,n$ and let $\{x_i\}$ and $\{z_i\}$ be sets of pairwise distinct orbits. Then the maps $\{u_i\#v_i\}_i$ are pairwise admissible and for every $\sigma$ in the symmetric group of $n$ elements, we have $$\sum_{i \neq j} \iota(u_i\#v_i,u_j\#v_j) = \sum _{i \neq j}\iota(u_i\#v_{\sigma(i)},u_j\#v_{\sigma(j)}).$$ 
\end{lemma}
\begin{proof}
    By the asymptotic behavior of the curves, the intersections of $\hat u_i, \hat u_j$  and $\hat v_i,\hat v_j$ respectively are disjoint from $\hat y$ and do not accumulate at $\hat y$ (see, for example, \cite[Cor. 2.5]{Si08}, whose hypothesis holds due to unique continuation). After a small homotopy through admissible pairs of curves, small enough that they intersect transversely, positively, and only finitely many times away from $\hat y$. Hence, we may assume that the curves $\{\hat u_i\#\hat v_i\}$ are in degenerate breaking position and Lemma~\ref{intersection number and knots} we have $$\iota(u_i\#v_i,u_j\#v_j) = \lambda(u_i\#v_i,u_j\#v_j) + |u_i\cap u_j|+|v_i\cap v_j|.$$
    Now, $\lambda(u_i\#v_i,u_j\#v_j) = \ell^+_{i,j} -\ell^-_{i,j}$ where $\ell^\pm_{i,j}$ is the linking number induced by the knots of $\hat u_i,\hat u_j$ (for negative sign) and $\hat v_i,\hat v_j$ (for positive sign) near $\hat y$. Thus, we have \begin{align*}\begin{split}\sum_{i\neq j} \iota(u_i\#v_i,u_j\#v_j) &= \sum_{i\neq j}\ell^+_{i,j}-\ell^-_{i,j} + \,|\,u_i\cap u_j\,|\,+ \,|\,v_i\cap v_j\,|\, \\ & = \sum_{i\neq j}\ell^+_{\sigma(i),\sigma(j)}-\ell^-_{i,j} + \,|\,u_i\cap u_j\,|\,+ \,|\,v_{\sigma(i)}\cap v_{\sigma(j)}\,|\,
    \\ &= \sum_{i\neq j}\iota (u_i\#v_{\sigma(i)},u_j\#v_{\sigma(j)})\end{split}.\end{align*}
\end{proof}

\section{The Braided Floer Chain Complex and its Homology}
\label{section: braided floer homology}
A \emph{(closed pure surface) braid} is a set of embeddings $\{s_1,\dots,s_k\}$ with pairwise disjoint images where the \emph{strands} $s_i$ are given by $$s_i: S^1\to S\times S^1,\quad t\mapsto(x_i(t),t)$$ for loops $x_i :S^1\to S$. We say two braids are isotopic if an isotopy of the strands makes their images agree. We refer to an isotopy class of this type as a \emph{braid class with $k$ strands}.
For a smooth Hamiltonian $H: S^1\times S\to \mathbb R$ we define $$\mathcal P(H) = \{x:S^1\to S\,|\,x(t) = \phi^t_H(x(0))\}$$ the space of 1-periodic orbits where we identify $S^1 = \mathbb R /\mathbb Z$. For every $k\in \mathbb Z_{\geq 1}$ and a set $X$, we denote by $\Delta_X\subset  X^k$ the big diagonal $$\Delta_X = \{(x_1,\dots,x_k)\in X^k\,|\,x_i = x_j\text{ for some $i\neq j$}\}.$$ Moreover, let $$B_k\mathcal P(H)= (\mathcal P(H)^k-\Delta_{\mathcal P(H)})/S_k$$ be the set of \emph{braids of orbits with k strands}, where we quotient by the action of the symmetric group of $k$ elements $S_k$ that permutes components. Any set $x = \{x_1,\dots,x_k\}\in B_k\mathcal P(H)$ gives rise to a closed surface braid defined by $$\mathcal B(H,x) = [\{(x_1(t),t),\dots,(x_k(t),t)\}_{t\in S^1}].$$ Here the brackets $[\cdot ]$ denote the isotopy class of the closed surface braid. The set of unordered $k$-tuples of orbits that form a $\beta$-braid is denoted by $$B\mathcal P_\beta(H) = \{x\in B_k\mathcal P(H)\,|\, \mathcal B(H,x) =  \beta\}.$$ The ordered $k$-tuples of orbits that form a $\beta$-braid is defined by  $$OB\mathcal P_\beta(H): = \pi^{-1}(B\mathcal P_\beta(H)),$$ where $\pi:OB\mathcal P_\beta(H)\to B_\beta\mathcal P(H)$ is the map that forgets the ordering. An element in $B\mathcal P_\beta(H)$ is called a \emph{$\beta$-braid of orbits} of the flow of $H$. An orbit in a braid of orbits is called a \emph{strand orbit}. The set of strands in a $\beta$-braid of orbits is denoted by $$\mathcal P_\beta(H) = \{x_1\in \mathcal P(H)\,|\, \exists \{x_1,\dots, x_k\} \text{ such that } \mathcal B(H,\{x_1,\dots,x_k\}) =\beta\}.$$ 
The vector space generated by $\beta$-braids of orbits is denoted $$\operatorname{BCF}_{\beta}(H) = \langle B\mathcal P_{\beta}(H)\rangle_{\mathbb Z_2}.$$ 
As the notation suggests, we call these spaces \emph{braided Floer complexes in the braid class $\beta$}. For a non-degenerate Hamiltonian $H$ and the additional generic auxiliary data of an almost complex structure, these vector spaces can be endowed with a differential $\partial$ that preserves the homotopy class of the braids, justifying the name. 
\subsection{The Differential and Continuation Maps}
In this section we define a homological differential $\partial$ on the space $\operatorname{BCF}_\beta(H,J)$ and prove that $\partial^2 = 0$. This differential counts Floer-theoretic isotopies that are given by tuples of Floer cylinders of Conley--Zehnder index 
$1$ that have pairwise vanishing intersection numbers.

Unfortunately, the case $S = T^2$ brings some additional technicalities when considering braids~$\beta$ with strands represented by non-contractible loops. This is because the moduli space defined shortly is not compact, since there are no global action bounds on holomorphic curves connecting non-contractible orbits. \textbf{In the remainder of Section~\ref{section: braided floer homology} we will only consider braids on $T^2$ where every strand is contractible.} However, the author believes that introducing Novikov coefficients resolves this issue, and that the proofs below go through analogously. For simplicity, we will not address this situation.

The Conley--Zehnder index for a (not necessarily contractible) orbit $x$ is defined as follows: For every homotopy class $\alpha$ and fixed reference loop $\eta_\alpha$ in this class, choose a symplectic trivialization of $\eta_\alpha^*TS$. Choose a cylinder $u$ from $\eta_\alpha$ to $x$ and push the trivialization along  $u$. This induces a symplectic trivialization on $x$, which is unique up to homotopy. The trivialization allows us to represent the differential $((d\phi_H^t)_{x(t)})_{t\in [0,1]}$ as a path of symplectic matrices. The index $\mu(x,u)$ is now given by the \emph{Maslov index} of this path. For genus $g\geq 1$ the index is independent of $u$, and we write $\mu(z)$, while for $g=0$ the index differs by a multiple of $4$ for different choices (since  $c_1(S^2) = 2$), and we write $\mu(z)$ for the index in $\mathbb Z_4$. For a cylinder $w:\mathbb R \times S^1\to S$ connecting periodic orbits $x(t) = \lim_{s\to -\infty}w(s,t)$ and $y(t) = \lim_{s \to \infty}w(s,t)$ we define the index $$\mu(w) = \mu(x,u)-\mu(y,u\#w)$$ which is independent of the capping $u$ of $x$. Before we get into the technicalities, note the following. \begin{rem}  
In general, the homology is not invariant under the choice of Hamiltonian. For a $C^2$-small Hamiltonian $h$, the generated Hamiltonian diffeomorphism has only non-degenerate and constant periodic orbits. Hence, in any non-trivial braid class $\beta$, the braided Floer homology $\operatorname{BHF}_\beta(h,J_h)$ is trivial for any choice of generic almost complex structure $J_h$. However, this need not be the case in general. For example, consider a surface of genus $g\geq 2$ and a Hamiltonian $H$ on $S$. Let $\beta_{H,\operatorname{tot}}$ be the braid class given by the collection of all periodic orbits. Hence $\operatorname{BCF}_{\beta_{H,tot}}(H)$ is one-dimensional. Since the total braid contains the orbit of minimal and maximal Conley--Zehnder index, no incoming or outgoing Floer isotopies can exist from this braid of orbits. Hence, the class is non-trivial in homology and $\operatorname{BHF}_{\beta_{H,tot}}(H,J_H)$ is one-dimensional. Now, in general $\beta_{H,tot}$ is non-trivial and hence $\operatorname{BHF}_{\beta_{H,tot}}(H,J_H)\neq \operatorname{BHF}_{\beta_{H,tot}}(h,J_h)$. \end{rem}

For every non-degenerate Hamiltonian $H$ a generic choice of almost complex structure $J$ ensures that the moduli space $\mathcal M_\ell(x,y)$ of solutions $u:\mathbb R\times S^1\to S$ to the Floer equation $$\partial_s u+J(\partial_t u-X_H\circ u) = 0$$ that are of index $\ell$ and connect periodic orbits $x$ to $y$ is a manifold of dimension $\ell$ when equipped with the $C^\infty_{\text{loc}}$-topology (see \cite{AD14}). We call such a pair $(H,J)$ \emph{regular}. For all $x,y\in \mathcal P(H)^k\setminus \Delta_{\mathcal P(H)}$ we define $$\mathcal M_{\ell}(x,y) = \{u\in \mathcal M_{\ell}(x_1,y_1)\times \cdots\times \mathcal M_{\ell}(x_k,y_k)\,|\,\iota(u)=0\}$$ which are the connected components of $\mathcal M_{\ell}(x_1,y_1)\times\cdots\times \mathcal M_\ell(x_k,y_k)$ on which the intersection number evaluates to $0$. On $\mathcal M_\ell(x,y)$, an $\mathbb R^k$-action shifts the $s$-coordinate in every component. The quotient $$\widehat{\mathcal M}_\ell(x,y) := \mathcal M_\ell(x,y) / \mathbb R^k$$ is a manifold of dimension $(\ell-1)k$. In the case $\ell =1$, all components of $\widehat{\mathcal M}_\ell(x,y)$ are $0$-dimensional manifolds and compact, as in standard Hamiltonian Floer theory. Hence, $\widehat {\mathcal M}_1(x,y)$ is a finite set, which allows us to define the differential of braided Floer complexes by $$\partial:\operatorname{BCF}_\beta(H,J)\to \operatorname{BCF}_\beta(H,J), \quad [x]\mapsto \sum_{y\in  OB\mathcal P_\beta(H)}\,|\,\widehat{\mathcal M}_1(x,y)\,|\,[y].$$ For convenience we will write $\xi(x,y) = \,|\,\widehat{\mathcal M}_1(x,y)\,|\,.$ 
\begin{lemma}\label{lemma: differential is well defined}
    The differential $\partial$ is well-defined.
\end{lemma}
\begin{proof}
   First, note that for every $\sigma \in S_k$ we have $$\sum_{y\in OB\mathcal P_{\beta}(H)}\xi(x,y)[y] = \sum_{y\in OB\mathcal P_{\beta}(H)}\xi(\sigma(x),\sigma(y))[\sigma(y)] = \sum_{y\in OB\mathcal P_{\beta}(H)}\xi(\sigma(x),y)[y].$$ Hence, $\partial$ is independent of the choice of a representative of $[x]$.
\end{proof}
\begin{rem}
    The homology theory we build differs essentially from the one defined by van den Berg, Ghrist, Vandervorst, and Wójcik in \cite{BGVW14}. This is already evident from the fact that their braid Floer homology is independent of $H$. A concrete difference is the following. In \cite{BGVW14} the differential essentially counts single cylinders of index $1$ that connect a collection of orbits that form a braid in some fixed (relative) class. In contrast, we count collections of multiple cylinders all of index $1$ that connect strands of the braids pairwise.
\end{rem}
Let us now prove that $\partial$ is indeed a differential.
\begin{prop}\label{differential squares to 0}
    For every regular pair $(H,J)$ and braid class $\beta$ we have $\partial^2 = 0$. 
\end{prop}
This shows that $\operatorname{BCF}_{\beta}(H,J)$ is an honest chain complex. We denote its homology by $$\operatorname{BHF}_\beta(H,J) = \ker(\partial)/\text{im}(\partial).$$
\begin{proof}
    Note that $\partial ^2$ counts exactly the number of tuples of broken trajectories with intersection number $0$ $$\partial^2([x]) = \sum_{\substack{y\in OB\mathcal P_\beta(H)}}\sum_{\substack{z\in OB\mathcal P_\beta(H)}} \xi(x,y)\xi(y,z)[z],$$ where $\xi(x,y)\xi(y,z)$ is the number of broken trajectories from $x$ to $z$ breaking at $y$ that have intersection number $0$. The proof reduces to showing the following.\\
    \underline{Claim:}
        Let $k$ be the number of strands of $\beta$. Then for every $x,z\in OB\mathcal P_\beta(H)$   we have $$\sum_{\sigma\in S_k}\sum_{\substack{y\in OB\mathcal P _\beta(H)}}\xi(x,y)\xi(y,\sigma(z)) = 0  \quad (\text{mod 2}).$$

    We first finish the proof of the proposition assuming this claim. Namely, we have \begin{align*}
        \begin{split}
            \partial ^2([x]) & = \sum_{y\in OB\mathcal P_\beta(H)}\sum_{z\in OB\mathcal P_\beta(H)} \xi(x,y)\xi(y,z)[z]\\
            & = \sum_{[z]\in B\mathcal P_\beta(H)} \sum_{\sigma\in S_k}\sum_{y\in OB\mathcal P_\beta(H)}\xi(x,y)\xi(y,\sigma(z))[z]\\ & = 0.
        \end{split}
    \end{align*}
    Let us now prove the claim. For periodic orbits $x,z\in \mathcal P(H)$, the space $\widehat {\mathcal M}_2(x,z)$ consists of a union of open intervals and copies of $S^1$. Let $\widehat {\mathcal M}^{\operatorname{int}}_2(x,z)$ denote the union of the connected components that are intervals and hence not yet compact.
    From standard Floer theory, it follows that the space $\widehat {\mathcal M}^{\operatorname{int}}_2(x,z)$ can be compactified by gluing broken trajectories $$\mathcal B(x,z) = \bigcup_{y\in \mathcal P(H)}\widehat{\mathcal M}_1(x,y) \times \widehat{\mathcal M}_1(y,z).$$ For every $x,z\in OB\mathcal P_\beta(H)$ let us denote the space $$\widehat{\mathcal M}_2^{\mathrm{cube}}(x,z) = \{(u_1,\dots,u_k) \in \widehat{\mathcal M}^{\operatorname{int}}_2(x_1,z_1)\times\cdots\times \widehat{\mathcal M}^{\operatorname{int}}_2(x_k,z_k)\,|\, \iota(u) = 0\}$$ which is the union of connected components of $\widehat{\mathcal M}_2(x,z)$ that are hypercubes. The space $\widehat{\mathcal M}^{\operatorname{int}}_2(x_1,z_1)\times\cdots\times \widehat{\mathcal M}^{\operatorname{int}}_2(x_k,z_k)$ can be compactified by compactifying each component. We obtain a union of closed hypercubes whose corner points are given by $$\mathcal B(x,z)= \bigcup _{y\in \mathcal P(H)^k}(\widehat{\mathcal M}_1(x_1,y_1)\times \widehat{\mathcal M}_1(y_1,z_1))\times\cdots\times (\widehat{\mathcal M}_1(x_k,y_k)\times \widehat{\mathcal M}_1(y_k,z_k)).$$ By Lemma~\ref{homotopy to breaking}, gluing preserves the algebraic intersection number of broken trajectories. Hence, $\widehat{\mathcal M}^{\mathrm{cube}}_2(x,z)$ can be compactified and we obtain a disjoint union of closed hypercubes with corner point $$\mathcal B^0(x,z) = \{((u_1,v_1),\dots,(u_k,v_k))\in \mathcal B(x,z)\,|\, \sum_{i \neq j}\iota(u_i\#v_i,u_j\#v_j)= 0\}.$$ Abusing notation, let us denote the compactification again by $\widehat{\mathcal M}^{\mathrm{cube}}_2(x,z)$, and note that this space is a compact manifold consisting of hypercubes that are given by products of closed intervals. The $0$-dimensional corner points are given by $\mathcal B^0(x,z)$ and therefore, we find that \begin{align}\label{equ: proof differential B0 even}\,|\,\mathcal B^0(x,z)\,|\, = 0 \quad \text{(mod 2)}.\end{align}
    To any element $w\in \mathcal B^0(x,z)$, we can associate its breaking orbits $y_w\in \mathcal P(H)^k$. An element $w\in\mathcal B^0(x,z)$ is called \emph{degenerate} if $y_w\in \Delta_{\mathcal P(H)}$, otherwise it is called \emph{regular}. Let us define $$\mathcal D(x,z) = \{w\in \mathcal B^0(x,z)\,|\,y_w\in \Delta_{\mathcal P(H)}\}$$ the set of degenerate broken trajectories.\\ $ $\\
    \underline{Observation:} Let $w\in \mathcal B^0(x,z)$ such that $y_w\in \mathcal P(H)^k$ is regular. Then we have $y_w\in OB\mathcal P_\beta(H)$ and therefore, $w\in \widehat{\mathcal M}_1(x,y_w)\times \widehat{\mathcal M}_1(y_w,z)$. \begin{proof}
    For $w=((u_1,v_1),\dots,(u_k,v_k))\in \mathcal B^0(x,z)$ such that $y_w$ is regular, we have that $$0= \sum_{i\neq j}\iota(u_i\#v_i,u_j\#v_j) = \sum_{i\neq j}\iota(u_i,u_j)+\iota(v_i,v_j)$$ and hence $\iota(u_i,u_j) = \iota(v_i,v_j) = 0$ for all $i \neq j$. Abusing notation let us write $u_i\in \mathcal M(x_i,(y_w)_i)$ for a representative of $u_i$ and note that we now have $u_i(s,t) \neq u_j(s,t)$ for all $(s,t)\in \mathbb R\times S^1$ and $i\neq j$. Hence, the cylinders $u_1,\dots,u_k$ give rise to an isotopy of braids associated to $x$ and $y_w$ and therefore $y_w\in OB \mathcal P_\beta(H)$. This concludes the proof of the observation.
    \end{proof}
    By the observation, $$\mathcal B^0(x,z) = \bigcup_{y\in OB\mathcal P_\beta(H)}\widehat{\mathcal M}_1(x,y)\times \widehat{\mathcal M}_1(y,z)\cup \mathcal D(x,z) $$ since breaking orbits that lie in $OB\mathcal P_\beta(H)$ are exactly the regular ones. Together with (\ref{equ: proof differential B0 even}) we conclude that $$0 = \sum_{y\in OB \mathcal P_\beta(H)}\xi(x,y)\xi(y,z) + \,|\,\mathcal D(x,z)\,|\, \quad \text{(mod 2)}.$$ To finish the proof of the claim it now suffices to show that \begin{align}\label{equ: proof differential sum even}\sum_{\sigma\in S_k}\,|\,\mathcal D(x,\sigma(z))\,|\, = 0\quad \text{(mod 2)}.\end{align}
   For $y\in \Delta_{\mathcal P(H)}$ denote $$\mathcal D(x,y,z) = \{w\in \mathcal D(x,z)\,|\,y_w = y\}$$ and note that $$\mathcal D(x,z) = \bigcup_{y\in \Delta_{\mathcal P(H)}}\mathcal D(x,y,z).$$ 
    The equality (\ref{equ: proof differential sum even}) follows from a symmetry argument. Namely, for $y\in \Delta_{\mathcal P(H)}$ write $$\text{Fix}(y) = \{\sigma\in S_k\,|\,\sigma(y) = y\}\subset S_k$$ for the subgroup of elements that fix $y$. We now define the map $$T^{x,y,z}_\sigma:\mathcal D(x,y,z)\to \mathcal D(x,y,\sigma(z)), \quad (u,v)\mapsto (u,\sigma(v)).$$ This is well-defined since after the swapping Lemma~\ref{swapping lemma} we have $\iota(u\#v) = \iota(u\#\sigma(v))$. Moreover, it is a bijection with  inverse given by $T^{x,y,\sigma(z)}_{\sigma^{-1}}$. Hence, \begin{align}\label{eq. symmetry argument}|\mathcal D(x,y,z)| = |\mathcal D(x,y,\sigma(z))|\end{align} for every $\sigma \in \text{Fix}(y)$. We now have \begin{align*}\begin{split}\sum_{\sigma \in S_k}|\mathcal D(x,\sigma(z))| & = \sum_{\sigma \in S_k}\sum_{y\in \Delta_{\mathcal P(H)}}|\mathcal D(x,y,\sigma(z))\,|\,\\ & = \sum_{y\in \Delta_{\mathcal P(H)}}\sum_{[\mu]\in S_k/\text{Fix}(y)}\sum _{\sigma\in \text{Fix}(y)}|\mathcal D(x,y,\sigma(\mu(z)))|\\ & \overset{\mathrm{(\ref{eq. symmetry argument})}}{=} \sum_{y\in \Delta_{\mathcal P(H)}}\sum_{[\mu]\in S_k/\text{Fix}(y)}\sum _{\sigma\in \text{Fix}(y)}|\mathcal D(x,y,\mu(z))| \\ & =  \sum_{y\in \Delta_{\mathcal P(H)}}\sum_{[\mu]\in S_k/\text{Fix}(y)}|\text{Fix}(y)|\,|\mathcal D(x,y,\mu(z))|
\end{split},
    \end{align*}
    where $S_k/\text{Fix}(y)= \{\text{Fix}(y)\mu\,|\,\mu\in S_k\}$. The subgroup $\text{Fix}(y)$ is a product of permutation groups, so its cardinality $|\text{Fix}(y)|$ is divisible by $2$.
    \end{proof}
\subsection{Continuation Maps}
We now define continuation maps on the braided Floer chain complex and show that they are chain maps. We first recall some standard Hamiltonian Floer theory. Given two non-degenerate Hamiltonians $H^{-\infty}$ and $H^\infty$ and generic choices of smooth loops of almost complex structures $J_{\pm \infty}= (J_{t,\pm\infty})_{t\in S^1}
$, the \emph{continuation data $\Gamma= (H,J)$} connecting  $(H_{-\infty},J_{-\infty})$ and $(H_\infty,J_\infty)$ consists of a Hamiltonian $H:\mathbb R\times S^1\times S\to \mathbb R$ and a smooth family of almost complex structures $J = (J_{s,t})_{t\in S^1, s\in\mathbb R}$ such that there exists $R> 0$ and $$H(s,t,x) = \begin{cases}
    H_{-\infty}(t,x),& s \leq  -R\\
    H_{\infty}(t,x),& s\geq R
\end{cases}$$
 and $$J_{s,t}= \begin{cases}
     J_{t,-\infty},& s\leq -R\\
     J_{t,\infty}, & s \geq R.
 \end{cases}$$
For periodic orbits $x\in \mathcal P(H^{-\infty})$ and $y\in \mathcal P(H_{\infty})$ we denote by $\mathcal M_\ell^\Gamma(x,y)$ the space of solutions $u:S^1\times \mathbb R \to S$ of the Floer equation $$\partial_su + J_{s,t}(\partial_t u - X_{H_{s,t}}\circ u) = 0.$$
Generic continuation data $\Gamma$ is \emph{regular}, which is a transversality condition ensuring that the subspace $\mathcal M_\ell^{\Gamma}(x,y)\subset \mathcal M^\Gamma(x,y)$ of solutions that are cylinders with Maslov index $\ell$ is a manifold of dimension $\ell$ equipped with the $C^\infty_{\text{loc}}$-topology (see \cite{AD14}).

Now given $x\in \mathcal P(H_{-\infty})^k\setminus\Delta_{\mathcal P(H_{-\infty})}$ and $y\in \mathcal P(H_{\infty})^k\setminus\Delta_{\mathcal P(H_{\infty})}$ a \emph{continuation isotopy} connecting $x$ and $y$ is given by a tuple $u=(u_1,\dots,u_k)$ of cylinders $u_i\in \mathcal M^{\Gamma}(x_i,y_i)$ such that for all $i\neq j$ and for all $(s,t)\in \mathbb R \times S^1$ we have $u_i(s,t) \neq u_j(s,t) $. This condition can be imposed purely algebraically. Namely, consider the intersection number $$\iota: \mathcal M^{\Gamma}(x_1,y_1)\times\cdots\times \mathcal M^{\Gamma}(x_k,y_k)\to \mathbb Z_{\geq 0},\quad \iota(u) = \sum_{i\neq j} \iota(u_i,u_j).$$ By positivity of intersection, this number is non-negative, and if $\iota(u) = 0$ we have $\iota(u_i,u_j) = 0 $ for all $i\neq j$. This in particular means that there cannot be any $(s,t)\in \mathbb R\times S^1$ such that $u_i(s,t) = u_j(s,t)$ since this would amount to a positive intersection. In other words, $u$ is an isotopy if and only if $\iota(u) = 0$. Let us define the moduli space of continuation isotopies $$\mathcal M^{\Gamma}(x,y) = \{(u_1,\dots,u_k)\in \mathcal M^{\Gamma}(x_1,y_1) \times \cdots \times \mathcal M^\Gamma(x_k,y_k)\,|\,\iota(u) = 0\}$$ and denote the subspace of isotopies consisting of curves with index $\ell\in \mathbb Z$ by \begin{align*}\mathcal M_\ell^\Gamma(x,y) = \{(u_1,\dots,u_k)\in \mathcal M_\ell^\Gamma(x_1,y_1)\times \cdots\times \mathcal M_\ell^\Gamma(x_k,y_k)\,|\,\iota(u)=0\}.\end{align*} Note that any path in the space $\mathcal M^{\Gamma}(x_1,y_1)\times\cdots\times \mathcal M^{\Gamma}(x_k,y_k)$ is, in particular, a homotopy of cylinders relative to limiting orbits. This means that $\iota$ is locally constant and $\mathcal M^{\Gamma}(x,y)$ is simply given by the connected component of $\mathcal M^{\Gamma}(x_1,y_1)\times\cdots\times\mathcal M^{\Gamma}(x_k,y_k)$. Hence, for generic continuation data $\Gamma$, the moduli space $\mathcal M^{\Gamma}(x,y)$ is a manifold. Similarly, $\mathcal M^{\Gamma}_\ell(x,y)$ is a manifold of dimension $\ell k$. (Warning: $\mathcal M^{\Gamma}_{\ell}(x,y)$ is not the part of $\mathcal M^{\Gamma}(x,y)$ with local dimension $\ell k$.)

For regular pairs $(H,J_H)$ and $(G,J_G)$ and regular continuation data $\Gamma$ between the pairs let $$ \Phi^{\Gamma}:\operatorname{BCF}_{\beta}(H)\to \operatorname{BCF}_{\beta}(G),\quad [x] \mapsto \sum_{y\in OB\mathcal P_\beta(G)}|\mathcal M_0^\Gamma(x,y)|[y].$$
We from now on write $$\xi^\Gamma(x,y) = |\mathcal M^{\Gamma}_0(x,y)|.$$
As in Lemma~\ref{lemma: differential is well defined}, this map is well-defined.
\begin{prop}\label{continuation are chain maps}
    Let $(H,J_H)$ and $(G,J_G)$ be regular pairs and $\Gamma$ regular continuation data connecting $(H,J_H)$ to $(G,J_G)$. Then continuation maps are chain maps, $$\Phi^{\Gamma}\circ \partial_{(H,J_H)} = \partial_{(G,J_G)}\circ\Phi^{\Gamma}.$$
\end{prop}

Before proving Proposition~\ref{continuation are chain maps}, we first discuss the proof strategy. To conclude that $\Phi^\Gamma\circ \partial
_{(H,J_H)}+\partial_{(G,J_G)}\circ \Phi^\Gamma = 0$ we need to show that the number of tuples of broken trajectories $((u_1,v_1),\dots,(u_k,v_k))$ connecting braided orbits $x$ to $y$ and $y$ to $z$ which are of type $\beta$ is even. First for $x,z\in OB\mathcal P_\beta (H)$ we define a moduli space which is given by a compactification of a subspace of the product $$\mathcal M^{\Gamma}_1(x_1,z_1)\times\cdots\times \mathcal M^{\Gamma}_1(x_k,z_k)$$ and consists of a union of hypercubes $$Q = I_1\times \cdots\times I_k,$$ where $I_j$ is a closed interval in the compactification of $\mathcal M^{\Gamma}_1(x_j,z_j)$. This moduli space is such that all broken trajectories $((u_1,v_1),\dots,(u_k,v_k))$ as above appear as corner points of its hypercubes. Such corner points, corresponding to tuples of broken trajectories, come in two types: $$\text{Type 1: }(u_i,v_i) \in \mathcal M^{\Gamma}_0(x_i,y_i)\times \widehat{\mathcal M}_1(y_i,z_i)  $$ $$\text{Type 2: }(u_i,v_i) \in  \widehat{\mathcal M}_1(x_i,y_i)\times \mathcal M^{\Gamma}_0(y_i,z_i).$$ A corner point is called \emph{pure} if all components are either of type I or of type II. These pure corner points are exactly the ones that are counted by the map $\Phi^\Gamma\circ \partial
_{(H,J_H)}+\partial_{(G,J_G)}\circ \Phi^\Gamma$. To conclude that the number of these corner points is even, we will apply the following combinatorial lemma. 
\begin{lemma}\label{combinatorial lemma}
     Let $\rho_{i}:\{\pm 1\} \to \{a,b\}$ for $i=1,\dots,k$ be colorings and view the set $\{\pm 1\}^k$ as the corner points of a hypercube. We call a corner point $c = (c_1,\dots,c_k)\in \{\pm 1\}^k$ pure if $\rho_i(c_i) = a$ for all $i$ or $\rho_i(c_i) = b$ for all $
     i$. Then the number of pure corner points is even.
 \end{lemma}
 \begin{proof}
      We prove the lemma by induction over $k$. Let $n_a$ and $n_b$ denote the number of $a$-pure and $b$-pure corner points, respectively, of a $k$-dimensional hypercube. For $k = 1$, the lemma holds; there are two corner points, and both are pure. Now suppose the result holds for $k\geq 1$. Let $\tilde n_a$ and $\tilde n_b$ denote the number of pure $a$- or $b$-corner points of the hypercubes obtained by forgetting the last coordinate. Using that $2 = |\rho^{-1}_{k+1}(a)|+|\rho^{-1}_{k+1}(b)|$ we find \begin{align*}\begin{split}n_a+n_b & = \tilde n_a|\rho_{k+1}^{-1}(a)|+\tilde n_b\,|\,\rho^{-1}_{k+1}(b)| \\ & = \tilde n_a(2-|\rho_{k+1}^{-1}(b)|)+\tilde n_b|\rho_{k+1}^{-1}(b)|\\ & = 2\tilde n_a+ (\tilde n_b-\tilde n_a)|\rho_{k+1}^{-1}(b)|\end{split}\end{align*} which is even since $\tilde n_a+\tilde n_b$ and therefore also $\tilde n_b-\tilde n_a$ is even by assumption. 
 \end{proof}
The boundary components of $I_j$ can be colored by the type of broken trajectories. We will apply Lemma~\ref{combinatorial lemma} to this coloring.
\begin{proof}[Proof of Proposition~\ref{continuation are chain maps}]
    Note that for $x\in OB\mathcal P_\beta(H)$ the map $\Psi= \Phi^{\Gamma}\circ \partial_{(H,J_H)} + \partial_{(G,J_G)}\circ\Phi^{\Gamma}$ is given by $$\Psi([x]) = \sum_{\substack{z\in OB\mathcal P_\beta(G)}}\eta(x,z)[z],$$ where $\eta(x,z)$ is the cardinality of the following set consisting of broken trajectories, $$\mathcal B(x,z) =  \bigcup_{\substack{y\in OB\mathcal P_\beta(G)}}\mathcal M_0^\Gamma(x,y)\times \widehat{ \mathcal M}_1(y,z) \cup \bigcup_{\substack{y\in OB\mathcal P_\beta(H)}}\widehat{\mathcal M}_1(x,y)\times \mathcal M^{\Gamma}_0(y,z).$$ To prove the proposition, we need to show that $\Psi = 0$. First note that we have \begin{align*}\begin{split}\Psi([x]) & = \sum_{z\in OB\mathcal P_\beta(G)}\eta(x,z)[z] = \sum_{[z]\in B\mathcal P_\beta(G)}\sum_{\sigma\in S_k}\eta(x,\sigma(z))[z].\end{split}\end{align*} In the rest of the proof, we will show that for $x\in OB\mathcal P_\beta(H)$ and $z\in OB\mathcal P_\beta(G)$ we have \begin{align}\sum_{\sigma\in S_k}\eta(x,\sigma(z)) = 0 \quad \text{(mod 2)}\end{align} which proves the proposition.\\

   For $x\in \mathcal P(H)$ and $z\in \mathcal P(G)$, the moduli space $\mathcal M_1^\Gamma(x,z)$ is a  disjoint union of open intervals and copies of $S^1$. Let $\mathcal M_1^{\Gamma,\operatorname{int}}(x,z)$ denote the subset that is given by the union of all intervals. From standard Floer theory we know that $\mathcal M^{\Gamma,\operatorname{int}}_1(x,z)$ can be compactified by gluing broken trajectories $$\bigcup_{y\in \mathcal P(G)}\mathcal M^{\Gamma}_0(x,y)\times \widehat{\mathcal M}_1(y,z)\cup \bigcup_{y\in \mathcal P(H)}\widehat{\mathcal M}_1(x,y)\times \mathcal M^{\Gamma}_0(y,z)$$ as boundary components.
   It follows that for $x\in OB\mathcal P_\beta(H)$ and $z\in OB\mathcal P_\beta(G)$ the space $\mathcal M^{\Gamma,\operatorname{int}}_1(x_1,z_1) \times\cdots\times \mathcal M^{\Gamma,\operatorname{int}}_1(x_k,z_k)$ consisting of open hypercubes can be compactified to a disjoint union of closed hypercubes; corner points are given by
   $$\mathcal C(x,z) = \prod_{i =1}^k\left (\bigcup_{y_i\in \mathcal P(G)}\mathcal M_0^\Gamma(x_i,y_i)\times \widehat{ \mathcal M}_1(y_i,z_i) \cup \bigcup_{y_i\in \mathcal P(H)}\widehat{\mathcal M}_1(x_i,y_i)\times \mathcal M_0^{\Gamma}(y_i,z_i)\right ) .$$ Let now $$\mathcal M^{\Gamma,\mathrm{cube}}_1(x,z)= \{(u_1,\dots,u_k)\in \mathcal M^{\Gamma,\operatorname{int}}_1(x_1,z_1)\times\cdots\times \mathcal M^{\Gamma,\operatorname{int}}_1(x_k,z_k)\,|\,\iota(u) = 0\}$$ which is the subset of $\mathcal M^{\Gamma}_1(x,z)$ given by the disjoint union of all hypercubes.
   By Lemma~\ref{homotopy to breaking} the intersection number of curves at corner points agrees with the intersection number of curves given by points in the interior of the hypercube. Hence, by taking the subset on which the intersection number $\iota$ evaluates to $0$, we find that $\mathcal M^{\Gamma,\mathrm{cube}}_1(x,z)$ can be compactified to a union of hypercubes with corner points given by $$\mathcal C^0(x,z) = \{((u_1,v_1),\dots,(u_k,v_k))\in \mathcal C(x,z)\,|\,\sum_{i\neq j}\iota(u_i\#v_i,u_j\#v_j) = 0\}.$$  A component of a corner point is given by a broken trajectory and can come in two types $$\text{Type 1: }(u_i,v_i) \in \mathcal M^{\Gamma}_0(x_i,y_i)\times \widehat{\mathcal M}_1(y_i,z_i)  $$ $$\text{Type 2: }(u_i,v_i) \in  \widehat{\mathcal M}_1(x_i,y_i)\times \mathcal M^{\Gamma}_0(y_i,z_i).$$ A corner point $w\in \mathcal C^0(x,z)$ is called \emph{pure} if all components of $w$  are of the same type, that is $w$ lies in \begin{align*}\left (\prod_{i=1}^k \bigcup_{y_i\in \mathcal P(G)}\mathcal M_0^\Gamma(x_i,y_i)\times \widehat{ \mathcal M}_1(y_i,z_i) \right)\cup \left ( \prod_{i = 1}^k \bigcup_{y_i\in \mathcal P(H)}\widehat{\mathcal M}_1(x_i,y_i)\times \mathcal M_0^{\Gamma}(y_i,z_i)\right )\cap \mathcal C^0(x,z)\end{align*} 
   Note that this set is in bijection to the set $$\mathcal C^0_{\operatorname{pure}}(x,z) = \bigcup_{y\in \mathcal P(G)^k} \mathcal M^{\Gamma}_0(x,y)\times \widehat{\mathcal M}_1(y,z)\cup \bigcup_{y\in \mathcal P(H)^k}\widehat{ \mathcal M}_1(x,y) \times \mathcal M^{\Gamma}_0(y,z)$$
   Now we can apply the combinatorial Lemma~\ref{combinatorial lemma} to deduce that \begin{align}\label{eq: Cpure even cardinality}
       |\mathcal C^0_{\operatorname{pure}}(x,z)| = 0 \quad \text{(mod 2)}.
   \end{align} Indeed, any $\mathcal M^{\Gamma,\mathrm{cube}}_1(x,z)$ is given by a disjoint union of hypercubes, which are products of intervals. The boundary components of these intervals can be colored with respect to their type of broken trajectory. The lemma asserts that the number of corner points with components of the same type is even. However, these are exactly the pure corner points.

   Now, for any $w\in \mathcal C^0_{\operatorname{pure}}(x,z)$, we denote its breaking orbit by $y_w\in \mathcal P(H)^k$ or $y_w\in \mathcal P(G)^k$, depending on its type of breaking. A corner point $w$ is called \emph{degenerate} if $y_w\in \Delta_{\mathcal P(H)}\cup  \Delta_{\mathcal P(G)}$ and is called \emph{regular} otherwise. Let $$\mathcal D(x,z) = \{w\in \mathcal C^0_{\operatorname{pure}}(x,z)\,|\,y_w\in \Delta_{\mathcal P(H)}\cup \Delta_{\mathcal P(G)}\}$$ be the set of degenerate pure cornerpoints. \\ $ $\\
   \underline{Observation 1:} Let $w\in \mathcal C^0_{\operatorname{pure}}(x,z)$ such that $y_w$ is regular. Then $y_w\in OB \mathcal P_\beta(H)$ or $y_w\in OB\mathcal P_\beta(G)$ depending on the type of breaking.\begin{proof}
       Assume that $y_w\in \mathcal P(H)^k\setminus \Delta_{\mathcal P(H)}$, the other case is analogous. Suppose that $w=((u_1,v_1),\dots,(u_k,v_k))$ so that $$0 = \sum_{i\neq j}\iota(u_i\#v_i,u_j\#v_j) = \sum_{i\neq j}\iota(u_i,u_j) + \iota(v_i,v_j),$$ where we use that $y_w\notin\Delta_{\mathcal P(H)}$ for the last equality. By positivity of intersection we can conclude that $\iota(u_i,u_j) = \iota(v_i,v_j) = 0$ for all $i\neq j$. Hence for every $(s,t)\in \mathbb R\times S^1$ we have $u_i(s,t) \neq u_j(s,t)$ and the cylinders $u_1,\dots,u_k$ induces an isotopy of braids corresponding to $x$ and $y$. Since the associated braid of $x$ is of class $\beta$, so must the braid of $y$ be as well. This concludes the proof of the observation.
   \end{proof} We now have that \begin{align*}\begin{split}\mathcal C^0_{\operatorname{pure}}(x,z) &=  \bigcup_{y\in \mathcal P(G)^k\setminus \Delta_{\mathcal P(G)}} \mathcal M^{\Gamma}_0(x,y)\times \widehat{\mathcal M}_1(y,z)\cup \bigcup_{y\in \mathcal P(H)^k\setminus \Delta_{\mathcal P(H)}}\widehat{ \mathcal M}_1(x,y) \times \mathcal M^{\Gamma}_0(y,z)\\ & \cup \bigcup_{y\in \Delta_{\mathcal P(G)}} \mathcal M^{\Gamma}_0(x,y)\times \widehat{\mathcal M}_1(y,z)\cup \bigcup_{y\in \Delta_{\mathcal P(H)}}\widehat{ \mathcal M}_1(x,y) \times \mathcal M^{\Gamma}_0(y,z) \\ & =  \bigcup_{y\in OB\mathcal P_\beta(G)} \mathcal M^{\Gamma}_0(x,y)\times \widehat{\mathcal M}_1(y,z)\cup \bigcup_{y\in OB\mathcal P_\beta(H)}\widehat{ \mathcal M}_1(x,y) \times \mathcal M^{\Gamma}_0(y,z)\cup\mathcal D(x,z) \\ & = \mathcal B(x,z) \cup \mathcal D(x,z) \end{split}.\end{align*} To show that $$\sum_{\sigma\in S_k}\eta(x,\sigma(z)) = \sum_{\sigma\in S_k}|\mathcal B(x,\sigma(z))|= 0 \quad \text{(mod 2)}$$ it suffices to show that \begin{align}
       \sum_{\sigma\in S_k}|\mathcal D(x,\sigma(z)) | = 0 \quad \text{(mod 2)}
   \end{align} since after (\ref{eq: Cpure even cardinality}) we have $|\mathcal C^0_{\operatorname{pure}}(x,z)| = 0 \text{ (mod 2)}$. For every $y\in \Delta_{\mathcal P(H)}\cup \Delta_{\mathcal P(G)}$ let us define $$\mathcal D(x,y,z) = \{w\in \mathcal D(x,z)\,|\,y_w = y\}$$ so that $$\mathcal D(x,z) = \bigcup_{y\in \Delta_{\mathcal P(H)}\cup \Delta_{\mathcal P(G)}}\mathcal D(x,y,z).$$
   For every $y\in \Delta_{\mathcal P(H)}\cup \Delta_{\mathcal P(G)}$ let us define $$\text{Fix}(y) = \{\sigma \in S_k\,|\,\sigma(y) = y\}$$ as the subgroup of $S_k$ that fixes $y$. \\ $ $\\
   \underline{Observation 2:} For every $\sigma \in \text{Fix}(y)$ we have $$|\mathcal D(x,y,z)| = |\mathcal D(x,y,\sigma(z))|.$$
   \begin{proof}
       For every $\sigma\in \text{Fix}(y)$ let us define $$T_{\sigma}^{x,y,z}: \mathcal D(x,y,z) \to \mathcal D(x,y,\sigma(z)),\quad ((u_1,v_1),\dots,(u_k,v_k))\mapsto ((u_1,v_{\sigma(1)}),\dots,(u_k,v_{\sigma(k)}).$$ To see that this map is well defined, note that after the swapping Lemma~\ref{swapping lemma} we have that $$0 = \sum_{i\neq j}\iota(u_i\#v_i,u_j\#v_j ) = \sum_{i \neq j}\iota(u_i\#v_{\sigma(i)},u_j\#v_{\sigma(j)})$$ and hence $((u_1,v_{\sigma(1)}),\dots,(u_k,v_{\sigma(k)}))\in \mathcal C^0_{\operatorname{pure}}(x,\sigma(z))$. Since the breaking orbit $y$ was left unchanged, we have that $((u_1,v_{\sigma(1)}),\dots,(u_k,v_{\sigma(k)}))\in \mathcal D(x,y,\sigma(z))$. Now, $T_\sigma^{x,y,z}$ is a bijection since $T_{\sigma^{-1}}^{x,y,\sigma(z)}$ is its inverse. This concludes the proof of the observation.
   \end{proof}
   Now, the subgroup $\text{Fix}(y)$ is isomorphic to a product $S_{n_1}\times\cdots\times S_{n_m}$ where $n_i$ are the multiplicities of breaking orbits in $y$. Since $y\in \Delta_{\mathcal P(H)}\cup \Delta_{\mathcal P(G)}$ there must be some $i$ such that $n_i\geq 2$. It follows that $\,|\,\text{Fix}(y)\,|\, = 0 \text{ (mod 2)}$.
   We now have that \begin{align*}\begin{split}\sum_{\sigma\in S_k}\,|\,\mathcal D(x,\sigma(z))\,|\, & = \sum_{y\in \Delta_{\mathcal P(H)}\cup \Delta_{\mathcal P(G)}} \sum_{\sigma\in S_k}\,|\,\mathcal D(x,y,\sigma(z))\,|\, \\ & = \sum_{y\in \Delta_{\mathcal P(H)}\cup \Delta_{\mathcal P(G)}}\sum_{[\mu]\in S_k/\text{Fix}(y)}\sum_{\sigma \in \text{Fix}(y)} |\mathcal D(x,y,\sigma(\mu(z)))| \\ &  = \sum_{y\in \Delta_{\mathcal P(H)}\cup \Delta_{\mathcal P(G)}}\sum_{[\mu]\in \text{Fix}(y)S_k}|\text{Fix}(y)||\mathcal D(x,y,\sigma(\mu(z)))| \\ &  = 0 \quad \text{(mod 2)}\end{split}\end{align*}
   This concludes the proof of Proposition~\ref{continuation are chain maps}.
  \end{proof}
  Next, let us discuss the composition of continuation maps. Consider regular pairs $(H^a,J^a)$, $(H^b,J^b)$, and $(H^c,J^c)$ and continuation data $\Gamma' = (F',J')$ from $(H^a,J^a)$ to $(H^b,J^b)$ and continuation data $\Gamma''=(F'',J'')$ from $(H^b,J^b)$ to $(H^c,J^c)$. For a real number $\rho>0$ that is large enough we define continuation data $\Gamma''\#_\rho\Gamma'= (F''\#_\rho F',J''\#_\rho J')$ from $(H^a,J^a)$ to $(H^c,J^c)$ in the following way. $$F''\#_{\rho}F'(s,t,x) = \begin{cases}
      F'(s+\rho,t,x),& s\leq 0\\F''(s-\rho,t,x),&s \geq 0
  \end{cases},$$ $$(J''\#_\rho J')_s = \begin{cases}
      J'_{s+\rho}, & s \leq 0\\J''_{s-\rho}, & s \geq 0
  \end{cases}.$$ Note that since the continuation data is eventually constant in the $s$ parameter, the composition $\Gamma''\#_\rho \Gamma'$ is well-defined for $\rho$ large enough. From standard Floer theory (see for example \cite{AD14}) it follows that (after possibly a arbitrarily small perturbation of the continuation data) for every $x\in \mathcal P(H^a)$ any $z \in \mathcal P(H^c)$ there is a gluing map $$\chi: \bigcup_{y\in \mathcal P(H^b)}\mathcal M^{\Gamma'}_0(x,y)\times \mathcal M^{\Gamma''}_0(y,z) \to \mathcal M^{\Gamma''\#_\rho \Gamma'}_0(x,z)$$ which is a bijection. For every $(u,v) \in \mathcal M^{\Gamma'}_0(x,y)\times \mathcal M^{\Gamma''}_0(y,z)$, the curve $\chi(u,v)$ is obtained by the exponential flow along a pre-gluing map associated to $u$ and $v$ (see \cite{AD14}) and hence falls into the realm we discussed in Section~\ref{section.gluing}. 
\begin{prop} \label{proposition: continuation map composition}
    Consider regular pairs $(H^a,J^a)$,$(H^b,J^b)$, and $(H^c,J^c)$ and braid class $\beta$. For continuation data $\Gamma'$ from $(H^a,J^a)$ to $(H^b,J^b)$ and continuation data $\Gamma''$ from $(H^b,J^b)$ to $(H^c,J^c)$ we have
    $$\Phi^{\Gamma''}\circ \Phi^{\Gamma'}= \Phi^{\Gamma''\#_\rho \Gamma'}:\operatorname{BCF}_{\beta}(H^a)\to \operatorname{BCF}_{\beta}(H^c)$$ for $\rho>0$ large enough. 
\end{prop}

\begin{proof}
    For every $x\in OB\mathcal P_\beta(H)$  we have \begin{align*}\Phi^{\Gamma''}\circ&\Phi^{\Gamma'}([x])   =\sum_{\substack{y\in OB\mathcal P_{\beta}(H^b)}}\sum_{\substack{z\in OB\mathcal P_{\beta}(H^c)}}\xi^{\Gamma'}(x,y)\xi^{\Gamma''}(y,z)[z]\end{align*}
    and  $$\Phi^{\Gamma''\#_\rho\Gamma'}([x]) =  \sum _{z\in OB\mathcal P_{\beta}(H^c)}\xi^{\Gamma''\#_\rho\Gamma'}(x,z)[z].$$ We would like to prove $\Phi^{\Gamma''}\circ \Phi^{\Gamma'} = \Phi^{\Gamma''\#_{\rho}\Gamma'}$ which boils down to showing the following equality. For every $x\in OB \mathcal P_\beta(H^a)$ and $z \in OB\mathcal P_\beta(H^c)$ we have 
    \begin{align}\label{eq: counting unbounded chain maps}
        \sum_{\sigma\in S_k}\sum_{\substack{y\in OB \mathcal P_{\beta}(H^b)}}\xi^{\Gamma'}(x,y)\xi^{\Gamma''}(y,\sigma(z)) = \sum_{\sigma\in S_k}\xi^{\Gamma''\#_\rho\Gamma'}(x,\sigma(z))\quad (\text{mod 2}).
    \end{align}
    Assuming this, let us finish the proof of Proposition~\ref{proposition: continuation map composition}.     
   For every $x\in OB\mathcal P_\beta(H^a)$ \begin{align*}\begin{split}\Phi^{\Gamma''}\circ\Phi^{\Gamma'}([x]) &= \Phi^{\Gamma''}(\sum_{y\in OB\mathcal P_{\beta}(H^b)}\xi^{\Gamma'}(x,y)[y]) \\ & =\sum_{z\in OB\mathcal P_{\beta}(H^c)}\sum_{y\in OB\mathcal P_{\beta}(H^b)}\xi^{\Gamma'}(x,y)\xi^{\Gamma''}(y,z)[z] \\  & =\sum_{[z]\in B\mathcal P_{\beta}(H^c)}\sum_{\sigma\in S_k}\sum_{y\in OB\mathcal P_{\beta}(H^b)}\xi^{\Gamma'}(x,y)\xi^{\Gamma''}(y,\sigma(z))[z]\\ \\& =\sum_{[z]\in B\mathcal P_{\beta}(H^c)} \sum_{\sigma\in S_k}\xi^{\Gamma''\#_\rho\Gamma'}(x,\sigma(z))[z] \\ &=\sum_{z\in OB\mathcal P_{\beta}(H^c)} \xi^{\Gamma''\#_\rho\Gamma'}(x,z)[z] \\ & = \Phi^{\Gamma''\#_{\rho}\Gamma'}([x])\end{split}\end{align*}

We now prove the equality (\ref{eq: counting unbounded chain maps}).
      For every $x\in \mathcal P(H^a)$ and $z\in \mathcal P(H^c)$ there is a gluing map $$ \chi:\bigcup_{\substack{y\in \mathcal P(H^b)}}\mathcal M_0^{\Gamma'}(x,y)\times \mathcal M_0^{\Gamma''}(y,z)\to \mathcal M_0^{\Gamma''\#_\rho\Gamma'}(x,z)$$ that is constructed by gluing solutions to the Floer equation as in the setting of Section~\ref{section.gluing}. For every $x\in OB \mathcal P_\beta(H^a)$ and $z\in OB\mathcal P_\beta(H^c)$ this induces a bijective map $$\chi:\bigcup_{y\in \mathcal P(H^b)^k}\mathcal M^{\Gamma'}_0(x,y)\times \mathcal M^{\Gamma''}_0(y,z)\to \mathcal M^{\Gamma''\#\rho\Gamma'}_0(x,z),$$ $$((u_1,\dots u_k),(v_1,\dots,v_k))\mapsto (\chi(u_1,v_1),\dots,\chi(u_k,v_k)).$$ 
      For every $y\in \mathcal P(H^b)^k$ and $(u,v)=((u_1,\dots,u_k),(v_1,\dots,v_k))\in \mathcal M^{\Gamma'}(x,y) \times \mathcal M^{\Gamma''}(y,z)$ we have after stability of intersections (Lemma~\ref{homotopy to breaking}) that $$\iota(u\#v) = \iota(\chi(u,v))$$ since the map $\chi$ is given by pairwise gluing the cylinders of the tuples $u$ and $v$. Moreover, if $y\in\mathcal P(H^b)^k-\Delta_{\mathcal P(H^b)}$, where $\Delta_{\mathcal P(H^b)}$ is the big diagonal of $\mathcal P(H^b)^k$, then we additionally have $$\iota(u)+\iota(v) = \iota(u\#v)= \iota(\chi(u,v)) =0.$$ By positivity of intersections we have that $\iota(u) = \iota(v) = 0$ and we must have that $y\in OB\mathcal P_{\beta}(H^b)$ since the tuples of cylinders $u$ give rise to an isotopy.
      Let now, $$\mathcal D(x,z) = \left \{(u,v)\in\bigcup_{\substack{y\in \Delta_{\mathcal P(H^b)}}}\mathcal M_{0}^{\Gamma'}(x,y)\times \mathcal M_{0}^{\Gamma''}(y,z)\middle |\iota(\chi(u,v)) = 0\right\}.$$ For every $w\in \mathcal M^{\Gamma''\#\Gamma'}_{0}(x,z)$ we find a unique $(u,v) \in \bigcup_{y\in \mathcal P(H^b)^k}\mathcal M^{\Gamma'}(x,y)\times \mathcal M^{\Gamma''}(y,z)$ such that $\chi(u,v) = w$ and with tuple of breaking orbits $y\in \mathcal P(H^b)^k$. Let $y_w$ denote this tuple of breaking orbits. Now, $y_w$ is called \emph{degenerate} if $y_w\in \Delta_{\mathcal P(H^b)}$, otherwise it is called \emph {regular}. Above we saw that when $y_w$ is regular, we have $y_w\in OB\mathcal P_{\beta}(H^b)$. With this we can see that $$\sum_{y\in OB\mathcal P_{\beta}(H^b)}\xi^{\Gamma'}(x,y)\xi^{\Gamma''}(y,z)$$ is exactly the cardinality of $$\bigcup_{\substack{y\in \mathcal P(H^b)^k}}\mathcal M_{0}^{\Gamma'}(x,y)\times \mathcal M_{0}^{\Gamma''}(y,z)\setminus \mathcal D(x,z) = \bigcup_{\substack{y\in  OB\mathcal P_{\beta}(H^b)}}\mathcal M_{0}^{\Gamma'}(x,y)\times \mathcal M_{0}^{\Gamma''}(y,z).$$ It follows that we have $$\sum_{\substack{y\in OB\mathcal P_{\beta}(H^b)}}\xi^{\Gamma'}(x,y)\xi^{\Gamma''}(y,z) + \,|\,\mathcal D(x,z)\,|\, = \xi^{\Gamma''\#_\rho\Gamma'}(x,z).$$ To conclude the proof we show that $\sum_{\sigma\in S_k}\,|\,\mathcal D(x,\sigma(z))\,|\,$ is divisible by $2$. First, for every $y\in \Delta_{\mathcal P(H^b)}$ let us define $$\mathcal D(x,y,z) = \{(u,v) \in \mathcal M_0^{\Gamma'}(x,y)\times \mathcal M_0^{\Gamma''}(y,z)\,|\,\iota(\chi(u,v)) = 0\}$$ and note that $$\mathcal D(x,z) =  \bigcup_{y\in \Delta_{\mathcal P(H^b)}}\mathcal D(x,y,z)$$ is given by a disjoint union. For every $y\in \Delta_{\mathcal P(H^b)}$ let $$\text{Fix}(y) = \{\sigma \in S_k\,|\,\sigma(y) = y\}.$$ Now it follows from the swapping Lemma~\ref{swapping lemma} that for every $(u,v)\in \mathcal D(x,y,z)$ and $\sigma\in S_k$ we have $$\iota(u\#v) = \iota(u \#\sigma(v)).$$ In particular, if $(u,v)\in \mathcal D(x,y,z)$ then for every $\sigma \in \text{Fix}(y)$ we have $(u,\sigma(v)) \in \mathcal D(x,y,z)$. Hence, for every $y\in \Delta_{\mathcal P(H^b)}$ and $\sigma \in \text{Fix}(y)$ the map $$T^{x,y,z}_\sigma:\mathcal D(x,y,z) \to \mathcal D(x,y,\sigma(z)), \quad (u,v)\mapsto (u,\sigma(v))$$ is well-defined. Moreover, it is a bijection with inverse $T^{x,y,\sigma(z)}_{\sigma^{-1}}$. It follows that $$|\mathcal D(x,y,z)| = |\mathcal D(x,y,\sigma(z))|$$ for every $\sigma\in \text{Fix}(y)$. We calculate \begin{align*}
        \begin{split}
            \sum_{\sigma\in  S_k}|\mathcal D(x,\sigma(z))|& = \sum_{\sigma \in S_k}\sum_{y\in \Delta_{\mathcal P(H^b)}}|\mathcal D(x,y,\sigma(z))| \\ & = \sum_{y\in \Delta_{\mathcal P(H^b)}}\sum_{[\mu]\in S_k/\text{Fix}(y)}\sum_{\sigma \in \text{Fix}(y)}|\mathcal D(x,y,\sigma(\mu(z)))|\\ & =  \sum_{y\in \Delta_{\mathcal P(H^b)}}\sum_{[\mu]\in S_k/\text{Fix}(y)}|\text{Fix}(y)|\,|\mathcal D(x,y,\mu(z))|,
        \end{split}
    \end{align*}
     where $S_k/\text{Fix}(y)= \{\text{Fix}(y)\mu\,|\,\mu\in S_k\}$.
    Now, $\text{Fix}(y) \cong S_{n_1}\times\cdots\times S_{n_m}$ where $n_i$ are the multiplicities of the loops appearing in $y$. Since $y\in \Delta_{\mathcal P(H^b)}$ has at least one loop appearing at least twice, we have $n_i \geq 2$ for some $i$, so $|\text{Fix}(y)|$ is divisible by $2$. Hence, $\sum_{\sigma\in  S_k}|\mathcal D(x,\sigma(z))|$ is divisible by $2$. This concludes the proof of Proposition~\ref{proposition: continuation map composition}
    \end{proof}
\subsection{Braided Floer Homology of the Trivial Braid}
On a fixed symplectic surface $(S,\omega)$ of genus $g$, let $\beta_0$ denote the trivial braid with $2g+2$ strands. We now prove Theorem~\ref{thm: braided floer homology} by using continuation maps. We show this in the following way. For a $C^2$-small Hamiltonian $h$, it follows from standard arguments that for a suitable almost complex structure $J_h$ we have $$\dim \operatorname{BHF}_{\beta_0}(h,J_h) = 1.$$ For such a Hamiltonian $h$ we then show that for a regular pair $(H,J_H)$ there is an injective continuation map $$\Phi:\operatorname{BHF}_{\beta_0}(h,J_h)\hookrightarrow \operatorname{BHF}_{\beta_0}(H,J_H).$$ We show the injectivity of this map by showing that there is a left inverse. Namely, we first prove in Proposition~\ref{lemma: selfcontinuation C^2 small} that for certain regular pairs $(h,J_h)$ and continuation data $\Gamma$ from $(h,J_h)$ to itself, the continuation data is the identity on a chain level $$\Phi^{\Gamma} = id_{\operatorname{BCF}_{\beta_0}(h)}.$$ We then choose $\Gamma'$ from $(h,J_h)$ to $(H,J_H)$ and $\Gamma'' $ from $(H,J_H)$ to $(h,J_h)$ and have that $$id_{\operatorname{BHF}_{\beta_0}(h,J_h)}= \Phi^{\Gamma''\#_{\rho}\Gamma'} = \Phi^{\Gamma''}\circ \Phi^{\Gamma'}.$$ \begin{prop}\label{lemma: selfcontinuation C^2 small} Suppose $H$ is a Hamiltonian on a closed surface of genus $g$ such that the following holds. \begin{enumerate}
    \item There are exactly $2g+2$ 1-periodic orbits, and they are all constant. 
    \item For all periodic orbits $x,y\in \mathcal P(H)$ the set $\widehat {\mathcal M}_1(x,y)$ has even cardinality.
    \item For all $x,y\in \mathcal P(H)$, all $u \in \mathcal M_1(x,y)$ are time-independent, that is, $u(s,t) = u(s,0)$ for all $t\in S^1$.
\end{enumerate}
    Suppose that $(H,J)$ is a regular pair and $\Gamma$ is regular continuation data starting and ending at $(H,J)$. Then $$
    \Phi^\Gamma = id_{\operatorname{BCF}_{\beta_0}(H,J)}.$$
\end{prop}  

In the proof of Proposition \ref{lemma: selfcontinuation C^2 small} the following topological lemma will be essential. \begin{lemma} \label{lemma: cylinder with legs}
    Let $u,w,v_1,v_2: [-1,1]\times S^1\to S$ such that the following holds (see Figure \ref{im: bacteria}). \begin{enumerate}
        \item $v_1$ and $v_2$ are independent of $t\in S^1$ $$v_1(s,t) = v_1(s,0)\text{ and }v_2(s,t) = v_2(s,0).$$
        \item The ends of $v_1$ and $v_2$ agree $$v_1(-1,0) = v_2(-1,0)\quad \text{and}\quad v_1(1,0) = v_2(1,0).$$
        \item The ends of $u$ and $w$ are independent of $t\in S^1$ $$u(\pm1,t) = u(\pm 1,0) \text{ and }w(\pm1,t) =w(\pm 1,0).$$
        \item $u(1,0) = v_1(-1,0) =v_2(-1,0)$ so that $u\#v_1$ and $u\#v_2$ are well-defined.
        \item  $u(-1,0)\neq w(-1,0), v_1(1,0) \neq w(1,0)$ (and $v_2(1,0)\neq w(1,0)$) so that the intersection number $\iota(u\#v_1,w)
    $ and $\iota(u\#v_2,w)$ are well-defined.
    \end{enumerate} Then $$\iota(u\#v_1,w) = \iota(u\#v_2,w).$$
    
\end{lemma}
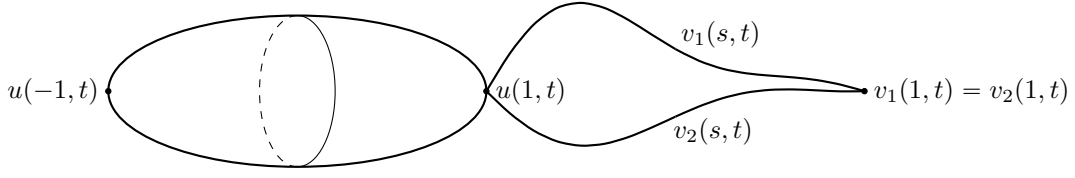
\begin{figure}[h!]
    \begin{tikzpicture}
         \filldraw[black] (-5,0) circle (1pt);
         \filldraw[black] (0,0) circle (1pt);
         \draw[thick] (-2.5,0) ellipse (2.5 and 1);
         \draw[] (-2.5,1) arc (-90:90:0.5 and -1);
         \draw[dashed] (-2.5,1) arc (-90:90:-0.5 and -1);
        \draw[thick,smooth,samples=9,domain=0:2.5] plot (\x,{(1-\x/5)*(0.5* sin(1.5*\x r) + 0.5 *sin(\x r)+  0.5*sin(0.5*\x r)+0.25*\x)}) coordinate (M1);
         \draw[thick,smooth,samples=9,domain=2.5:5] plot (\x,{(1-\x/5)*(0.5* sin(1.5*\x r) + 0.5 *sin(\x r)+  0.5*sin(0.5*\x r)+0.25*\x)}) coordinate (E1);
         \node[] at ($(M1)+(0.6,0.2)$) {$v_1(s,t)$};
        \filldraw[black] (E1) circle (1pt);
        \node[right] at (E1) {$v_1(1,t) = v_2(1,t)$};
        \draw[thick,smooth,samples=9,domain=2.5:5] plot (\x,{(1-\x/5)*(-0.4* sin(1.2*\x r) - 0.6 *sin(0.9*\x r)+  0.5*sin(0.5*\x r)-0.25*\x)}) coordinate (E2);
        \draw[thick,smooth,samples=9,domain=0:2.5] plot (\x,{(1-\x/5)*(-0.4* sin(1.2*\x r) - 0.6 *sin(0.9*\x r)+  0.5*sin(0.5*\x r)-0.25*\x)}) coordinate (M2);
        \node[] at ($(M2) + (0.5,-0.2)$) {$v_2(s,t) $};
        \node[left] at (-5,0) {$u(-1,t)$};
        \node[right] at (0,0) {$u(1,t)$};
    \end{tikzpicture}
    \caption{The concatenation $u\#v_i$ from Lemma~\ref{lemma: cylinder with legs}.}
    \label{im: bacteria}
\end{figure}
\begin{proof}
    For $z\in S$ let $$\text{const}_z:[-1,1]\times S^1\to S,\quad (s,t) \mapsto z$$ be the constant cylinder. Note that $w$ is homotopic to $w\#\text{const}_{w(1,0)}$ relative endings. Hence we have $$\iota(u\#v_i,w) = \iota(u\#v_i,w\#\text{const}_{w(1,0)}).$$ For dimension reasons, after an arbitrarily small perturbation of $u,v_1,v_2$, we may assume the lemma's assumptions hold, and we also have $v_i(s,t)\neq w(1,0)$ for every $s,t$ (and, in particular, $u(1,0) \neq w(1,0)$). After another perturbation of $u$ relative to its endpoints, we may assume that the graphs of $u\#v_i$ and $w\#\text{const}_{w(1,0)}$ are transverse. The intersection number $\iota(u\#v_i,w\#\text{const}_{w(1,0)})$ is now simply a count of the intersections of the graphs of the curves. Since $v_i(s,t)\neq w(1,t)$ for every $s,t$, the part associated to $v_i$ does not contribute to the count, and we have $$\iota(u\#v_1,w\#\text{const}_{w(1,0)}) = \iota(u\#v_2,w\#\text{const}_{w(1,0)}).$$
\end{proof}

To prove Proposition~\ref{lemma: selfcontinuation C^2 small}, let us first recall some facts from standard Hamiltonian Floer theory. Let $\Gamma_0 = (F_0,J_{F_0})$ and $\Gamma_1 = (F_1,J_{F_1})$ be regular continuation data connecting $H_{-\infty}$ to $H_{\infty}$ and a generic interpolation $(\Gamma_s)_{s\in [0,1]} = (F_s,J_{F_s})_{s\in [0,1]}$ of continuation data such that $\Gamma_s$ is constant near the boundary. For every $x\in \mathcal P(H^{-\infty})$ and $z\in \mathcal P(H^\infty)$ the moduli space $$\mathcal M^{\{\Gamma_s\}}(x,z) = \{(u,\lambda)\,|\, u\in \mathcal M_0^{\Gamma_\lambda}(x,z)\}$$ is a smooth 1-dimensional manifold with boundary. The boundary is given by $$\mathcal M^{\Gamma_0}_0(x,z)\cup \mathcal M^{\Gamma_1}_0(x,z).$$ It can be compactified by gluing the boundary of the form of broken trajectories $$
\Pi^{\{\Gamma_\lambda\}}(x,z) = \bigcup_{y\in \mathcal P(H^\infty)}\mathcal M_{-1}^{\{\Gamma_\lambda\}}(x,y)\times \widehat {\mathcal M}_1(y,z) \cup \bigcup_{y\in \mathcal P(H^{-\infty})}\widehat{\mathcal M}_1(x,y) \times \mathcal M_{-1}^{\{\Gamma_\lambda\}}(y,z).$$  The identification of this space as the boundary of the moduli space is done via a gluing map of the form as we discussed in Section~\ref{section.gluing} (see \cite{AD14}). 
\begin{proof}[Proof of Proposition~\ref{lemma: selfcontinuation C^2 small}]
    Let $\Gamma_0 = (H,J_H)$ be the constant continuation data and $\Gamma_1 = \Gamma$, and let $(\Gamma_s)_{s\in[0,1]}$ be a generic choice of interpolation of continuation data between $\Gamma_0$ and $\Gamma_1$ such that for any $x,z\in\mathcal P(H)$, $ \mathcal M^{\{\Gamma_s\}}(x,z)$ is a smooth manifold with boundary that consists of a disjoint union of intervals and copies of $S^1$. Let $\mathcal M^{\{\Gamma_\lambda\},\operatorname{int}}(x,z)$ denote the subset of $\mathcal M^{\{\Gamma_\lambda\}}(x,z)$ given by the union of all open intervals. We can color the boundary of the compactification of $\mathcal M^{\{\Gamma_\lambda\},\operatorname{int}}(x,z)$ with three colors by the property that the boundary component either lies in $\mathcal M^{\Gamma_0}(x,z)$, in $\mathcal M^{\Gamma_1}(x,z)$ or $\Pi^{\{\Gamma_\lambda\}}(x,z)$. Now for $x,z\in OB \mathcal P_{\beta_0}(H)$ let $$\mathcal M^{\{\Gamma_\lambda\},\operatorname{cube}}(x,z) = \{((u_1,\lambda_1),\dots,(u_{2g+2},\lambda_{2g+2}))\in \prod_{i = 1}^{2g+2}\mathcal M^{\{\Gamma_\lambda\},\operatorname{int}}(x_i,z_i)\,|\,\sum_{i\neq j}\iota(u_i,u_j) = 0\}.$$ 
    This space is a manifold given by the union of hypercubes that are products of intervals.  Let $Q_1,\dots,Q_\ell$ denote the resulting hypercubes. For every hypercube $Q_j = I_{1,j}\times \cdots \times I_{k,j}$ we color the boundary components of the intervals $I_{i,j}$ as follows \begin{align}\label{def: 3-coloring} \rho_{i,j}: \partial I_{i,j} \to \{a,b,c\},\quad p \mapsto \begin{cases}
        a, & p\in \mathcal M^{{\Gamma_0}}(x_i,z_i)\\ b , &p \in \mathcal M^{\Gamma_1}(x_i,z_i)\\ c , & p\in \Pi^{\{\Gamma_s\}}(x_i,z_i)
    \end{cases}.\end{align}
    Note that by assumption there are $2g+2$ periodic orbits and if $x\in OB\mathcal P_{\beta_0}$ then every orbit appears in the tuple $x$. Moreover, for any $z\in OB\mathcal P_{\beta_0}(H)$ there exists $\sigma\in S_{2g+2}$ such that $z = \sigma(x)$.\\
    \underline{Observation:} If $z=x$ there is exactly one $a$-pure corner point and there is none if $z \neq x$. \begin{proof}
     Since $\Gamma_0$ is the constant continuation data, we have that $\mathcal M^{\Gamma_0}(x_i,z_i)$ is empty if $x_i\neq z_i$ and consists of the constant cylinder at $x_i$ if $x_i = z_i$. If there is an $a$-pure corner point, it must consist of the $2g+2$ constant cylinders at the $2g+2$ constant periodic orbits, which forces $x_i = z_i$. On the other hand, if $x= z$, then the $2g+2$ constant cylinders must appear as a corner point since their pairwise intersection number is trivial. This proves the observation.
    \end{proof}Note that $$\Phi^{\Gamma_1}([x]) = \sum_{\sigma\in S_k}\xi^{\Gamma_1}(x,\sigma(x))[x],$$ where $\xi^{\Gamma_1}(x,z)$ is the number of $b$-pure corner points for fixed braid of orbits $x$ and $z$. We say a boundary component is $d$-colored if it is either $b$- or $c$-colored. Applying Lemma \ref{combinatorial lemma}, we can conclude that the number of $a$-pure corner points plus the number of $d$-pure corner points is even. We now study the cardinality of the latter corner points.  Let us denote $$\Pi^{\{\Gamma_\lambda\}}(x,y,z) := \mathcal M_{-1}^{\{\Gamma_\lambda\}}(x,y)\times \widehat {\mathcal M}_1(y,z) \cup \widehat{\mathcal M}_1(x,y) \times \mathcal M_{-1}^{\{\Gamma_\lambda\}}(y,z),$$ so that $$\Pi^{\{\Gamma_\lambda\}}(x,z) = \bigcup_{y\in \mathcal P(H)} \Pi^{\{\Gamma_\lambda\}}(x,y,z),$$ since in our specific case $H^\infty = H^{-\infty} = H$. For every $y\in (\mathcal P(H)\cup\{\emptyset\})^k$ we define $$\mathcal C(y) =  \left \{(w_1,\dots,w_{2g+2})\middle |\begin{aligned} &w_i\in\begin{cases}\mathcal M^{\Gamma_1}(x_i,z_i), & y_i = \emptyset\\\mathcal M_{-1}^{\{\Gamma_\lambda\}}(x_i,y_i)\times \widehat{\mathcal M}_1(y_i,z_i)\cup\widehat {\mathcal M}_1(x_i,y_i)\times \mathcal M_{-1}^{\{\Gamma_\lambda\}}(y_i,z_i),& y_i \neq \emptyset\end{cases},\\&\sum _{i \neq j}\iota(w_i,w_j) = 0\end{aligned}\right\},$$ which is the set of corner points where every component is $b$- or $c$-colored consisting of trajectories breaking at the tuple $y$ (where a component is not broken if $y_i= \emptyset$). Note that $$|\mathcal C(\emptyset,\dots,\emptyset)| = \text{number of $b$-pure corner points}.$$ \underline{Claim:} For every $y\in (\mathcal P(H)\cup \{\emptyset\})^k$ such that $y \neq (\emptyset,\dots,\emptyset)$ we have $$|\mathcal C(y)| = 0 \quad (\text{mod }2).$$ $ $\\ Assuming this claim, let us conclude the result. We first conclude that the parity of the number of $b$-pure corner points equals that of the number of $d$-pure corner points. Indeed, we have \begin{align*}
        \text{number of $d$-pure corner points} & = \sum _{y\in (\mathcal P(H)\cup \{\emptyset \})^k}|\mathcal C(y)|\\ & = |\mathcal C(\emptyset,\dots,\emptyset)| + \sum_{\substack{y\in (\mathcal P(H)\cup \{\emptyset \})^k\\ y \neq (\emptyset,\dots,\emptyset)}}|\mathcal C(y)| \\ & =\text{number of $b$-pure corner points} \quad (\text{mod }2),
    \end{align*} where we use the claim in the last line. Now, if $z = x$, we know from the observation that the number of $a$-pure corner points is odd; hence the number of $d$-pure corner points, and therefore also the number of $b$-pure corner points, is odd. It follows that $\xi^{\Gamma_1}(x,x)$ is odd. If $z \neq x$, we similarly conclude that the number of $b$-pure corner points is even, and hence $\xi^{\Gamma_1}(x,z)$ is even.  We can conclude that $\Phi^{\Gamma_1}([x]) = [x]$. \\ Let us now prove the claim. After relabeling, we may and do assume that $y_1 \neq \emptyset.$ If $\mathcal C(y)$ is empty there is nothing to show. Otherwise, there exists $w_1 = (u_1,v_1) $ where $u_1\in \mathcal M_{-1}^{\{\Gamma_\lambda\}}(x_1,y_1)\times \widehat{\mathcal M}_1(y_1,z_1)$ or $w_1 = (v_1,u_1)\in \widehat {\mathcal M}_1(x_1,y_1)\times \mathcal M_{-1}^{\{\Gamma_\lambda\}}(y_1,z_1).$ We will assume that the first holds; the latter case is analogous. By assumption the cardinality of  $\widehat{\mathcal M}_1(y_1,z_1)$ is even. Hence, the set admits a free $\mathbb Z_2$-action $$s:  \widehat{\mathcal M}_1(y_1,z_1)\times \mathbb Z_2\to  \widehat{\mathcal M}_1(y_1,z_1).$$ This induces  the following $\mathbb Z_2$-action on $\mathcal C(y)$ $$\bar s: \mathcal C(y)\times \mathbb Z_2\to \mathcal C(y), \quad ((u_1,v_1),w_2,\dots,w_{2g+2})\mapsto((u_1,s(v_1)),w_2,\dots,w_{2g+2}).$$ By Lemma \ref{lemma: cylinder with legs} we have that $$0=\iota(u_1\#v_1,w_i) = \iota(u_1\#s(v_1),w_i)$$ and hence $((u_1,s(v_1)),w_2,\dots,w_{2g+2})$ is an element of $\mathcal C(y)$. Since $s$ is free, the action $\bar s$ is free, and hence every orbit of $\bar s$ consists of two elements. Thus $|\mathcal C(y)|$ is even, which concludes the proof of the claim and the proposition.
\end{proof}
\begin{lemma}\label{lemma: existence of good Hamiltonian}
    There is a non-degenerate Hamiltonian $H$ satisfying the conditions of Proposition \ref{lemma: selfcontinuation C^2 small}
\end{lemma}
\begin{proof}
Choose a Morse function $H$ on $S$ such that there are exactly $2g+2$ critical points. By considering $H/n$ for $n$ large enough, we may assume that the periodic orbits $\mathcal P(H)$ are non-degenerate, constant, and correspond to the critical points of $H$. For a generic almost complex structure, we may further assume that any $u\in \mathcal M_1(x,y)$ satisfies $u(s,t) = u(s,0)$ for all $(s,t)\in \mathbb R\times S^1$ (see \cite{AD14}). Now, since the standard Hamiltonian Floer complex and homology with coefficients in $\mathbb Z_2$ satisfy $$\dim CF(H,J) = \dim HF(H,J)$$ we can conclude that the differential must be trivial and the cardinality of moduli spaces $\widehat {\mathcal M}_1(x,y)$ must be divisible by $2$. 
\end{proof}
We can now prove Theorem~\ref{thm: braided floer homology}. \begin{proof}[Proof of Theorem~\ref{thm: braided floer homology}]
By Lemma~\ref{lemma: existence of good Hamiltonian}, there is a non-degenerate Hamiltonian $H$ and almost complex structure $J_H$ that satisfy the assumptions of Proposition~\ref{lemma: selfcontinuation C^2 small}. For every non-degenerate Hamiltonian diffeomorphism $\phi$, let $G$ be a generating Hamiltonian and $J_G$ a generic choice of almost complex structure. Choose continuation data $\Gamma'$ from $(H,J_H)$ to $(G,J_G)$ and $\Gamma''$ from $(G,J_G)$ to $(H,J_H)$. Let $\beta_0$ be the trivial braid with $2g+2$ strands. Then by Proposition~\ref{proposition: continuation map composition} we have $$\Phi^{\Gamma''}\circ \Phi^{\Gamma'} = \Phi^{\Gamma''\#\Gamma'}$$ since any periodic orbit of $G$ is non-degenerate. Here, the continuation maps relate the braided Floer complexes in the braid class $\beta_0$. Since $\Gamma''\#\Gamma'$ is self continuation data at $(H,J_H)$ by Proposition~\ref{lemma: selfcontinuation C^2 small} we have that $$\Phi^{\Gamma''\#\Gamma'} = id_{\operatorname{BCF}_{\beta_0}}(H,J).$$ By Proposition~\ref{continuation are chain maps}, we have that the continuation maps are chain maps and hence $\Phi^{\Gamma''\#\Gamma'}$ induces the identity in homology. In particular, we have that $$\Phi^{\Gamma'}:\operatorname{BHF}_{\beta_0}(H,J_H)\hookrightarrow \operatorname{BHF}_{\beta_0}(G,J_G)$$ is injective and hence, $\dim \operatorname{BHF}_{\beta_0}(G,J_G)\geq 1$.
\end{proof}

\section{Braid Stability and Bounded Continuation Maps} \label{section braid stability}
A \emph{framed braid} is a braid $b= \{s_1,\dots,s_k\}$ together with a framing $$f:\{s_1,\dots,s_k\} \to \mathbb Z \quad \text{or} \quad f:\{s_1,\dots,s_k\} \to \mathbb Z_4.$$ We say two framed braids are isotopic if they are isotopic as braids via an isotopy that respects the framing. For a framed closed surface braid class $(\beta,f)$ we define the \emph{framed braided Floer space of class $(\beta,f)$}  of a non-degenerate Hamiltonian $H$ as $$\operatorname{BCF}_{\beta,f}(H) = \left \langle x = \{x_1,\cdots,x_k \}\subset \mathcal P(H)\,\big | \,\mathcal B(H,x) = \beta, f_x = f\right \rangle_{\mathbb Z_2},$$ where $f_x: \{x_1,\dots,x_k\}\to \mathbb Z$ when the genus $g\geq 1$ and $f_x: \{x_1,\dots,x_k\}\to \mathbb Z_4$ when the genus $g =0$ is given by $f_x(x_i) = \mu(x_i)$.
In this section, we prove the braid stability statement from Theorem~\ref{thm: braid stability}: For every non-degenerate Hamiltonian $H$ and $x\subset \mathcal P(H)$ we define a quantity $\mathcal S(H,x)\in \mathbb R\cup \{\infty\}$ such that for every non-degenerate Hamiltonian $G$ satisfying $E(H-G)<\mathcal S(H,x)$ we have $$\dim \operatorname{BCF}_{\beta,f}(G)\geq 1,$$ where $\beta = \mathcal B(H,x)$, $f = f_x$, and $k$ is the number of strands. We prove this by using \emph{$\varepsilon$-bounded continuation map} $$\Phi_\varepsilon:\operatorname{BCF}_{\beta,f}(H)\to \operatorname{BCF}_{\beta,f}(G)$$ for $\varepsilon\geq 0$ which counts Floer-theoretic braid isotopies consisting of cylinders $u_1,\dots,u_k$ each connecting orbits in a braid of class $(\beta,f)$ of $H$ and $G$ whose action difference is absolutely bounded by $\varepsilon$.
\subsection{Bounded Continuation Maps}
For two Hamiltonians $H,G: S^1\times S\to \mathbb R$ and a cylinder $u:\mathbb R\times S^1\to S$ such that $\lim_{s\to -\infty}u(s,t) = x(t)$ and $\lim_{s\to\infty}u(s,t) = y(t)$ we will consider the action difference of two capped periodic orbits $$ \mathcal A_H([x,C])-\mathcal A_G([y,C\#u]),$$ where $C$ is a capping of $x$, that is, a cylinder from $x$ to a fixed reference loop in its homotopy class. Note that this expression is independent of the choice of capping. 
Moreover, for a smooth family of $\omega$-compatible almost complex structures $(J_{s,t})_{(s,t)\in \mathbb R \times S^1}$ we define the \emph{energy of the cylinder $u$} by $$E(u) = \int_\mathbb R \int _0^1 \omega(\partial_su,J_{s,t}\partial_su)\,dtds.$$  In the case when $(J_{s,t}) = (J_t)$ is independent of $s$ and $u$ is a Floer-cylinder solving $$\partial_su+J_t(\partial_tu-X_{H_t}\circ u)=0$$ a standard calculation shows that $E(u)$ equals the action difference of the limiting orbits $$E(u) = \mathcal A_H([x,C_{x_i}])-\mathcal A([y,C_{x_i}\#u])$$ which, in particular, is non-negative. Let us introduce a measure of the energy of the continuation data $\Gamma = (H_s,J_s)$. That is, let $$E^+(\Gamma) = \int_{-\infty}^\infty\int_0^1\max_{x\in S}\partial_{s}H(s,t,x)\,dtds$$ and note that we have $$E^+(\Gamma''\#_\rho\Gamma')= E^+(\Gamma'')+E^+(\Gamma')$$ for $\rho>0$ large enough.
 For every $\varepsilon\geq 0$ we refine the moduli space $\mathcal M^{\Gamma}_\ell(x,y)$ and define $$\mathcal M^{\Gamma}_{\ell,\leq \varepsilon}(x,y) = \left\{u\in \mathcal M^{\Gamma}_\ell(x,y)\,\middle | \, \left |\sum_{i =1}^k\, \mathcal A_{H}([x_i,C])-\mathcal A_{G}([y_i,C\#u_i])\right |\leq \varepsilon\right\},$$ where $C_{x_i}$ is an arbitrary choice of capping of $x_i$. Since the expression $\mathcal A_{H}([x_i,C_{x_i}])-\mathcal A_{G}([y_i,C_{x_i}\#u_i])$ is invariant under homotopy of $u_i$ relative limiting orbits, it is locally constant on $\mathcal M^{\Gamma}_\ell(x,y)$. Hence, $\mathcal M^{\Gamma}_{\ell,\leq \varepsilon}(x,y)$ is the union of connected components for which any element $u$ and any component of it satisfies $\left |\sum_{i = 1}^k\mathcal A_{H}([x_i,C])-\mathcal A_{G}([y_i,C\#u_i])\right | \leq \varepsilon$. A standard computation (which is also done in the proof of Proposition \ref{Proposition: bounded continuation maps identity} in more generality) shows that for all $u_i\in \mathcal M_\ell^{\Gamma}(x_i,z_i)$ $$-E^+(\Gamma)\leq \mathcal A_H([x_i,C_{x_i}])-\mathcal A_G([z_i,C_{x_i}\#u_i])$$ and $$E(u_i)\leq \mathcal A_H([x_i,C_{x_i}])-\mathcal A_G([z_i,C_{x_i}\#u_i])+E^+(\Gamma).$$ It follows that there is the following energy bound $$E(u_i) \leq kE^+(\Gamma)+\varepsilon.$$
With this energy bound, as in standard Hamiltonian Floer theory, the space $\mathcal M^{\Gamma}_{0,\leq \varepsilon}(x,y)$ is compact and therefore consists of a finite number of points. Importantly, in contrast to Section~\ref{section: braided floer homology}, we here also consider the case of non-contractible orbits on $T^2$. Compactness of the moduli spaces in these cases is guaranteed by the  predescribed action bound. This leads us to the definition of $\varepsilon$-\emph{bounded continuation maps of braided Floer spaces in the class} $(\beta,f)$. 
Namely, regular pairs $(H,J_H)$ and $(G,J_G)$ and regular continuation data $\Gamma$ connecting them we define $$\Phi^\Gamma_\varepsilon: \operatorname{BCF}_{\beta,f}(H) \to \operatorname{BCF}_{\beta,f}(G), \quad [x]\to \sum_{y\in OB\mathcal P_{\beta,f}(G)}|\mathcal M^{\Gamma}_{0,\leq \varepsilon}(x,y)|[y],$$ where $[\cdot]$ stands for the equivalence class of unordered braids. For better readability, let us from now on write $$\xi^\Gamma_\varepsilon(x,y) = |\mathcal M^\Gamma_{0,\leq \varepsilon}(x,y)|.$$
\begin{lemma}
    The map $\Phi^{\Gamma}_\varepsilon$ is well-defined.
\end{lemma}
\begin{proof} 
We need to check that the map does not depend on the representative $x$ of the class $[x]$. Indeed, if $\sigma\in S_k$ is a permutation of $k$ objects, we have $$\sum _{y\in OB\mathcal P_{\beta,f}(G)} \xi_\varepsilon^\Gamma(x,y)[y] = \sum _{\sigma(y)\in OB\mathcal P_{\beta,f}(G)} \xi_\varepsilon^\Gamma(\sigma(x),\sigma(y))[\sigma(y)] = \sum _{y\in OB\mathcal P_{\beta,f}(G)} \xi_\varepsilon^\Gamma(\sigma(x),y)[y],$$ where $\sigma(z) = (z_{\sigma(1)},\dots,z_{\sigma(k)})$.
\end{proof}
\subsection{The Quantity $\mathcal S(H,x)$}
We now define the quantity $S(H,x)$ for a non-degenerate $H$ and $x\subset \mathcal P(H)$. Let $\beta$ be a fixed closed surface braid class. For a smooth family of $\omega$-compatible almost complex structures $(J_t)_{t\in S^1}$ and $x,y\in OB\mathcal P_\beta(H)$ we define $$\mathcal M(x,y) = \{u\in \mathcal M(x_1,y_1)\times \cdots \times \mathcal M(x_k,y_k)\, |\, \iota(u) = 0\}$$  the space of \emph{Floer isotopies}. The energy of a Floer isotopy is defined as the sum of the energy of the Floer cylinders $$E(u) = \sum_{i = 1}^k E(u_i),$$ and we write $$\mu_{\operatorname{max}}(u) = \max_i\mu(u_i).$$ The stability radius is defined as $$\mathcal S(H,x) =  \frac 1 k\sup_J\min \{\mathcal E(H,J,x),\mathcal E^{\operatorname{int}}(H,J,x),\mathcal E^2(H,J,x)\},$$ where the supremum runs over all $J$ such that $(H,J)$ is regular and $k$ is the number of strands $ k = |x|$. The quantities $\mathcal E(H,J,x)$, $\mathcal E^{\operatorname{int}}(H,J,x)$, and  $\mathcal E^2(H,J,x)$ are the minimal energy $E(u)$ of Floer cylinders in certain constellations of Floer cylinders.  
First, the quantity $\mathcal E(H,J,x)$ is the minimal energy of Floer isotopies $u\in \mathcal M(x,y)$ or $u\in \mathcal M(y,x)$ for $y\in  OB\mathcal P_\beta(H)$ with $\mu_{\operatorname{max}}(u)  = 1$ $$\mathcal E(H,J,x) =\min \{E(u)\,|\, y\in OB\mathcal P_\beta(H), u \in \mathcal M(x,y)\cup \mathcal M(y,x),\mu_{\operatorname{max}}(u) = 1\}.$$ This is the most relevant quantity of the three appearing in the definition of $\mathcal S(H,x)$, as the other two can be neglected when choosing the right framing. Namely, the quantity $\mathcal E^{\operatorname{int}}(H,J,x)$ is the minimal energy of an index $1$ cylinder connecting orbits $x_i$ to $x_j$ $$\mathcal E^{\operatorname{int}}(H,J,x) = \min \{E(u)\,|\, u \in \mathcal M(x_i,x_j),\,\mu(u ) = 1\} $$ which is by convention equal to $\infty$ when no such cylinders exist. The quantity $\mathcal E^2(H,J,x)$ is the most nuanced of the three and accounts for the configuration (see Fig.~\ref{fig: Configurations for E2}) of pairs of Floer trajectories that start (or end) in distinct orbits in a single braid of class $\beta$ and end (or start) in a single orbit $y\in \mathcal P(H)$ (warning: $y$ need not be a strand of a braid of periodic orbits of type $(\beta,f)$). 
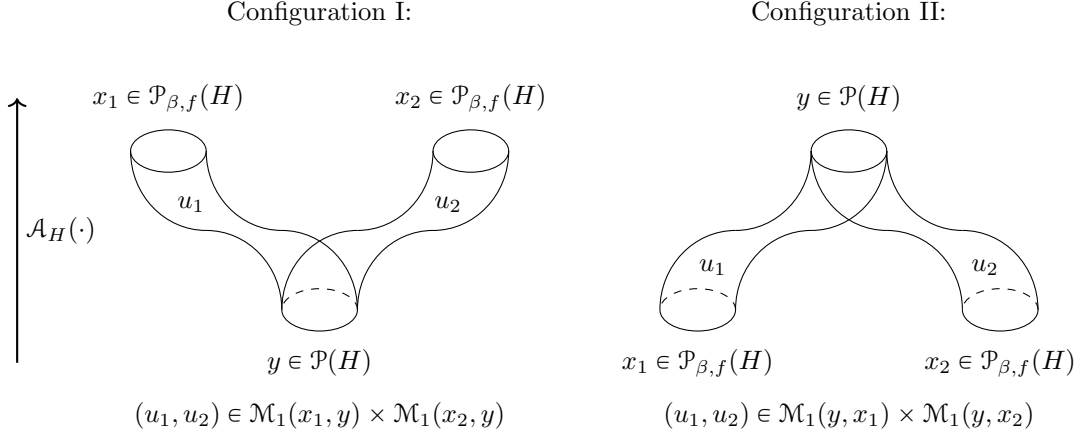
\begin{figure}[h!]
    \begin{tikzpicture}[yscale = 1.4]
    \draw[thick, ->] (-11,-2)--(-11,0.5) node[midway, right]{$\mathcal A_H(\cdot)$};
    \begin{scope}[shift = {(-7,0)}]
        \node[] at (0,1.3) {Configuration I:};
        \node[] at (-1.7,-0.5){$u_1$};
        \node[] at (1.7,-0.5){$u_2$};
        \draw[] (-2,0) ellipse (0.5 and 0.2);
        \node[] at (-2,0.5) {$x_1\in \mathcal P_{\beta,f}(H)$};
        \draw[] (2,0) ellipse (0.5 and 0.2);
        \node[] at (2,0.5) {$x_2\in \mathcal P_{\beta,f}(H)$};;
        \draw[] (0.5,-1.5) arc (0:-180:0.5 and 0.2);
        \draw[dashed] (0.5,-1.5) arc (0:-180:0.5 and -0.2);
        \node[] at (0,-2) {$y\in \mathcal P(H)$};
        \draw[] (-2.5,0) arc (0:90:-1 and -0.75);
        \draw[] (-0.5,-1.5) arc (0:90:1 and 0.75);
        \draw[] (-1.5,0) arc (0:90:-1 and -0.75);
        \draw[] (0.5,-1.5) arc (0:90:1 and 0.75);
        \draw[] (2.5,0) arc (0:90:1 and -0.75);
        \draw[] (0.5,-1.5) arc (0:90:-1 and 0.75);
        \draw[] (1.5,0) arc (0:90:1 and -0.75);
        \draw[] (-0.5,-1.5) arc (0:90:-1 and 0.75);
        \node[] at (0,-2.5) {$(u_1,u_2) \in \mathcal M_1(x_1,y)\times \mathcal M_1(x_2,y)$};
        \end{scope}
        \begin{scope}
        \node[] at (-1.8,-1.1){$u_1$};
        \node[] at (1.8,-1.1){$u_2$};
        \node[] at (0,1.3) {Configuration II:};
        \draw[] (-1.5,-1.5) arc (0:-180:0.5 and 0.2);
        \draw[dashed] (-1.5,-1.5) arc (0:-180:0.5 and -0.2);
        \node[] at (-2,-2) {$x_1\in \mathcal P_{\beta,f}(H)$};
        \draw[] (2.5,-1.5) arc (0:-180:0.5 and 0.2);
        \draw[dashed] (2.5,-1.5) arc (0:-180:0.5 and -0.2);
        \node[] at (2,-2) {$x_2\in \mathcal P_{\beta,f}(H)$};;
        \draw[] (0,0) ellipse (0.5 and 0.2);
        \node[] at (0,0.5) {$y\in \mathcal P(H)$};
        \draw[] (-2.5,-1.5) arc (0:90:-1 and 0.75);
        \draw[] (-0.5,0) arc (0:90:1 and -0.75);
        \draw[] (-1.5,-1.5) arc (0:90:-1 and 0.75);
        \draw[] (0.5,0) arc (0:90:1 and -0.75);
        \draw[] (2.5,-1.5) arc (0:90:1 and 0.75);
        \draw[] (0.5,0) arc (0:90:-1 and -0.75);
        \draw[] (1.5,-1.5) arc (0:90:1 and 0.75);
        \draw[] (-0.5,0) arc (0:90:-1 and -0.75);
          \node[] at (0,-2.5) {$(u_1,u_2) \in \mathcal M_1(y,x_1)\times \mathcal M_1(y,x_2)$};
        \end{scope}
    \end{tikzpicture}
    \caption{Configurations of Floer cylinders relevant for $\mathcal E^2(H,J,x_1,x_2)$.}
    \label{fig: Configurations for E2}
\end{figure} We define $$\mathcal E^2(H,J,x) = \inf\left \{\max\{E(u_1),E(u_2)\}\,\middle |\, \begin{aligned}&y\in \mathcal P(H), \,  1\leq i < j\leq k, \\ &(u_1,u_2) \in \mathcal M_1(x_i,y)\times \mathcal M_1(x_j,y)\cup \mathcal M_1(y,x_i)\times \mathcal M_1(y,x_j)\end{aligned}\right \}$$ which considers constellations as in Figure~\ref{fig: Configurations for E2}. Note that when $\mu(x_i) \neq \mu(x_j)$ for $i \neq j$, then the infimum runs over the empty set and $\mathcal E^2(H,J,x) = \infty$ by convention. Hence in the case where $f_x$ is injective: $$ \mathcal S(H,x) = \frac 1 k\sup_J\mathcal \min \{\mathcal E(H,J,x), \mathcal E^{\operatorname{int}}(H,J,x)\}$$ as promised in the Introduction \ref{section: intro}.
 For every $x\in OB\mathcal P_{\beta}(H)$ with framing $f_x = f$ let $$\pi_{[x]}:\operatorname{BCF}_{\beta,f}(H) \to \langle [x]\rangle_{\mathbb Z_2}$$ be the projection to the factor $[x]$. We prove the following result, which is essential for braid stability.
\begin{prop}\label{Proposition: bounded continuation maps identity}
Let $(H,J)$ be a regular pair and $\Gamma$ continuation data connecting $(H,J)$ to itself. Let $\varepsilon\geq 0$ and suppose there exists $x\in OB\mathcal P_\beta(H)$ with $f = f_x$ and such that $0\leq E^+(\Gamma)$ and $$kE^+(\Gamma)+\varepsilon<\min\{\mathcal E(H,J,x),\mathcal E^{\operatorname{int}}(H,J,x),\mathcal E^2(H,J,x)\},$$ where $k$ is the number of strands. Then we have $$\pi_{[x]}\circ \Phi^\Gamma_\varepsilon([x]) =[x].$$
\end{prop}
The proof shows the following. When $kE^+(\Gamma)+\varepsilon<\min\{\mathcal E(H,J,x),\mathcal E^{\operatorname{int}}(H,J,x),\mathcal E^2(H,J,x)\}$ holds for \emph{all} $x\subset \mathcal P(H)$ such that $\beta = \mathcal B(H,x)$ and $f_x = f$ then $$\Phi_\varepsilon^{\Gamma} = id_{\operatorname{BCF}_{\beta,f}(H)}.$$
 Before proving this proposition, we first do some preparation. First, recall some standard facts from Hamiltonian Floer theory. Given regular pairs $(H,J_H)$ and $(G,J_G)$ and continuation data $\Gamma_0$ and $\Gamma_1$ connecting them, one can choose a generic smooth interpolation of continuation data $(\Gamma_\lambda)_{\lambda\in [0,1]}$ where $\Gamma_\lambda=\Gamma_0$ for $\lambda$ close to 0 and $\Gamma_\lambda =\Gamma_1$ for $\lambda$ close to $1$. For two $x\in \mathcal P(H)$ and $z\in \mathcal P(G)$  and $\ell\in \mathbb Z$ the space $$\mathcal M_\ell^{\{\Gamma_\lambda\}}(x,z) = \{(\lambda,u)\,|\,\lambda\in [0,1], u \in \mathcal M_\ell^{\Gamma_ \lambda}(x,z)\}$$ is a smooth finite dimensional manifold with boundary of dimension $\ell+1$. Importantly, if $\ell = -1$, the manifold is compact and consists of finitely many points. When $\ell = 0$, the possibly non-compact manifold is 1-dimensional with boundary given by $$\mathcal M_0^{\Gamma_0}(x,z)\cup \mathcal M_0^{\Gamma_1}(x,z).$$ 
To compactify this space, one adds a boundary of the form of broken trajectories. Namely, consider $$\Pi^{\{\Gamma_\lambda\}}(x,z) = \bigcup_{\substack{y\in \mathcal P(H)}}\widehat{ \mathcal M}_1(x,y)\times \mathcal M^{\{\Gamma_\lambda\}}_{-1}(y,z)\cup \bigcup_{\substack{y\in \mathcal P(H)}}\mathcal M_{-1}^{\{\Gamma_\lambda\}}(x,y)\times \widehat{ \mathcal M}_1(y,z).$$
Any broken trajectory $(u,v)\in \Pi^{\{\Gamma_\lambda\}}(x,z)$ is identified as the boundary of $\mathcal M^{\{\Gamma_\lambda\}}(x,z)$ by a gluing process that can be performed with a method as described in Section~\ref{section.gluing}. For this brief recapitulation from standard Floer theory, see \cite{AD14}.
Now, let us return to braided orbits. For pairs $(H,J_H)$ and $(G,J_G)$ and $x\in OB\mathcal P_{\beta,f}(H),y\in OB\mathcal P_{\beta,f}(G)$ we define $$\mathcal M_{0,\leq \varepsilon}^{\{\Gamma_\lambda\}}(x,z) = \left\{((u_1,\lambda_1),\dots,(u_k,\lambda_k))\,\middle|\, \begin{aligned}&(u_i,\lambda_i)\in \mathcal M^{\{\Gamma_\lambda\}}_0(x_i,z_i),\quad
\sum_{i\neq j}\iota(u_i,u_j) = 0,\\& \left |\sum_{i = 1}^k\mathcal A_H([x_i,C_{x_i}])-\mathcal A_G([z_i,C_{x_i}\#u_i])\right|\leq \varepsilon\end{aligned}\right\}.$$ 
Any path in $\mathcal M_{0}^{\{\Gamma_\lambda\}}(x_1,y_1)\times\cdots\times\mathcal M^{\{\Gamma_\lambda\}}_{0}(x_k,y_k)$ gives rise to a smooth homotopy of the cylinders $u_1,\dots,u_k$ relative limiting orbits. Since the intersection number is invariant under homotopy, the space $\mathcal M_{0,\leq \varepsilon}^{\{\Gamma_\lambda\}}(x,z)$ is given by the connected component of $\mathcal M_{0}^{\{\Gamma_\lambda\}}(x_1,y_1)\times\cdots\times\mathcal M_{0}^{\{\Gamma_\lambda\}}(x_k,y_k)$ on which the intersection number evaluates to $0$ and we have $$\left |\sum_{i = 1}^k\mathcal A_H([x_i,C_{x_i}])-\mathcal A_G([z_i,C_{x_i}\#u_i])\right |\leq \varepsilon.$$ These connected components are products of $1$-dimensional manifolds. \\

To prove Proposition~\ref{Proposition: bounded continuation maps identity} we first choose a suitable homotopy of continuation data from $\Gamma_1 :=\Gamma$ to the constant one $\Gamma_0$ so that $\Phi^{\Gamma_0}_\varepsilon = id_{BCF_\beta(H)}$. The goal is to show that $$\pi_{[x]}\circ (\Phi^{\Gamma_1}_\varepsilon+\Phi^{\Gamma_0}_\varepsilon)([x])=0.$$ For every $z\in OB\mathcal P_{\beta,f}(H)$ with $[z] = [x]$ (which means $z = \sigma(x)$) any $u\in \mathcal M^{\Gamma_0}_{0,\leq \varepsilon}(x,z)\cup \mathcal M^{\Gamma_1}_{0,\leq \varepsilon}(x,z)$ arises as the corner point of a hypercube which is a connected component of $\mathcal M_{0,\leq \varepsilon}^{\{\Gamma_\lambda\}}(x,z)$. We show that under the conditions in Proposition~\ref{Proposition: bounded continuation maps identity}, the connected components with such corner points are compact and hence are hypercubes, that is, products of closed intervals in $\mathcal M_{0,\leq \varepsilon}^{\{\Gamma_\lambda\}}(x_i,z_i)$. Corner points of these hypercubes are elements in $$(\mathcal M_{0,\leq \varepsilon}^{\Gamma_0}(x_1,z_1)\cup \mathcal M_{0,\leq \varepsilon}^{\Gamma_1}(x_1,z_1))\times\cdots\times (\mathcal M_{0,\leq \varepsilon}^{\Gamma_0}(x_k,z_k)\cup \mathcal M_{0,\leq \varepsilon}^{\Gamma_1}(x_k,z_k)).$$ 
Such a corner point is called \emph{pure} if it lies in $$\mathcal M_{0,\leq \varepsilon}^{\Gamma_0}(x_1,z_1)\times \cdots\times \mathcal M_{0,\leq \varepsilon}^{\Gamma_0}(x_k,z_k)\cup \mathcal M^{\Gamma_1}_{0,\leq \varepsilon}(x_1,z_1)\times\cdots\times \mathcal M^{\Gamma_1}_{0,\leq \varepsilon}(x_k,z_k)$$ and they are called \emph{mixed} otherwise. The value of $\Phi_{\varepsilon}^{\Gamma_0}([x])$ and $\Phi_{\varepsilon}^{\Gamma_1}([x])$ is given by counting the pure corner points hypercubes in $\mathcal M_{0,\leq \varepsilon}^{\{\Gamma_\lambda\}}(x,z)$. To show $\pi_{[x]}\circ (\Phi^{\Gamma_1}_\varepsilon+\Phi^{\Gamma_0}_\varepsilon)([x])=0$, we use the fact that the number of pure corner points is even, as stated in Lemma~\ref{combinatorial lemma}. 
 For a compact hypercube $Q$ in $\mathcal M_{0,\leq \varepsilon}^{\{\Gamma_\lambda\}}(x,z)$, we have that $$Q = I_1\times \cdots \times I_k,$$ where $I_i \subset \mathcal M^{\{\Gamma_\lambda\}}_{0,\leq \varepsilon}(x_i,z_i)$ is a closed interval. The boundary components of these intervals come in two colors $$\text{Type 0: }(u_i,0) \text{ where } u_i \in \mathcal M^{\Gamma_0}_{0,\leq \varepsilon}(x_i,z_i), $$ $$\text{Type 1: }(u_i,1) \text{ where }u_i \in \mathcal M^{\Gamma_1}_{0,\leq \varepsilon}(x_i,z_i). $$ Applying Lemma~\ref{combinatorial lemma} to the induced coloring, we find that the number of pure corner points is even. Let us now prove Proposition~\ref{Proposition: bounded continuation maps identity}.
\begin{proof}[Proof of Proposition~\ref{Proposition: bounded continuation maps identity}] Let $z\in OB\mathcal P_\beta(H)$ with $[z] = [x]$.
    Let us write $\Gamma_1 = \Gamma = (F_1,J_1)$ and let us define a homotopy to the constant continuation data $\Gamma_0 =(H,J) = (F_0,J_0)$ given by $\Gamma_\lambda= (F_\lambda,J_\lambda)$ for $\lambda\in [0,1]$ where $$F_\lambda(s,t,x) = H(t,x) + f(\lambda)(F_1(s,t,x)-H(t,x))$$ and $f:\mathbb R\to [0,1]$ is a smooth increasing function such that $f(s) = 0$ for $s<{1\over 3}$ and $f(s) = 1$ for $s>{2\over 3}$. After a choice of almost complex structure and a $C^\infty$-small compactly supported perturbation, we can assume that $\Gamma_1$ and $\{\Gamma_\lambda\}_{[0,1]}$ are regular. However, for the following energy estimates, we can neglect this perturbation. A standard computation shows that for every $\lambda\in [0,1]$ and solution $w\in \mathcal M_\ell^{\Gamma_\lambda}(x,z)$ we have that $$\mathcal A_H([z,C\#w])-\mathcal A_{H}([x,C])\leq E^+(\Gamma).$$  Indeed we have \begin{align} \label{computation: action estimate}
        \begin{split}
            0 & \leq \int_{-\infty}^{\infty}\int _0^1\omega(\partial_sw,J_{\lambda}\partial_sw)\,dtds \\ & =\int _{-\infty}^{\infty}\int_0^1 \omega(\partial_sw,\partial_tw-X_{F_{\lambda}}\circ w)\,dtds\\ 
            & = \int_{-\infty}^{\infty}\int_0^1\omega(\partial_sw,\partial_t w)+ \omega(X_{F_{\lambda}}\circ w,\partial_sw)\,dtds\\ & = \int_{-\infty}^{\infty}\int_0^1\omega(\partial_sw,\partial_t w)\,dtds - \int _{-\infty}^{\infty}\int _0^1dF_{\lambda}\partial_sw\,dt ds.
        \end{split}
    \end{align} and \begin{align*}
        \begin{split}
            \int _{-\infty}^{\infty}\int _0^1dF_{\lambda}\partial_sw\,dt ds& = \int_{-\infty} ^{\infty}\int _0^1\partial_s(F_{\lambda}\circ w)\,dt ds-\int _{-\infty}^{\infty}\int_0^1\partial_s F_\lambda \circ w\,dt ds  \\ & = \int_0^1H(t,z(t))\,dt- \int_0^1 H(t, x(t))dt-\int_{-\infty}^{\infty}\int_0^1f(\lambda)\partial_sF_1(s,t,w(s,t))dtds \\ 
            &\geq \int_0^1H(t,z(t))\,dt-\int_0^1H(t,x(t))\,dt -f(\lambda)E^+(\Gamma).
        \end{split}
    \end{align*} 
    In total, we have \begin{align*}\begin{split}
       \mathcal A_H([z,C\#w])-\mathcal A_{H}([x,C]) &= -\int_{-\infty}^\infty\int_0^1\omega(\partial_sw,\partial_tw)\,dtds-\int_0^1H(t,x(t))-H(t,z(t))\,dt\\ &\leq f(\lambda)E^+(\Gamma)\leq E^+(\Gamma).
   \end{split} \end{align*}
    A broken trajectory $(u,v)\in \Pi^{\{\Gamma_\lambda\}}(x,z)$ can come in two forms, either $u\in \widehat{ \mathcal M}(x,y) $ or $v\in  \widehat{ \mathcal M}(y,z)$. Let us consider the first case; the second case is analogous. Abusing the notation, let $u$ again denote the representative in $\mathcal M(x,y)$. 
    The existence of a broken trajectory ensures that there is a sequence $(\lambda_n)$ such that $\lambda_n\to \lambda_*$, a sequence $(s_n)$ such that $s_n \to \infty$, and a sequence $w_n\in \mathcal M^{\lambda_n}(x,z)$ such that $w_n(s-s_n,t)$ converges to $u$ and $w_n(s+s_n,t)$ converges to $v$ in $C^\infty_{\text{loc}}$-topology.
    For every such sequence $w_n$ we have that for $n$ large enough $w_n$ is homotopic to $u\#v$ relative to limiting orbits. We find \begin{align*}\begin{split}E(u) &=\mathcal A_{H}([x,C] )-\mathcal A_H([y,C\#u])\\ &= \big (\mathcal A_{H}([x,C] )-\mathcal A_H([z,C\#w_n])\big )-\big(\mathcal A_H([y,C\#u])-\mathcal A_H([z,C\#u\#v])\big)\\ &\leq \mathcal A_{H}([x,C] )-\mathcal A_H([z,C\#w_n])+E^+(\Gamma).\end{split}\end{align*}
    This inequality will be the tool to rule out the breaking of orbits.

   The goal is to show that $\Phi^{\Gamma_0}_\varepsilon([x])+\Phi^{\Gamma_1}_\varepsilon([x]) = 0$ which concludes the result since $\Phi^{\Gamma_0}_\varepsilon$ is the identity. We do this by first observing that $\Phi^{\Gamma_0}_\varepsilon([x])+\Phi^{\Gamma_1}_\varepsilon([x])$ counts certain corner points of certain hypercubes in the manifold $\mathcal M_{0,\leq \varepsilon}^{\{\Gamma_\lambda\}}(x,z)$ where $z\in OB\mathcal P_{\beta,f}(H)$ with $[z] = [x]$. We then show that this count is even.\\ 
   The space $\mathcal M_{0,\leq \varepsilon}^{\{\Gamma_\lambda\}}(x,z)$ is a product of 1-dimensional moduli spaces, all of which can be compactified by gluing broken trajectories. The compactification of these moduli spaces induces a compactification of $\mathcal M_{0,\leq \varepsilon}^{\{\Gamma_\lambda\}}(x,z)$, given by a disjoint union of products of compact $1$-dimensional manifolds. We consider the subset given by the union of all hypercubes, which are given by a product of intervals, and denote this subspace by $$\mathcal M_{0,\leq \varepsilon}^{\{\Gamma_\lambda\},\mathrm{cube}}(x,z) \subset \mathcal M_{0,\leq \varepsilon}^{\{\Gamma_\lambda\}}(x,z).$$ 
   The $0$-dimensional corner points of such a hypercube are elements of $$\prod_{i =1}^k \mathcal M_{0,\leq \varepsilon}^{\Gamma_0}(x_i,z_i)\cup\mathcal M_{0,\leq \varepsilon}^{\Gamma_1}(x_i,z_i)\cup \Pi^{\{\Gamma_\lambda\}}(x_i,z_i).$$ Note that all elements in $$\prod_{i =1}^k \mathcal M_{0,\leq \varepsilon}^{\Gamma_0}(x_i,z_i) \cup \prod_{i =1}^k \mathcal M_{0,\leq \varepsilon}^{\Gamma_1}(x_i,z_i)$$ appear as corner points of a hypercube in $\mathcal M^{\{\Gamma_\lambda\},\mathrm{cube}}_{0,\leq \varepsilon}(x,z)$. A corner point is called \emph{pure} if all components either lie in $\mathcal M_{0,\leq \varepsilon}^{\Gamma_0}(x_i,z_i)$ (in which case we say it is $0$-pure), all in $\mathcal M_{0,\leq \varepsilon}^{\Gamma_1}(x_i,z_i)$ (in which case we say $1$-pure), or all in $\Pi^{\{\Gamma_\lambda\}}(x_i,z_i)$ (in which case we say purely broken).  To show that $\Phi_\varepsilon^{\Gamma_0}+\Phi_\varepsilon^{\Gamma_1}([x]) = 0$ it suffices to show that the number of 0- and 1-pure corner points of hypercubes in $\mathcal M^{\{\Gamma_\lambda\}}_{0,\leq \varepsilon}(x,z)$ is even. To do so we show that if a hypercube $Q\subset \mathcal M^{\{\Gamma_\lambda\},\mathrm{cube}}_{0,\leq \varepsilon}(x,z)$ has a $0$- or $1$-pure corner points, then it trivially has an even number of $0$- or $1$-pure corner points, or there no component of a corner point is a broken trajectory, that is all corner points lie in  $$\prod_{i = 1}^k\mathcal M^{\Gamma_0}_{0,\leq \varepsilon}(x_i,z_i)\times \mathcal M^{\Gamma_1}_{0,\leq \varepsilon}(x_i,z_i)$$ In the latter case a broken trajectory must not appear in any component of a cornerpoint.
   We show this in the case where a $1$-pure corner point exists. The other case is analogous. \\
   
   So, let $Q\subset \mathcal M^{\{\Gamma_\lambda\}}_{0,\leq \varepsilon}(x,z)$ be a hypercube that has a $1$-pure corner point. Suppose that some corner point contains broken trajectories. The hypercube is given by a product of closed intervals $$Q = I_1\times \cdots\times I_k,$$ where points in the interior of $I_j$ correspond to pairs $(s_\lambda,\lambda)$ with $s_\lambda\in \mathcal M_{0,\leq \varepsilon}^{\{\Gamma_\lambda\}}(x_j,z_j)$ and the boundary correspond to either a broken trajectory or an element in $\mathcal M^{\Gamma_0}_{0,\leq \varepsilon}(x_j,z_j)\cup \mathcal M_{0,\leq \varepsilon}^{\Gamma_1}(x_j,z_j)$ while at least one component boundary corresponds to a point in $\mathcal M^{\Gamma_0}_{1,\leq \varepsilon}(x_j,z_j)$. If both boundary components of some interval $I_j$ correspond to an element in $\mathcal M^{\Gamma_1}(x_j,z_j)$, there is an even number of $1$-pure corner points in the hypercube. This does not affect the count modulo $2$, hence we may assume that there are none. If the other boundary component corresponds to a broken trajectory $(u_j,v_j)$ or $(v_j,u_j)$, we can associate with it a $\lambda_{j}\in (0,1)$ that is such that $$(u_j,v_j)\in \mathcal M_{-1}^{\Gamma_{\lambda_j}}(x_j,y_j)\times \widehat {\mathcal M}_1(y_j,z_j)\text{ or }(v_j,u_j)\in  \widehat {\mathcal M}_1(x_j,y_j)\times \mathcal M_{-1}^{\Gamma_{\lambda_j}}(y_j,z_j).$$ Let $\lambda_\star = \max \lambda_j$ where the maximum runs over $\lambda_j$ associated to a broken trajectory at a boundary of some $I_j$ ( in the case of $0$-pure corner points we would take the minimum here). We find that for all $j$, there exists a broken trajectory with $\lambda_j = \lambda_\star$ corresponding to a bounded point of $I_j$ or simply $u_j\in \mathcal M^{\Gamma_{\lambda_\star}}(x_j,z_j)$ corresponding to an interior point of $I_j$. In the first case, we write $y_j$ for the breaking orbit of the broken trajectory. After relabeling, we have the following tuple at our hand $$ u = (u_1,\dots,u_p,(u_{p+1},v_{p+1}),\dots,(u_q,v_q),(v_{q+1},u_{q+1}),\dots,(v_k,u_k)),$$ where $u_j\in \mathcal M^{\Gamma_{\lambda_\star}}(x_j,z_j)$ for $j =1,\dots,p$, $(u_j,v_j)\in \mathcal M^{\Gamma_{\lambda_*}}(x_j,y_j)\times \widehat {\mathcal M}(y_j,z_j)$ for $j =p+1,\dots,q$, and $(v_j,u_j) \in  \widehat {\mathcal M}(x_j,y_j)\times \mathcal M^{\Gamma_{\lambda_*}}(y_j,z_j)$ for $j =q+1,\dots,k$. Abusing notation, we shall from now on write $v_j$ for a representative in $\mathcal M_1(x_i,z_j)$ or in $\mathcal M_{1}(y_i,z_i)$. Now, for every loop $\gamma:S^1\to S$ let $$Z_\gamma:\mathbb R \times S^1\to S,\quad (s,t) = \gamma(t)$$ be the constant cylinder. Let us define \begin{align*}
   w_i = \begin{cases}
              Z_{x_i}\#u_i\#Z_{z_i},& \text{for } i =1,\dots,p\\ Z_{x_{i}}\# u_i\# v_i,& \text{for }i = p+1,\dots,q\\
              v_i\#u_i\#Z_{z_i},& \text{for } i = q+1,\dots,k
   \end{cases}
\end{align*} and denote $$w_{i,1} = \begin{cases}Z_{x_i}, & i = 1,\dots,q\\ v_i , & i = q+1,\dots,k\end{cases},\quad w_{i,2} = u_i,\quad w_{i,3} = \begin{cases}
    Z_{z_i}, & i = 1,\dots,p \text{ or }i = q+1,\dots,k\\v_i,& i = p+1,\dots,q
\end{cases}.$$Consider the tuple of cylinders given by $$w =(w_1,\dots,w_k)$$ which is homotopic to $u$ and given by extending by a constant cylinder in some of the components. Any cylinder in this tuple is divided into three parts $w_{i,1},w_{i,2},$ and $w_{i,3}$ where the first and the last part satisfy $$\partial_sw_{i,j}+J_H(\partial_tw_{i,j} -X_H\circ w_{i,j}) = 0$$ for $j = 1,3$ and the middle part satisfies $$\partial_sw_{i,2}+J_{\lambda_\star}(\partial_t w_{i,2}- X_{F_{\lambda_\star}}\circ w_{i,2}) = 0.$$ 
\begin{figure}[h!]
\begin{tikzpicture}[scale = 0.85]
    
    %cylinders
    \draw[] (-8,3) ellipse (1 and 0.2);  
    \draw[] (-4,3) ellipse (1 and 0.2);
    \draw[] (0,3) ellipse (1 and 0.2);
    \draw[] (-7,3)-- (-7,1);
    \draw[] (-9,3)-- (-9,1);
    \draw[] (-5,3)-- (-5,1);
    \draw[] (-3,3)-- (-3,1);
    \draw[] (-1,3)-- (-1,1);
    \draw[] (1,3)-- (1,1);
    \draw[] (-7,1) arc (0:-180:1 and 0.2);  
    \draw[] (-3,1) arc (0:-180:1 and 0.2);
    \draw[] (1,1) arc (0:-180:1 and 0.2);
    \draw[dashed] (-7,1) arc (0:180:1 and 0.2);  
    \draw[dashed] (-3,1) arc (0:180:1 and 0.2);
    \draw[dashed] (1,1) arc (0:180:1 and 0.2);
    \draw[] (-7,1)-- (-7,-1);
    \draw[] (-9,1)-- (-9,-1);
    \draw[] (-5,1)-- (-5,-1);
    \draw[] (-3,1)-- (-3,-1);
    \draw[] (-1,1)-- (-1,-1);
    \draw[] (1,1)-- (1,-1);
    \draw[] (-7,-1) arc (0:-180:1 and 0.2);  
    \draw[] (-3,-1) arc (0:-180:1 and 0.2);
    \draw[] (1,-1) arc (0:-180:1 and 0.2);
    \draw[dashed] (-7,-1) arc (0:180:1 and 0.2);  
    \draw[dashed] (-3,-1) arc (0:180:1 and 0.2);
    \draw[dashed] (1,-1) arc (0:180:1 and 0.2);
    \draw[] (-7,-1)-- (-7,-3);
    \draw[] (-9,-1)-- (-9,-3);
    \draw[] (-5,-1)-- (-5,-3);
    \draw[] (-3,-1)-- (-3,-3);
    \draw[] (-1,-1)-- (-1,-3);
    \draw[] (1,-1)-- (1,-3);
    \draw[] (-7,-3) arc (0:-180:1 and 0.2);  
    \draw[] (-3,-3) arc (0:-180:1 and 0.2);
    \draw[] (1,-3) arc (0:-180:1 and 0.2);
    \draw[dashed] (-7,-3) arc (0:180:1 and 0.2);  
    \draw[dashed] (-3,-3) arc (0:180:1 and 0.2);
    \draw[dashed] (1,-3) arc (0:180:1 and 0.2);
    %equations
    \node[right] at (2,2) {$\partial_sw_{i,1}+J_H(\partial_tw_{i,1}-X_H\circ w_{i,1})=0$};
     \node[right] at (2,0) {$\partial_sw_{i,2}+J_{\lambda_*}(\partial_tw_{i,2}-X_{F_{\lambda_*}}\circ w_{i,2})=0$};
      \node[right] at (2,-2) {$\partial_sw_{i,3}+J_H(\partial_tw_{i,3}-X_H\circ w_{i,3})=0$};
    %labels
    %1
    \node[] at (-8,4) {$i = 1,\dots,p$};
    \node[right] at(-7,3) {$x_i$};
    \node[right] at (-7,1) {$y_{i,1} = x_i$};
    \node[] at (-8,2) {$Z_{x_i}$};
    \node[right] at (-7,-1) {$y_{i,2} = z_i$};
    \node[] at (-8,0) {$u_i$};
    \node[right] at(-7,-3) {$z_i$};
    \node[] at (-8,-2) {$Z_{z_i}$};
    %2
    \node[] at (-4,4) {$i = p+1,\dots,q$};
    \node[right] at (-3,3) {$x_i$};
    \node[] at (-4,2){$Z_{x_i}$};
    \node[right] at (-3,1) {$y_{i,1}=x_i$};
    \node[] at (-4,0){$u_i$};
    \node[right] at (-3,-1) {$y_{i,2}=y_i$};
    \node[right] at (-3,-3) {$z_i$};
    \node[] at (-4,-2){$v_i$};
    %3
    \node[] at (0,4) {$i = q+1,\dots,k$};
    \node[right] at(1,3) {$x_i$};
    \node[] at (0,2){$v_i$};
    \node[right] at (1,1) {$y_{i,1} = y_i$};
    \node[] at (0,0) {$u_i$};
    \node[right] at (1,-1) {$y_{i,2} = z_i$};
    \node[] at (0,-2){$Z_{z_i}$};
    \node[right] at (1,-3) {$z_i$};
\end{tikzpicture}
\caption{The cylinders $w_i$.}
\end{figure}
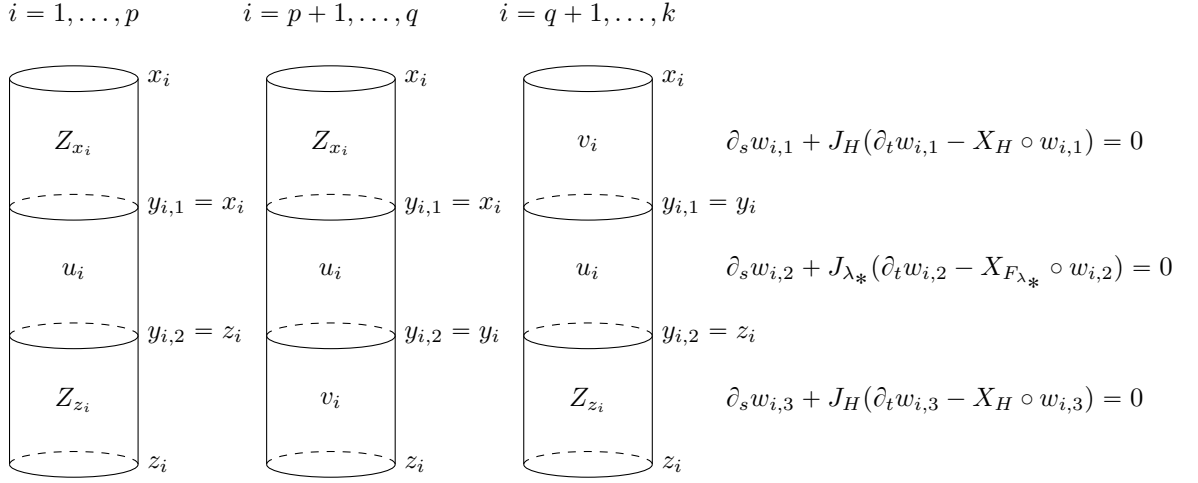
Let $y_{j,1}$ and $y_{j,2}$ denote the periodic orbits connecting the first two parts and the latter two parts, respectively, of the $j$-th component.
    That is $$y_{j,1}=\begin{cases}
        x_j, &\text{for } j = 1,\dots,q\\y_j,&\text{for } j = q+1,\dots,k
    \end{cases}$$ and $$y_{j,2}=\begin{cases}
        z_j, &\text{for } j = 1,\dots,p\text{ or } j = q+1,\dots,k\\y_j,&\text{for } j = p+1,\dots,q
    \end{cases}.$$ \\
    \underline{Observation 1:} For every hypercube $Q\subset \mathcal M^{\{\Gamma_\lambda\}}_{0,\leq \varepsilon}(x,z)$ that has a $1$-pure corner point $c$ we have $$\sum_{i=1}^kE(w_{i,1})+E(w_{i,3}) <\min\{\mathcal E(H,J,x),\mathcal E^{\operatorname{int}}(H,J,x),\mathcal E^2(H,J,x)\}.$$ In particular, for any $v_i = w_{i,j}$ for $j = 1$ or $j = 3$ we have $$E(v_i)\leq\min\{\mathcal E(H,J,x),\mathcal E^{\operatorname{int}}(H,J,x),\mathcal E^2(H,J,x)\} $$
    
    \begin{proof}We have $$ \mathcal A_{H}([x_i,C_{x_i}])-\mathcal A_H([z_i,C_{x_i}\#w_i])  = E(w_{i,1})+e_i+E(w_{i,3}),$$ where $e_i =\mathcal A_H([y_{i,1},C_{y_{i,1}}])-\mathcal A_H([y_{i,2},C_{y_{i,1}}\#w_{i,2}])
    \geq -E^+(\Gamma)$. Using that $$\sum_{i =1}^k\mathcal A_{H}([x_i,C_{x_i}])-\mathcal A_H([z_i,C_{x_i}\#w_i])\leq\varepsilon$$ we find\begin{align*}
        \begin{split}
            \sum_{i=1}^k E(w_{i,1})+E(w_{i,3})&\leq \sum_{i=1}^k E(w_{i,1})+E(w_{i,3}) +\sum_{i = 1}^ke_i+ kE^+(\Gamma)\\ & \leq\left ( \sum_{i = 1}^k\mathcal A_{H}([x_i,C_{x_i}])-\mathcal A_H([z_i,C_{x_i}\#w_i]) \right)+kE^+(\Gamma)\\ &\leq  \varepsilon+kE^+(\Gamma)\\ & < \min\{\mathcal E(H,J,x),\mathcal E^{\operatorname{int}}(H,J,x),\mathcal E^2(H,J,x)\}
        \end{split}
    \end{align*}  This concludes the proof of the observation. \end{proof}
    The quantities, $\mathcal E^{\operatorname{int}}(H,J,x)$ and $\mathcal E^2(H,J,x)$, essentially only play a role in the following observation.\\ $ $\\
    \underline{Observation 2:} For $i = 1,2$ the set $\{y_{j,i}\}_{j =1,\dots,k}$ consist of distinct orbits.

\begin{proof}
    We prove this or $i = 1$; the other case is analogous. (For the other case, recall that we consider $z\in OB\mathcal P_\beta(H)$ with $[x] = [z]$. This means that there is $\sigma \in S_k$ such that $z = \sigma(x)$.)
    Suppose that $y_{j,1} = y_{j',1}$ for $j \neq j'$. For every index $\ell$ we either have $ y_{\ell,1} = x_\ell$ or $y_{\ell,1} = y_{\ell}$ is a breaking orbit. There are four cases to consider, and we show that each leads to a contradiction. \begin{enumerate}
    \item[] Case 1: $x_j = y_{j,1}$ and $ y_{j',1} = x_{j'}$. \\
    This is the easiest case since we have that $x_j\neq x_{j'}$ by definition.\\
    \item[]Case 2: $y_{j,1} = y_j$ and $y_{j',1} = x_{j'}$.\\
    Since $y_{j,1} = y_{j',1}$ by assumption, we have $y_j = x_{j'}$ and in this case $v_j = w_{i,j}$ for $j= 1,3$ satisfies $$\partial_sv_i+J_H(\partial_tv_i-X_H\circ v_i)=0.$$ The cylinder $v_j$ connects $x_j$ and $x_{j'}$ which are both orbits in $x$. However, in Observation 1 we saw that we have $E(v_i)<\min\{\mathcal E(H,J,x),\mathcal E^{\operatorname{int}}(H,J,x),\mathcal E^2(H,J,x)\}\leq \mathcal E^{\operatorname{int}}(H,J,x)$. By the definition of $\mathcal E^{\operatorname{int}}(H,J,x)$, this is not possible, and hence we conclude that $y_j\neq x_{j'}$ for every $j,j'$. \\
    \item[] Case 3: $y_{j,1} = x_j$ and $y_{j',1} = y_{j'}$.\\
    This is Case 2 with indices swapped. \\
    \item[] Case 4: $y_{j,1} = y_j$ and $y_{j',1} = y_{j'}$.\\
  We find that for one of the curves $v\in\{v_j,v_{j'}\}$ we have $$\mathcal E^2(H,J,x)\leq \max\{E(v_j),E(v_{j'})\}=  E(v).$$ This again contradicts Observation 1 where we found that $E(v)= \max\{E(v_j),E(v_{j'})\}<\min\{\mathcal E(H,J,x),\mathcal E^{\operatorname{int}}(H,J,x),\mathcal E^2(H,J,x)\}$.
     
    \end{enumerate}This concludes the proof of the observation.
    \end{proof}

    After this observation, we have $$0 = \iota(u) = \iota(w) = \sum_{i \neq j}\iota(w_{i,1},w_{j,1})+\iota(w_{i,2},w_{j,2})+\iota(w_{i,3},w_{j,3}).$$ By positivity of intersection (Lemma~\ref{prop.positivityofintersection}), we can now conclude that for all $i \neq j$ $$\iota(w_{i,1},w_{j,1}) = \iota(w_{i,2},w_{j,2})= \iota(w_{i,3},w_{j,3}) = 0$$ and hence for all $i \neq j$ and all $(s,t)\in \mathbb R\times S^1$ we have $$w_{i,r}(s,t) \neq w_{j,r}(s,t),$$ for $r =1,2,3$.  Hence, the curves induce isotopies between the braids   $B(H,\{x_1,\dots,x_k\})$, $B(H,\{y_{1,1},\dots,y_{k,1}\})$, $B(H,\{y_{1,2},\dots,y_{k,2}\}),$ and  $B(H,\{z_1,\dots,z_k\})$ and we have that $$(w_{1,1},\dots,w_{k,1})\in \mathcal M(x,(y_{1,1},\dots,y_{k,1})),\quad (w_{1,3},\dots,w_{k,3})\in \mathcal M((y_{1,2},\dots,y_{k,2}),z).$$  By assumption we have that $\mu_{\operatorname{max}}(w_{1,1},\dots,w_{k,1}) =1$ or $\mu_{\operatorname{max}}(w_{1,3},\dots,w_{k,3}) =1$. Without loss of generality, we may assume that this holds in the first case. After the observation we know that $$E(w_{1,1},\dots,w_{k,1})=\sum_{i =1}^kE(w_{i,1}) < \mathcal E(H,J,x).$$ However, this contradicts the definition of $\mathcal E(H,J,x)$ and hence the assumption that broken trajectories appear as components of admissible corner hypercubes. This means that the corner points must lie in $$\prod_{i =1}^k\mathcal M^{\Gamma_0}(x_i,z_i)\cup \mathcal M^{\Gamma_1}(x_i,z_i).$$ Now for a hypercube $Q\subset \mathcal M^{\{\Gamma_\lambda\},\mathrm{cube}}_{0,\leq \varepsilon}(x,z)$ $$Q= I_1\times \cdots\times I_k$$ with a $0$- or $1$-pure corner point where $I_j \subset \mathcal M^{\{\Gamma_\lambda\}}_0(x_i,z_i)$ are connected components which are closed intervals, we define colorings of the boundary $\rho_i :\partial I_i\to \{0,1\}$ by the type of boundary component $$\text{Type 0: }(u_i,0) \text{ where } u_i \in \mathcal M^{\Gamma_0}_{0,\leq \varepsilon}(x_i,z_i), $$ $$\text{Type 1: }(u_i,1) \text{ where }u_i \in \mathcal M^{\Gamma_1}_{0,\leq \varepsilon}(x_i,z_i). $$ By Lemma~\ref{combinatorial lemma}, we have that the number of $0$- and $1$-pure corner points of such hypercubes is even. These corner points correspond to the set $$\mathcal M^{\Gamma_0}_{0,\leq \varepsilon}(x,z)\cup \mathcal M^{\Gamma_1}_{0,\leq \varepsilon}(x,z).$$ Therefore we find $$ \pi_{[x]}(\Phi_\varepsilon^{\Gamma_1}([x])+\Phi_{\varepsilon}^{\Gamma_0}([x])) = \sum _{\substack{z\in OB\mathcal P_{\beta,f}(H)\\ [z] = [x]}}\,|\,\mathcal M^{\Gamma_0}_{0,\leq \varepsilon}(x,z)\cup \mathcal M^{\Gamma_1}_{0,\leq \varepsilon}(x,z)\,|\,[z] = 0 \quad (\text{mod }2).$$  Since $\Gamma_0$ is the constant continuation data, it follows from standard arguments (see \cite[Proposition 11.1.14]{AD14}) $$\,|\,\mathcal M^{\Gamma_0}_{0,\leq \varepsilon} (x,z)\,|\, = \begin{cases}
        0 , & x \neq z\\ 1, & x= z
    \end{cases}$$ and hence $\Phi_\varepsilon^{\Gamma_0} = id_{\operatorname{BCF}_{\beta,f}(H)}$. This means that $$\pi_{[x]}\circ \Phi^{\Gamma_1}_\varepsilon([x]) = [x]$$ which concludes the proof of Proposition~\ref{Proposition: bounded continuation maps identity}.
    \end{proof}
Using this result, we can now prove the braid stability statement of Theorem~\ref{thm: braid stability}.

\subsection{Braid Stability}
This section proves the quantitative braid stability result in Theorem~\ref{thm: braid stability}. We briefly outline the proof strategy. Given non-degenerate Hamiltonians $ H, G$, and $x\subset \mathcal P(H)$ with associated braid class $\beta$ we choose suitable almost complex structures and continuation data $\Gamma''$ from $G$ to $H$ and $\Gamma'$ from $H$ to $G$ such that $$E^+(\Gamma''\#_\rho\Gamma')= E(H-G)<\mathcal S(H,x).$$ The continuation data $\Gamma''\#_\rho\Gamma'$ connects $H$ to itself, and we can apply Proposition~\ref{Proposition: bounded continuation maps identity}. We conclude that there is an odd number of tuples of continuation cylinders $w$ from $x$ to some $z\in OB\mathcal P_\beta(H)$ with $[x] = [z]$. For such a tuple $w=(w_1,\dots,w_k)$ the continuation cylinders $w_i$ are obtained by gluing the continuation cylinders $u_i$ and $v_i$ with respect to the data $\Gamma''$ and $\Gamma'$. The tuples $(u_1,\dots,u_k)$ end in breaking orbits $y_1,\dots,y_k\in \mathcal P(G)$ which may come with multiplicities. Applying a symmetry argument, we show that the number of such tuples, where the breaking orbits come with multiplicities, is even. This implies that there is at least one tuple of $(u_1,\dots,u_k)$ ending in pairwise distinct orbits. Using that $\iota(u_i,u_j) = 0$ we conclude that the tuple $(u_1,\dots,u_k)$ induces an isotopy of the braids $B(H,x)$ and $B(G,y)$, which means that $\dim \operatorname{BCF}_{\beta,f}(G)\geq 1$. 
\begin{proof}[Proof of Theorem~\ref{thm: braid stability}]
     Choose an almost complex structure $J$ such that $(H,J)$ is regular pair and $kE(H-G)<\min\{\mathcal E(H,J,x),\mathcal E^{\operatorname{int}}(H,J,x),\mathcal E^2(H,J,x)\}$. We define self-continuation data $\Gamma = \Gamma''\#_\rho \Gamma'$ at $(H,J)$ and will apply Proposition~\ref{Proposition: bounded continuation maps identity} to it. Let $f:\mathbb R \to [0,1]$ be smooth increasing such that $f(s) = 0$ for $s<-1$ and $f(s) = 1$ for $s>1$. Let $H'$ and $H''$ be interpolations $\mathbb R\times S^1\times S \to \mathbb R$ between $H$ and $G$ defined as follows $$H'(s,t,x) = H(t,x) + f(s)(G(t,x)-H(t,x)),$$ $$H''(s,t,x) = G(t,x)+f(s)(H(t,x)-G(t,x)).$$ Note that $$\int_{-\infty}^{s_0}\int_0^1\max_{p\in S}\partial _sH'(s,t,p)\,dtds = -f(s_0)\int_0^1\min_{x\in S}\left(H(t,x)-G(t,x)\right)\, dt$$ $$\int_{-\infty}^{s_0}\int_0^1\max_{p\in S}\partial _sH''(s,t,p)\,dtds =f(s_0)\int_0^1\max_{x\in S}\left(H(t,x)-G(t,x)\right)\,dt.$$ 
    We have $$E^+(\Gamma') = -\int_0^1\min_{x\in S}\left(H(t,x)-G(t,x)\right)\, dt\quad \text{and}\quad E^+(\Gamma'') =  \int_0^1\max_{x\in S}\left(H(t,x)-G(t,x)\right)\,dt.$$ 
     In particular, $kE^+(\Gamma''\#_\rho\Gamma') = kE(H-G)$ which is bounded above by $k\mathcal S(H,x)$ and hence also by $\min\{\mathcal E(H,J,x),\mathcal E^{\operatorname{int}}(H,J,x),\mathcal E^2(H,J,x)\}$. By choosing suitable interpolations of almost complex structures we have continuation data $\Gamma' = (H',J')$ and $\Gamma''= (H'',J'')$ such that after a $C^\infty$-small perturbation we have that $\Gamma',\Gamma'',$ and $\Gamma''\#_\rho \Gamma'$ is regular for $\rho$ large enough and such that $kE^+(\Gamma''\#_{\rho}\Gamma')<\min\{\mathcal E(H,J,x),\mathcal E^{\operatorname{int}}(H,J,x),\mathcal E^2(H,J,x)\}$. Now, applying Proposition~\ref{Proposition: bounded continuation maps identity} we find  $\pi_{[x]}\circ \Phi^{\Gamma''\#_\rho\Gamma'}_0([x]) = [x].$ This means that there are an odd number of cylinders in the following set $$\mathcal T(x) = \bigcup_{\substack{z\in OB\mathcal P_\beta\\ [z] = [x]}}\mathcal M^{\Gamma''\#_{\rho}\Gamma'}_{0,\leq 0}(x,z).$$ We now use a standard result from Floer theory on gluing continuation data. Namely, there is a bijection $$\chi: \bigcup_{y\in \mathcal P(G)} \mathcal M^{\Gamma'}_0(x_i,y)\times \mathcal M^{\Gamma''}_0(y,z_i)\to \mathcal M^{\Gamma''\#_{\rho}\Gamma'}_0(x_i,z_i),\quad (u,v)\mapsto \chi(u,v)$$ which is obtained by gluing the cylinders in a fashion described in Section~\ref{section.gluing}. In particular, Lemma~\ref{homotopy to breaking} applies which means that for pairs $(u_i,v_i)\in \bigcup_{y\in \mathcal P(G)} \mathcal M^{\Gamma'}_0(x_i,y)\times \mathcal M^{\Gamma''}_0(y,z_i)$ we have $$\sum_{i \neq j}\iota(u_i\#v_i,u_j\#v_j) = \sum_{i\neq j}\iota(\chi(u_i,v_i),\chi(u_j,v_j)).$$ Let $$B\mathcal T(x) = \{(\chi^{-1}(w_1),\cdots,\chi^{-1}(w_k))\,|\,w= (w_1,\cdots,w_k)\in \mathcal T(x)\}$$ be the set of tuples of broken trajectories corresponding to a tuple of trajectories 
     $w\in \mathcal T(x)$. For every $w\in \mathcal  T(x)$ let $y^w = (y_1^w,\cdots,y_k^w)\in \mathcal P(G)^k$ be the tuple of breaking orbits of the broken trajectory $(\chi^{-1}(w_1),\cdots,\chi^{-1}(w_k))\in B\mathcal T(x)$. We say a trajectory $w\in \mathcal T(x)$ is \emph{degenerate} if the associated tuple of breaking orbits $y^w$ has multiplicities in its entries, that is, if 
    $y^w\in \Delta_{\mathcal P(G)}$. \\ $ $ \\ \underline{Observation:} It is sufficient to prove that there is a regular $w\in \mathcal T(x)$. \begin{proof}
        Assume that $w\in \mathcal T(x)$ is regular. Then consider the associated broken trajectory $$((u_1,v_1),\dots,(u_k,v_k))= (\chi^{-1}(w_1),\dots,\chi^{-1}(w_k))\in B\mathcal T(x).$$ We noticed before that $$0 = \sum_{i\neq j}\iota(u_i\#v_i,u_j\#v_j).$$ Since $w$ is regular, the breaking orbits $y_1^w,\cdots,y_k^w$ are pairwise distinct which means that the intersection numbers $\iota(u_i,u_j)$ and $\iota(v_i,v_j)$ are well-defined for $i\neq j$ and that $$0 =  \sum_{i\neq j}\iota(u_i\#v_i,u_j\#v_j)=\sum_{i \neq j}\iota(u_i,u_j)+\iota(v_i,v_j).$$ After positivity of intersection, this means that $$0 = \iota(u_i,u_j) = \iota(v_i,v_j)$$ for all $i \neq j$. This in particular means that for all $(s,t)\in \mathbb R\times S^1$ we have $u_i(s,t) \neq u_j(s,t)$ and therefore, the cylinders $(u_1,\cdots,u_k)$ induce an isotopy of the braids $B(H,x)$ and $B(G,y^w)$. Since the cylinders $u_i$ connect orbits of the same index, the framing is preserved. Thus, there is a braid of orbits of $G$ with framed braid class $(\beta,f)$, so $\dim \operatorname{BCF}_{\beta,f}(G)\geq 1$. This concludes the proof of the observation.
    \end{proof} We now show that there is a regular trajectory in $\mathcal T(x)$. Let $$D\mathcal T(x) = \{w\in \mathcal T(x)\,|\,y^w\in \Delta_{\mathcal P(G)}\}$$ be the set of degenerate trajectories. Since $\mathcal T(x)$ splits as a disjoint union into the set of regular and degenerate trajectories, and since we already established that the cardinality of $\mathcal T(x)$ is odd, it suffices to show that $$|D\mathcal T(x)| = 0 \quad (\text{mod }2).$$ We remark here that if the framing is injective, we are done since in this case $\mu(x_i ) = \mu(y_i^w)$ are pairwise distinct and hence, the orbits $y_i^w$ are always pairwise distinct, which means that $D\mathcal T(x) = \emptyset$. In the general case, we argue by symmetry. For every $y=(y_1,\dots,y_k)\in \Delta_{\mathcal P(G)}$ and $z\in OB\mathcal P_\beta(H)$ such that $[x] = [z]$ let $$DB\mathcal T(x,y,z) = \left \{((u_1,v_1),\dots,(u_k,v_k))\in B\mathcal T(x)\,\middle |\,\begin{aligned} &\text{for } w= (\chi(u_1,v_1),\dots,\chi(u_k,v_k)),\\ &  y^{w} = y,w\in \mathcal M^{\Gamma''\#_\rho\Gamma'}_0(x,z)\end{aligned}\right \},$$ the set of broken trajectories with breaking orbits $y$ corresponding to degenerate trajectories connecting $x$ to $z$. Note that the cardinality of $$D\mathcal T(x) \quad  \text{ and } \quad \bigcup_{y\in \Delta_{\mathcal P(G)}}\bigcup_{\substack{z\in OB\mathcal P_\beta(H)\\ [x] =[z]}}DB\mathcal T(x,y,z) $$ agrees.  For every $y\in \Delta _{\mathcal P(G)}$ we show that the cardinality of $\bigcup_{\substack{z\in OB\mathcal P_\beta(H)\\ [x] =[z]}} DB\mathcal T(x,y,z)$ is even which concludes the result. Let us denote by $$\text{Fix}(y) = \{\sigma \in S_k\,|\, \sigma(y) = y\},$$ the set of permutations that leave $y$ invariant. For every $\sigma\in \text{Fix}(y)
    $ we have \begin{align}
        \label{eq: proof braid stability symmetry} |DB\mathcal T(x,y,z)| = |DB\mathcal T(x,y,\sigma(z))|.
    \end{align} Indeed, the following map is a bijection $$S^{x,y,z,\sigma}:DB\mathcal T(x,y,z) \to DB\mathcal T(x,y,\sigma(z)),\quad ((u_1,v_1),\dots,(u_k,v_k))\mapsto ((u_1,v_{\sigma(1)}),\dots ,(u_k,v_{\sigma(k)})).$$ We argue that this map is well-defined. First, note that $((u_1,v_{\sigma(1)}),\dots ,(u_k,v_{\sigma(k)}))\in \mathcal M_0^{\Gamma'}(x,y)\times \mathcal M_0^{\Gamma''}(y,\sigma(z))$ since we have $\sigma(y) =y$. By Lemma~\ref{swapping lemma} we have $$0= \sum_{i \neq j}\iota(u_i\#v_i,u_j\#v_j) = \sum_{i \neq j}\iota(u_i\#v_{\sigma(i)},u_j\#v_{\sigma(j)}).$$ By Lemma~\ref{homotopy to breaking} we then also have $$0=\sum_{i\neq j}\iota(\chi(u_i,v_{\sigma(i)}),\chi(u_j,v_{\sigma(j)})).$$ Moreover, \begin{align*}
        \begin{split}0 &= \left |\sum_{i =1}^k\mathcal A_H([x_i,C_{x_i}])-\mathcal A_H([z_i,C_{x_i}\#u_i\#v_i])\right| \\ & =  \left |\sum_{i =1}^k\mathcal A_H([x_i,C_{x_i}])-\mathcal A_H([z_{\sigma(i)},C_{x_i}\#u_i\#v_{\sigma(i)}])\right |\end{split}\end{align*} and hence $(\chi(u_1,v_{\sigma(1)}),\dots,\chi(u_k,v_{\sigma(k)}))\in \mathcal T(x).$ By definition, $((u_1,v_{\sigma(1)}),\cdots,(u_k,v_{\sigma(k)}))\in DB\mathcal T(x,y,\sigma(z))$. To see that the map $S^{x,y,z,\sigma}$ is a bijection, note that an inverse is given by $S^{x,y,\sigma(z),\sigma^{-1}}$, which is defined since $\text{Fix}(y)$ is a group and $\sigma^{-1}\in \text{Fix}(y)$. This shows (\ref{eq: proof braid stability symmetry}). We now have \begin{align*}
        \begin{split}
        |D\mathcal T(x)| & = \sum_{y\in \Delta_{\mathcal P(G)}}\sum_{\substack{z\in OB\mathcal P_\beta(H)\\ [x] = [z]}}|DB\mathcal T(x,y,z)| \\ & = \sum_{y\in \Delta_{\mathcal P(G)}}\sum_{\sigma \in S_k}|DB\mathcal T(x,y,\sigma(x))| \\ & =  \sum_{y\in \Delta_{\mathcal P(G)}}\sum_{[\nu] \in S_k/\text{Fix}(y)}\sum _{\sigma \in \text{Fix}(y)}|DB\mathcal T(x,y,\sigma(\nu(x)))| \\ & =  \sum_{y\in \Delta_{\mathcal P(G)}}\sum_{[\nu] \in S_k/\text{Fix}(y)} |\text{Fix}(y)||DB\mathcal T(x,y,\nu(x))| \\ & = 0 \quad (\text{mod } 2),
        \end{split}
    \end{align*} where we use that $[x] = [z]$ if and only if there is $\sigma \in S_k$ such that $\sigma(z) = x$ and that for every $y\in\Delta_{\mathcal P(G)}$ the group $\text{Fix}(y)$ is a non-trivial product of permutation groups and therefore has even cardinality.
\end{proof}

\section{An Application: Contractible Orbits and Entropy of Eggbeaters on $T^2$}\label{section: application}
The goal of this section is to prove Theorem~\ref{thm:entropy}. We outline the strategy. \begin{enumerate}
    \item In Section~\ref{section: braids and entropy} we define a notion of entropy $h(\beta)$ of a closed braid $\beta$ that satisfies the following. For a Hamiltonian $H$ and periodic orbits $\{x_1,\dots,x_k\}\subset \mathcal P(H)$ at fixed points $z_1,\dots, z_k$ of $\phi_H^1$ we have $$h(\mathcal B(H,x_1,\dots,x_k)) \leq h_{\operatorname{top}}(\phi_H^1).$$
    \item In Section~\ref{subsection: torus braids} we study closed braids on symplectic tori and define a family of closed braids $\beta_{m,n}$ with two contractible strands such that $${1\over 5}\log(|mn|)\leq h(\beta_{m,n}) $$ and hence $h(\beta_{m,n}) \to \infty$ as $m,n\to \infty$.
    \item In Section~\ref{subsection: eggbeater} we first define a sequence of Hamiltonians $H_k$ whose flows are eggbeater maps on the torus. 
    \item We then analyze the contractible orbits of the flow of $H_{2k}$ and fix a pair of contractible orbits $x_1^{(k)},x_2^{(k)}\in \mathcal P(H_{2k})$ whose associated framings $$f_{k}: x_1^{(k)}\mapsto 2,\quad x_2^{(k)}\mapsto -2$$  are injective and such that $ \mathcal B(H_{2k},x_1^{(k)},x_2^{(k)})=\beta_{2k+1,2k+1}.$
    \item Using topological obstructions we restrict the set of pairs $\{y_1,y_2\}\subset \mathcal P(H_{2k})$ that can be connected to $\{x_1,x_2\}$ by an isotopy. Using this restriction on Floer isotopies, a computation then shows the following for the stability radius of the braids of orbits $\{x_1^{(k)},x_2^{(k)}\}$ $$\mathcal S(H_{2k},x_1^{(k)},x_2^{(k)})\to \infty \quad (k\to \infty).$$ Here, the fact that we use braids is essential since it refines the set of actions to which we need to compare the action of $x_1^{(k)},x_2^{(k)}$. 
    \item Applying the braid stability result Theorem~\ref{thm: braid stability} we then conclude that in a large Hofer neighborhood of $\phi_{2k}$ any Hamiltonian diffeomorphism is generated by a Hamiltonian $G$ whose flow has a braid of orbits of type $\beta_{2k+1,2k+1}$. This gives a lower bound on the topological entropy of $\psi$ and proves Theorem~\ref{thm:entropy}.
\end{enumerate}
\subsection{Preliminaries: Surface Braids and Entropy}\label{section: braids and entropy}
 For $k$ distinct points $y_1,\dots,y_k\in S$ the \emph{ braid group} with $k$ strands is defined by equipping the set of isotopy classes of a set of $k$  (indistinguishable) strands $$s_i: [0,1]\to S\times [0,1],\quad t\mapsto (x_i(t),t)$$ for paths $x_i:[0,1]\to S$ such that $\{x_1(0),\dots,x_k(0)\} = \{x_1(1),\dots,x_k(1)\} = \{y_1,\dots,y_k\}$, with
 the group multiplication given by juxtaposition of braids (for more details see \cite{BM12}). We denote this group by $\mathrm{BG_k}(y_1,\dots,y_k).$ The \emph{pure braid group} $$\mathrm{PB_k}(y_1,\dots,y_k) \subset \mathrm{BG_k}(y_1,\dots,y_k)$$ is defined to be the subgroup given by the elements where $x_i$ are loops, that is $x_i(0) = x_i(1)$ for all $i$. For every element in the pure braid group, the closure gives a well-defined braid class invariant under conjugation in the total braid group. More precisely, there is the following correspondence $$\{\mathrm{Cl}_{\mathrm{BG_k}}(b)\,|\,b\in \mathrm{PB_k}(y_1,\dots,y_k)\}\xleftrightarrow[\text{closure}]{1:1}\{\text{classes of closed surface braids with $k$ strands}\},$$ where $\mathrm{Cl}_{\mathrm{BG_k}}(b)$ is the conjugacy class of $b$ in the $\mathrm{BG_k}(y_1,\cdots,y_k)$.
 
 Moreover, there is a group homomorphism $${\text{Push}}: \mathrm{BG_k}(y_1,\dots,y_k) \to \mathrm{MCG}(S\setminus\{y_1,\dots,y_k\}),$$ where $\mathrm{MCG}(S\setminus\{y_1,\dots,y_k\})$ is the mapping class group of the $k$-times punctured surface $S\setminus \{y_1,\dots,y_k\}$ (see \cite{BM12}). It is characterized by the property that for the flow $\phi^t$ with fixed points  $y_1,\dots y_k$ of $\phi^1$
 the mapping class represented by $\phi^1\,|\,_{S\setminus \{y_1,\dots,y_k\}}$ is the image of $((\phi^t(y_1),t),\dots,(\phi^t(y_k),t))$.

 For a compact topological space  $X$, there are several ways to define the topological entropy of a homeomorphism $f:X\to X$. When the topology of $X$ is metrizable, one can define it as follows. First, choose a metric $d$ inducing the topology and for every $n\in \mathbb Z_{\geq 1}$ define a new metric by $$d_{n}(x,y) = \max_{i = 0,\dots,{n-1}}d(f^i(x),f^i(y)).$$ We say a set $P$ is $(n,\varepsilon)$-spanning if for all $x\in X$ there is some $p\in P$ with $d_n(x,p)<\varepsilon.$ Let $r(f,n,\varepsilon)$ be the minimal cardinality of a set $P$ that is $(n,\varepsilon)$-spanning; by compactness this number is finite. We then define $$h_{\mathrm{\operatorname{top}}}(f) = \lim_{\varepsilon\to 0}\limsup_{n\to \infty}{1\over n}\log(r(f,n,\varepsilon)).$$ This agrees with the general definition of topological entropy when $X$ is not metrizable and is independent of the choice of metric $d$.

 Suppose that $F\subset X$ is a finite set. To a mapping class $[f]\in \mathrm{MCG}(X\setminus F)$ one can associate a notion of entropy, which is a measure of the growth rate of words under the action of $f_*$  on $\pi_1(X\setminus F)$. This algebraic definition of entropy is a lower bound on the topological entropy of the mapping class. Let $A$ be a set of generators of $\pi_1(X\setminus F)$. For every $w\in \pi_1(X\setminus F)$ denote by $L_A(w)$ the minimal length of the word $w$ in the letters $A\cup A^{-1}$. Then define $$h([f]) = \sup_{w\in \pi_1(X\setminus F)}\limsup_{n\to \infty}{1\over n}\log(L_A(f^n_*(w))).$$ It is not hard to see that it suffices only to consider $w\in A$ and that the expression is independent of the choice of generators $A$. 
Given a homeomorphism $f:X\to X$ such that $f(F) = F$, for a finite set $F$, one can conclude that $$h([f\,|\,_{X\setminus F}]) \leq h_{\mathrm{\operatorname{top}}}(f),$$by compactifying the space $S\setminus F$ by adding boundary components at each puncture consisting of copies of $S^1$ (see, for example, \cite{Bo06}).
Passing to conjugacy classes, the push map associates closed surface braids to conjugacy classes of mapping classes. Since the topological entropy of the mapping class is invariant under conjugation, the push map induces a notion of entropy for closed surface braids $\beta$ $$h(\beta) := h(\text{Push}(\beta)).$$ Importantly, if $\phi^t$ is a flow and $F$ is a set of fixed points of $\phi^1$, the topological entropy of the closed braid $\mathcal B(\phi^t,F)$, which is the isotopy class of $\{(\phi^t(x))\}_{x\in F}$, gives a lower bound on the entropy of $\phi^1$, namely $$h(\mathcal B(\phi^t,F))\leq h_{\mathrm{\operatorname{top}}}(\phi^1).$$
A class of dynamically interesting homeomorphisms is \emph{pseudo-Anosov} maps. For such maps, there exist two transverse measured foliations and a \emph {dilatation} $\lambda>1$ such that $f$ stretches one foliation by a factor $\lambda$ and the other by $1/\lambda$. A mapping class $[f]$ is pseudo-Anosov if there is a pseudo-Anosov representative $f$. In this case, this representative minimizes the topological entropy in the mapping class, and we have the following (see \cite{FS79}) $$h_{\operatorname{top}}(f) = h([f]) = \log(\lambda).$$ We say $\lambda$ is the dilatation of the mapping class. There are many ways to estimate $\lambda$; in the case when $[f]$ is the image of a point-push map along a loop $\gamma$, we can estimate $\lambda$ by the self-intersection number defined by $$i(\gamma) = {1\over 2}\min_{\mu}|\{(t_1,t_2)\in S^1\times S^1\;|\;t_1\neq t_2 \text{ and }\mu(t_1) = \mu(t_2)\}|,$$ where the minimum runs over all loops $\mu$ freely homotopic to $\gamma$  (see \cite{Do11}). A loop $\gamma$ on a surface $S_{g,n}$ of genus $g$ with $n$ punctures is called \emph{filling} if any loop freely homotopic to $\gamma$ intersects any simple closed essential loop on $S_{g,n}$ (where essential means non-contractible and not homotopic to a simple loop around a puncture). By Kra \cite{Kr81}, the associated point-push mapping class $[f_\gamma]$ has positive entropy if $\gamma$ is filling. Analyzing the self-intersection number of $\gamma$, one can give a quantitative estimate (see \cite{Do11}). Let $p$ be a marked point on $
S_{g,n}$ and $\gamma$ a loop starting at $p$. By \cite[Theorem 1.5]{Do11} if $3g+n>3$ and $\gamma$ represents a primitive element in $\pi_1(S_{g,n},p)$ then for the point push map $\mathrm{\text{Push}}(\gamma)\in \mathrm{MCG}(S_{g,n}\setminus\{p\})$ the dilatation $\lambda$ of the mapping class satisfies  $$\lambda \geq \sqrt[5]{(i(\gamma)+1)}.$$ We will work on the once punctured torus $S_{1,1}$ where the inequality $3g+n>3$ holds.
\subsection{Torus Braids}\label{subsection: torus braids}
For two base points $z_1,z_2\in T^2$ the \emph{pure braid group on the torus with two strands} is given by $\mathrm{PB}_2(T) = \pi_1(X_2(T),(z_1,z_2))$ where $$X_2(T) = \{(x,y)\in T^2\times T^2\,|\,x\neq y\}$$ is the configuration space of two points. A presentation of this group is given by  $$\mathrm{PB}_2(T) = \langle \tau_{1,a},\tau_{1,b},\tau_{2,a},\tau_{2,b}\,|\, \text{Relations}\rangle$$ with complicated relations that we omit here (we point the interested reader to the reference \cite{Sc08}). Here the generators $\tau_{i,a}$ and $\tau_{i,b}$ thought of the braids that correspond to \begin{align}\label{def: torus braid representatives}\begin{split}&\tau_{1,a} = (z_1+(t,0),z_2),\quad \tau_{1,b}= (z_1+(0,t),z_2)\\& \tau_{2,a} = (z_1,z_2+(t,0)),\quad \tau_{2,b}= (z_1,z_2+(0,t))\end{split}\end{align} which are pairs of loops for $t\in [0,1]$.\\ $ $\\ 
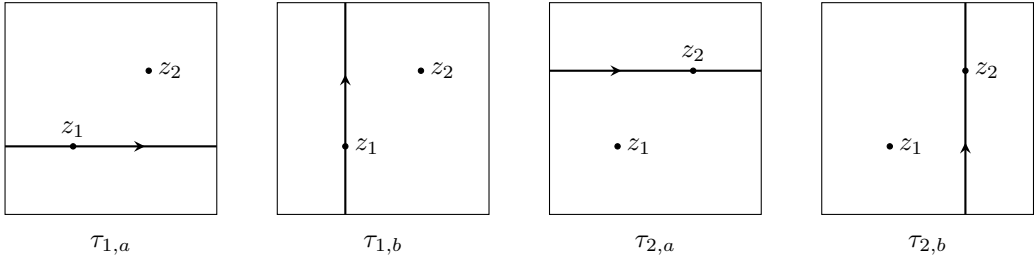
\begin{center}
\begin{tikzpicture}
    \begin{scope}[shift = {(-5.4,0)}]
        \draw[] (-1.4,-1.4) -- (1.4,-1.4) -- (1.4,1.4)--(-1.4,1.4)--cycle;
         \filldraw[black] (-0.5,-0.5) circle (1pt) node[above] {$z_1$};
         \filldraw[black] (0.5,0.5) circle (1pt) node[right] {$z_2$};
         \draw[thick, postaction={decorate,
  decoration={markings, mark=at position 0.5 with {\arrow{stealth}}}}](-0.5,-0.5)-- (1.4,-0.5);
  \draw[thick](-1.4,-0.5)-- (-0.5,-0.5);
  \node[] at(0,-1.8) {$\tau_{1,a}$};
    \end{scope}
    \begin{scope}[shift = {(-1.8,0)}]
        \draw[] (-1.4,-1.4) -- (1.4,-1.4) -- (1.4,1.4)--(-1.4,1.4)--cycle;
         \filldraw[black] (-0.5,-0.5) circle (1pt) node[right] {$z_1$};
         \filldraw[black] (0.5,0.5) circle (1pt) node[right] {$z_2$};
          \draw[thick, postaction={decorate,
  decoration={markings, mark=at position 0.5 with {\arrow{stealth}}}}] (-0.5,-0.5) -- (-0.5,1.4);
  \draw[thick] (-0.5,-1.4)--(-0.5,-0.5);
  \node[] at(0,-1.8) {$\tau_{1,b}$};
    \end{scope}
    \begin{scope}[shift = {(1.8,0)}]
        \draw[] (-1.4,-1.4) -- (1.4,-1.4) -- (1.4,1.4)--(-1.4,1.4)--cycle;
         \filldraw[black] (-0.5,-0.5) circle (1pt) node[right] {$z_1$};
         \filldraw[black] (0.5,0.5) circle (1pt) node[above] {$z_2$};
         \draw[thick, postaction={decorate,
  decoration={markings, mark=at position 0.5 with {\arrow{stealth}}}}] (-1.4,0.5) -- (0.5,0.5);
  \draw[thick] (0.5,0.5)--(1.4,0.5);
  \node[] at(0,-1.8) {$\tau_{2,a}$};
    \end{scope}
    \begin{scope}[shift = {(5.4,0)}]
        \draw[] (-1.4,-1.4) -- (1.4,-1.4) -- (1.4,1.4)--(-1.4,1.4)--cycle;
         \filldraw[black] (-0.5,-0.5) circle (1pt) node[right] {$z_1$};
         \filldraw[black] (0.5,0.5) circle (1pt) node[right] {$z_2$};
         \draw[thick, postaction={decorate,
  decoration={markings, mark=at position 0.5 with {\arrow{stealth}}}}] (0.5,-1.4) -- (0.5,0.5);
  \draw[thick] (0.5,0.5)--(0.5,1.4);
  \node[] at(0,-1.8) {$\tau_{2,b}$};
    \end{scope}
\end{tikzpicture}
\end{center}$ $\\
Note that there is a homeomorphism $$\psi: X_2(T) \to \left(T^2\setminus \{(0,0)\}\right)\times T^2,\quad (x,y) \mapsto (x-y,y)$$ which gives rise to an isomorphism $$\psi_*: \mathrm{PB}_2(T) \to F_2\times \mathbb Z^2,$$ where $F_2 = \langle a,b\rangle$ is the free group with two generators (Attention: we use the convention that words in $F_2$ and in the braid group are read from left to right). We use that the group $F_2$ is naturally isomorphic to $\pi_1(T^2\setminus \{(0,0)\},z_1-z_2)$ via an isomorphism that maps the loops $$z_1-z_2+(t,0)\mapsto a,\quad z_1-z_2 + (0,t) \mapsto b.$$ With the representation as in (\ref{def: torus braid representatives}) we have an isomorphism $$\langle \tau_{1,a},\tau_{1,b},\tau_{2,a},\tau_{2,b}\,|\, \text{Relations}\rangle\to F_2\times \mathbb Z^2,$$ where \begin{alignat}{2}&\tau_{1,a}\mapsto (a,0),&& \quad  \tau_{1,b}\mapsto (b,0),\\ & \tau_{2,a}\mapsto (a^{-1},a), &&\quad \tau_{2,b}\mapsto (b^{-1},b),\end{alignat}where $\mathbb Z^2 = \langle a,b\rangle^\mathrm{ab}$.
Now consider the abelianization $$\pi: F_2\times \mathbb Z^2\to \mathbb Z^2\times\mathbb Z^2$$
\begin{lemma}\label{lemma: braids with contractible strands}
    A braid has contractible strands if and only if for the corresponding element $(w,c)\in F_2\times \mathbb Z^2$ we have $$\pi(w,c) = 0.$$
\end{lemma}
\begin{proof}
    Firstly, the second strand is contractible if and only if $c = 0$. We now find a representation of the braid as a product of letters $\tau_{1,a}^{\pm 1}$ and $\tau_{1,b}^{\pm 1}$. The abelianization of $w$ in $\mathbb Z^2$ is exactly the homology class of the first strand, which is contractible if and only if this class is trivial.
\end{proof}
A path $\gamma:[0,1]\to T^2\setminus \{(0,0)\}$ is \emph{admissible} if it intersects $\{0\}\times \mathbb R/\mathbb Z$ and $\mathbb R \times \{0\}/\mathbb Z^2$ transversely. We now define a function $$\text{spell}: \left \{\gamma:[0,1]\to T^2\setminus \{(0,0)\}\,|\,\gamma \text{ admissible}\right \}\to F_2.$$
Consider the universal cover $$p: \mathbb R^2 \to \mathbb R ^2 /\mathbb Z^2$$ and a lift $\tilde\gamma$ of an admissible path $\gamma$. The word $\text{spell}(\gamma)$ is obtained in the following way. For every intersection of $\tilde \gamma$ with a horizontal grid line $\mathbb R\times \{k\}$ or vertical grid line $\{k\}\times \mathbb R$ for some $k\in \mathbb Z$, we write a letter in the obvious order given by $\tilde \gamma$. Namely, write $a$ if $\tilde \gamma$ intersects some vertical grid line from left to right and $a^{-1}$ if it intersects from right to left. Similarly, write $b$ if $\tilde \gamma$ intersects a horizontal grid line from bottom to top and write $b^{-1}$ if it intersects from top to bottom. The resulting word is defined to be $\text{spell}(\gamma)$. Clearly for a loop $ \gamma$ the word $\text{spell}(\gamma)$ is a possibly unreduced representative of the word in $F_2$ that corresponds to the class $[\gamma]\in \pi_1(T^2\setminus \{(0,0)\})$. For two loops $x_1,x_2:S^1\to T^2$ such that $x_1(t)\neq x_2(t)$ for all $t\in S^1$ we define $$\text{spell}(x_1,x_2) = \text{spell}(x_1-x_2).$$ Note that in general $\text{spell}(x_1,x_2)\neq \text{spell}(x_2,x_1)$.
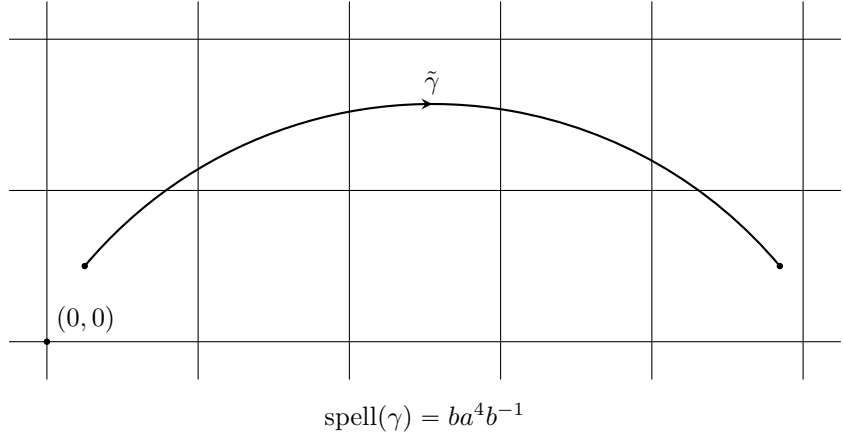
\begin{figure}[h!]
  \begin{tikzpicture}
    %vertical
    \draw[thin] (-4,-2.5) -- (-4,2.5);
    \draw[thin] (-2,-2.5) -- (-2,2.5);
    \draw[thin] (0,-2.5) -- (0,2.5);
    \draw[thin] (2,-2.5) -- (2,2.5);
    \draw[thin] (4,-2.5) -- (4,2.5);
    \draw[thin] (6,-2.5) -- (6,2.5);
    %horizontal
    \draw[thin] (-4.5,2)-- (6.5,2);
    \draw[thin] (-4.5,0)-- (6.5,0);
    \draw[thin] (-4.5,-2)-- (6.5,-2);
    %%%%%%%%%%%%%%%%%%%%%%%%%%%%%%%%%%%%
    \node[above right] at (-4,-2) {$(0,0)$};
    \draw[thick, postaction={decorate,
  decoration={markings, mark=at position 0.5 with {\arrow{stealth}}}}] (-3.5,-1) arc (-40:-140:-6) coordinate (endpoint) node[midway, above]{$\tilde \gamma$};
    \filldraw[black] (-4,-2) circle (1pt);
    \filldraw[black] (endpoint) circle (1pt);
    \filldraw[black] (-3.5,-1) circle (1pt);
    \node[] at (1,-3) {$\text{spell}(\gamma) = ba^4b^{-1}$};
  \end{tikzpicture} 
  \caption{The spelling of a path $\gamma$.}
\end{figure}
\begin{lemma}\label{lemma: conjugacy classes and spelling}
    Let $(x_1,x_2)$ and $(y_1,y_2)$ be loops in $X_2(T^2)$ such that  each loop $x_i$ and $y_i$ is contractible in $T^2$. If the loops are homotopic in $X_2(T^2)$ (that is the associated braids are isotopic) then $\text{spell}(x_1,x_2)$ and $\text{spell}(y_1,y_2)$ are conjugate. 
\end{lemma}
\begin{proof}
    First, suppose that $(x_1(0),x_2(0)) = (y_1(0),y_2(0))$. In this case, the lemma follows from the fact that the closures of two braids are isotopic if and only if they are conjugate to each other in the braid group at $(x_1(0),x_2(0))$. Indeed, $(\text{spell}(x_1,x_2),0)$ and $(\text{spell}(y_1,y_2),0)$ in $ F_2\times \mathbb Z^2$ are the corresponding braids in the braid group, and they are conjugate if and only if the first components are conjugate. In general, let $(\alpha_1,\alpha_2)$ be a path from $(x_1(0),x_2(0))$ to $(y_1(0),y_2(0))$. Note that $(\alpha_1\#y_1\#\alpha_1^{-},\alpha_2\#y_2\#\alpha_2^-)$ is a braid at $(x_1(0),x_2(0))$ whose closure is isotopic to the one of $(y_1,y_2)$. Clearly, the word $\text{spell}(\alpha_1\#y_1\#\alpha_1^{-},\alpha_2\#y_2\#\alpha_2^-)$ is conjugate to $\text{spell}(y_1,y_2)$. The general case now follows from the case $(x_1(0),x_2(0))= (y_1(0),y_2(0))$.
\end{proof}
We now discuss the entropy of a family of closed braids that will come up in the construction of eggbeater maps. For integers $n,m\in \mathbb Z\setminus\{0\}$ let $\beta_{m,n}$ denote the closed torus braid that is obtained by the closure of a braid given by the element $(a^mb^na^{-m}b^{-n},0)$. Note that this definition is independent of the choice of base points $(x_1,x_2)$, since by Lemma~\ref{lemma: conjugacy classes and spelling} the closures are isotopic if the spelling of the difference loops is conjugate.  
\begin{lemma}\label{lemma: entropy of the braids}
    For every $m,n\in \mathbb Z$ such that $|n|,|m|\geq 2$ we have $$h(\beta_{m,n})\geq  {1\over 5}\log(|mn|).$$ 
\end{lemma}
\begin{proof}
     Let $x_2 =(0,0)$ and $x_1 = (1/2,1/2)$ in $T^2 = \mathbb R^2/\mathbb Z^2$. The braid $$b_{m,n} = \tau_{1,a}^m\tau_{1,b}^n\tau_{1,a}^{-m}\tau_{1,b}^{-n}$$ corresponds to the word $(a^mb^na^{-m}b^{-n},0)$ and hence the closure of the braid is given by $\beta_{m,n}$. The associated mapping class of the braid is given by a point push map of the point $x_1$ along the loop $\gamma_{m,n} = B^{-n}\#A^{-m}\#B^n\#A^m$ where $A(t) = (1/2+t,1/2)$ and $B(t) = (1/2,1/2+t)$ for $t\in [0,1]$. We consider the loop $\gamma_{m,n}$ as a loop in $\pi_1(T^2\setminus \{x_2\},x_1)$ where it is primitive (i.e. not a power of another element) and filling. We now have that (see \cite{cL16} for an algorithm to compute the self-intersection number or see \cite{CR04} for a direct calculation) $$\iota(\gamma_{m,n}) = 4|mn|-2|m|-2|n|.$$
     After \cite{Do11} for the dilatation $\lambda_{m,n}$ of the mapping class of $\varphi_{m,n} = \operatorname{Push(\gamma_{m,n})}$ in $ \mathrm{MCG}(T^2\setminus\{x_1,x_2\})$ the following estimate holds $$\lambda_{m,n}\geq \sqrt[5]{\iota(\gamma_{m,n})+1}= \sqrt[5]{4|mn|-2|m|-2|n|+1}.$$ Therefore, $$h(\beta_{m,n}) =h(\varphi_{m,n})\geq \log(\lambda_{m,n}) \geq {1\over 5}\log(4|mn|-2|m|-2|n|+1)\geq {1\over 5}\log(|nm|),$$ where we use \cite{FS79} for the estimate with the dilatation and use that $|n|,|m|\geq 2$ in the last inequality.
\end{proof}
\subsection{Proof of Theorem~\ref{thm:entropy}: Eggbeater Maps, Braids of Orbits, and Action Estimates}\label{subsection: eggbeater}
In this section, inspired by Khanevsky's eggbeater map \cite{Kha22} on $T^2$ and eggbeater maps on higher-genus surfaces \cite{PS16,AGKKKPRRSSZ19}, we construct Hamiltonian diffeomorphisms $\phi_k$ generated by Hamiltonians $H_k$ on $T^2=\mathbb R^2/\mathbb Z^2$ equipped with the standard symplectic form $\omega = dx\wedge dy$. If not stated otherwise, we identify a point $z\in T^2$ with the unique tuple $(x,y) \in [0,1)^2$. We then choose a sequence of pairs of periodic orbits $\{x_1^{(k)},x_2^{(k)}\}\subset \mathcal P(H_{2k})$ with  $\mathcal B(H_{2k},x_1^{(k)},x_2^{(k)}) = \beta_{2k+1,2k+1}$ such that $\mathcal S(H_{2k},x_1^{(k)},x_2^{(k)})\to \infty$. We use these maps to prove Theorem~\ref{thm:entropy} by applying the braid stability result from Theorem~\ref{thm: braid stability}.

 Let us now define the Hamiltonians $H_k$. Consider the function $$s:[0,1]\to \mathbb R,\quad t\mapsto\begin{cases}- 4t, & t \in [0,1/4]\\4t-2& t\in[1/4,3/4] \\-4t+4, & t\in [3/4,1]
\end{cases}.$$ For convenience of notation we often will consider $s$ as a function on $S^1$ as it is periodic. For an integer $k$ define $$s_{k}(t) = (k+\varepsilon)s(t),$$ where $\varepsilon>0$ is chosen (depending on $k$) to be small enough.
Smooth $s_k$ near $t = 1/4,3/4$ by a $C^0$-small perturbation such that $$\int_0^1s_{k}(t)\,dt = 0,$$ the symmetry $$s_k(t+1/2) =-s_{k}(t)$$ still holds, and that the smoothening is supported on $\{t\in [0,1]\,|\, |s_k(t)|>k\}$. For convenience, we will view the function $s_k$ as a $1$-periodic function on $\mathbb R$.
Next, define the functions $$h_{k}(t) = \int_0^ts_{k}(r)dr$$ and let $H_{k,a}(x,y) = h_{k}(y),\quad H_{k,b}(x,y) = h_{k}(x)$ be two autonomous Hamiltonians on $T^2$. 
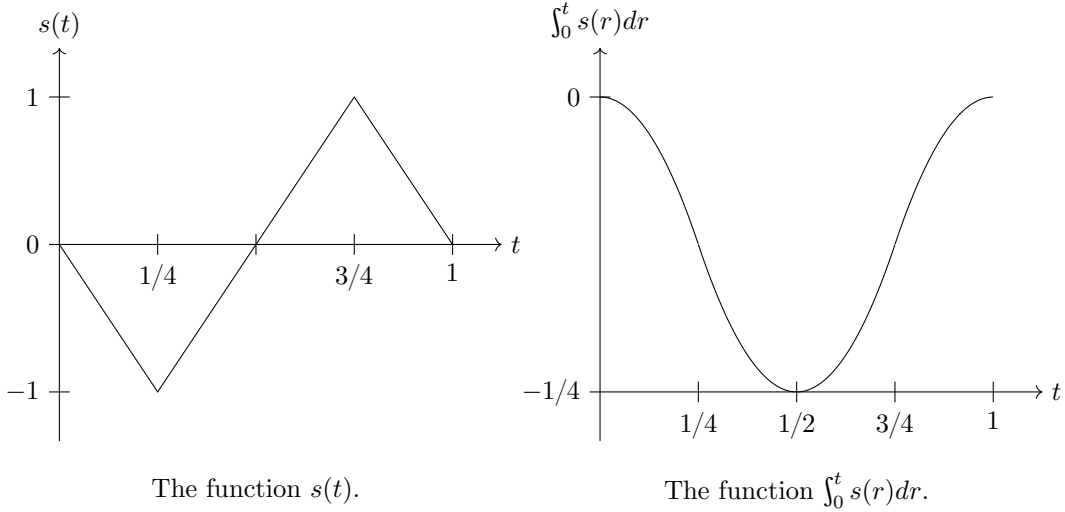
\begin{center}
\begin{tikzpicture}
    \begin{scope}[scale = 1.3, shift = {(-5.5,0)}]
    %grid
    \draw[->] (0,-2)--(0,2) node[above] {$s(t)$};
    \draw[->] (0,0) --(4.5,0) node[right]{$t$};
    \draw[] (-3pt,0) node[left] {$0$} -- (0,0);
    \draw[] (-3pt,1.5) node[left] {$1$}--(3pt,1.5);
    \draw[] (-3pt,-1.5) node[left] {$-1$}--(3pt,-1.5);
    \draw[] (1,3pt)--(1,-3pt) node[below]{$1/4$};
    \draw[] (2,3pt)--(2,-3pt);
    \draw[] (3,3pt)--(3,-3pt) node[below]{$3/4$};
    \draw[] (4,3pt) --(4,-3pt)  node[below]{$1$};
    %function
    \draw[](0,0)--(1,-1.5) -- (3,1.5) -- (4,0);
    \node[] at (2,-2.5) {\text{The function $s(t)$.}};
    \end{scope}
     \begin{scope}[scale = 1.3]
    %grid
    \draw[->] (0,-2)--(0,2)node[above] {$\int_0^ts(r)dr$};
    \draw[] (-3pt,1.5) node[left] {$0$}--(3pt,1.5);
   \node[left] at (-3pt,-1.5) {$-1/4$};
    %function
    \begin{scope}[shift = {(0,-1.5)}]
    \draw[->] (-3pt,0) --(4.5,0) node[right] {$t$};
    \draw[] (1,3pt)--(1,-3pt) node[below]{$1/4$};
    \draw[] (2,3pt)--(2,-3pt) node[below]{$1/2$};
    \draw[] (3,3pt)--(3,-3pt) node[below]{$3/4$};
    \draw[] (4,3pt) --(4,-3pt)  node[below]{$1$};
    \draw[samples=150,smooth,domain=0:1,variable=\t]
  plot ({\t},{3-1.5*\t^2});
  \draw[samples=150,smooth,domain=-1:1,variable=\t]
  plot ({2+\t},{1.5*\t*\t});
  \draw[samples=150,smooth,domain=-1:0,variable=\t]
  plot ({4+\t},{3-1.5*\t*\t});
  \end{scope}
  \node[] at (2,-2.5) {\text{The function $\int_0^ts(r)dr$.}};
    \end{scope}
\end{tikzpicture}
\end{center}Denote by $\phi_{k,a}^t$ and $\phi_{k,b}^t$ the flows induced by these Hamiltonians. They are given by $$\phi_{k,a}^t(x,y) = (x-ts_{k}(y),y)\quad \text{and}\quad \phi_{k,b}^t(x,y)=(x,y+ts_{k}(x)).$$ For $\varepsilon>0$ small enough, consider  $$\phi_k = \phi_{k,b}^{1}\circ\phi_{k,a}^{1}\circ \phi_{k,b}^{1}\circ \phi_{k,a}^{1}$$ generated by $H_k = H_{k,b}\#H_{k,a}\#H_{k,b}\#H_{k,a}$, where $\#$ denotes the concatenation in the time variable. We will show that $\phi_k$ is an example of a sequence promised in Theorem~\ref{thm:entropy}. 
For every homeomorphism $\phi$ of $T^2$ let $\tilde \phi$ denote the lift to the universal cover $\mathbb R^2$. For every fixed point $z\in [0,1]^2$ of $\phi_k$ we call the points $$z_1= \tilde\phi_{k,a}(z),\quad z_2 = \tilde \phi_{k,b}(\tilde \phi_{k,a}(z)),\quad  z_3 = \tilde \phi_{k,a}(\tilde \phi_{k,b}(\tilde \phi_{k,a}(z))),\quad  \text{and}\quad z_4 = \tilde \phi_k(z)$$ intermediate points. Note that $z = (x,y)$ is a fixed point of $\phi_k$ whose orbit is contractible if and only if $\tilde \phi_k(z) = z$. In this case there are two real numbers $R_a,R_b\in \mathbb R$ such that the intermediate points are given by $$z_1 =(x+R_a,y),\quad z_2 = (x+R_a,y+R_b),\quad z_3 = (x,y+R_b) \quad \text{and}\quad  z_4 = (x,y).$$ Indeed, we have $R_a = -s_{k}(y)$ and $R_b= -s_{k}(x)$ as follows from the proof of the following lemma. 
\begin{figure}[h!]
    \begin{tikzpicture}[scale = 1.25]

        \node[] at (-4.5,-2.3) {$(0,0)$}; 
        \filldraw[black] (-4,-2) circle (1pt);
        \draw[thin] (-4,-2.2)--(-4,-0.75);
        \draw[thin] (-4,2.2)--(-4,0.75);
        \draw[dotted] (-4,0.75)--(-4,-0.75);
        \draw[thin] (-3,-2.2)--(-3,-0.75);
        \draw[thin] (-3,2.2)--(-3,0.75);
        \draw[dotted] (-3,0.75)--(-3,-0.75);
        \draw[thin] (3,-2.2)--(3,-0.75);
        \draw[thin] (3,2.2)--(3,0.75);
        \draw[dotted] (3,0.75)--(3,-0.75);
        \draw[thin] (4,-2.2)--(4,-0.75);
        \draw[thin] (4,2.2)--(4,0.75);
        \draw[dotted] (4,0.75)--(4,0.3);
        \draw[dotted] (4,-0.75)--(4,-0.3);
        \node[] at (4,0) {$R_b$};
        %%%%%%%%%%%%%%%%%%%%%%%%%%%%%%%%
        \draw[thin] (-4.2,-2)--(-2.25,-2);
        \draw[thin] (2.25,-2)--(4.2,-2);
        \draw[dotted](-2.2,-1)--(-0.3,-1);
        \draw[dotted](2.2,-1)--(0.3,-1);
        \node[] at (0,-1) {$R_a$};
        \draw[thin] (-4.2,-1)--(-2.25,-1);
        \draw[thin] (2.25,-1)--(4.2,-1);
        \draw[dotted](-2.2,-2)--(2.2,-2);
        \draw[thin] (-4.2,1)--(-2.25,1);
        \draw[thin] (2.25,1)--(4.2,1);
        \draw[dotted](-2.2,1)--(2.2,1);
        \draw[thin] (-4.2,2)--(-2.25,2);
        \draw[thin] (2.25,2)--(4.2,2);
        \draw[dotted](-2.2,2)--(2.2,2);
        \filldraw[black] (-3.5,-1.5) node[below] {$z=z_0$} circle (1.5pt);
        \filldraw[black] (3.5,-1.5) node[below] {$z_1$} circle (1.5pt);
        \filldraw[black] (3.5,1.5) node[above] {$z_2$} circle (1.5pt);
        \filldraw[black] (-3.5,1.5) node[above] {$z_3$} circle (1.5pt);
        \draw[thick] (-3.5,-1.5)  -- (-2.25,-1.5);
        \draw[thick] (2.25,-1.5)--(3.5,-1.5);
        \draw[dotted,thick, postaction={decorate,
  decoration={markings, mark=at position 0.5 with {\arrow{stealth}}}}] (-2.2,-1.5)--(2.2,-1.5);
        \draw[thick] (3.5,-1.5)  -- (3.5,-0.75);
        \draw[thick] (3.5,0.75)--(3.5,1.5);
        \draw[dotted,thick, postaction={decorate,
  decoration={markings, mark=at position 0.5 with {\arrow{stealth}}}}] (3.5,-0.7)--(3.5,0.7);
        \draw[thick] (3.5,1.5)  -- (2.25,1.5);
        \draw[thick] (-2.25,1.5)--(-3.5,1.5);
        \draw[dotted,thick, postaction={decorate,
  decoration={markings, mark=at position 0.5 with {\arrow{stealth}}}}] (2.2,1.5)--(-2.2,1.5);
        \draw[thick] (-3.5,1.5)  -- (-3.5,0.75);
        \draw[thick] (-3.5,-0.75)--(-3.5,-1.5);
        \draw[dotted,thick, postaction={decorate,
  decoration={markings, mark=at position 0.5 with {\arrow{stealth}}}}] (-3.5,0.7)--(-3.5,-0.7);
        \draw[decorate,decoration={brace,amplitude=8pt}](-3.5,-1.5)--(3.5,-1.5);
         \draw[decorate,decoration={brace,amplitude=8pt}](3.5,1.5)--(3.5,-1.5);
    \end{tikzpicture}
    \caption{Dynamics of the intermediate points on $\mathbb R^2$.}
\end{figure}
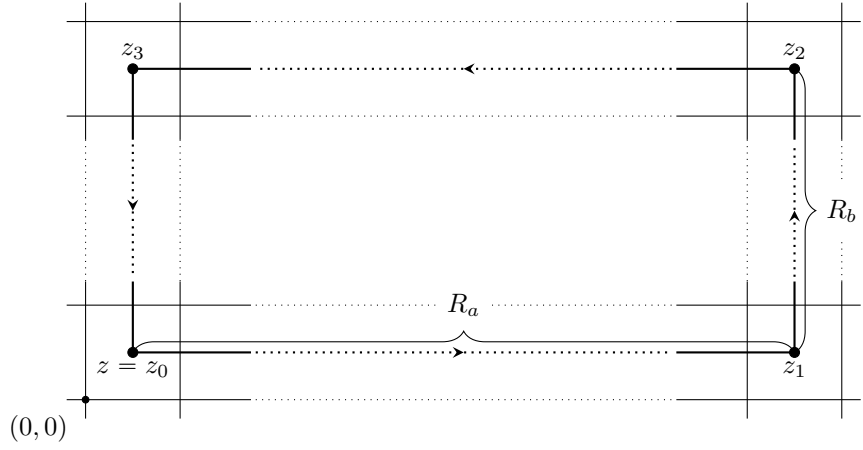
\begin{lemma}\label{lemma: condition on contractible orbits}
    A point $(x,y)\in T^2$ is a fixed point with contractible orbit if and only if $$ s_{k}(y) =-s_{k}(y-s_{k}(x))\quad \text{and} \quad s_{k}(x) = -s_{k}(x-s_{k}(y)).$$ 
\end{lemma} 
\begin{proof}
    The intermediate points of any point $(x_0,y_0)$ are given by $$(x_1,y_1)=(x_0-s_{k}(y_0),y_0),\quad (x_2,y_2)=(x_1,y_1+s_{k}(x_1)),\quad (x_3,y_3)= (x_2-s_{k}(y_2),y_2).$$ Moreover, $(x_0,y_0)$ is a fixed point with a contractible orbit if and only if $$(x_0,y_0)= (x_4,y_4)= (x_3,y_3+s_{k}(x_3)).$$ Importantly, we consider the equality in $\mathbb R^2$ and not merely on $T^2$. We find that $(x_2,y_2) = (x_0-s_k(y_0),y_0+s_k(x_0-s_{k}(y_0)))$ and, similarly, $(x_0,y_0) = (x_4,y_4) = (x_2-s_k(y_2),y_2+s_k(x_2-s_k(y_2)))$. Adding one equation to the other we find \begin{align}\label{eq:lemmma conditions on contractibel orbits 1}s_k(y_2) = -s_k(y_0) \quad \text{and}\quad  s_k(x_2-s_k(y_2))=-s_k(x_0-s_k(y_0))\end{align} and since $x_0=x_4 = x_3 = x_2-s_k(y_2)$ we find \begin{align}\label{eq:lemmma conditions on contractibel orbits 2}s_k(x_0) = -s_k(x_0-s_k(y_0)).\end{align}  Using (\ref{eq:lemmma conditions on contractibel orbits 2}) we now find $$s_k(y_2 ) = s_k(y_1+s_k(x_1))= s_k(y_0+s_k(x_0-s_k(y_0)))=s_k(y_0-s_k(x_0)),$$ where we use $y_0 = y_1$. Hence, combining this with (\ref{eq:lemmma conditions on contractibel orbits 1}) we find $$ s_k(y_0) = -s_k(y_0-s_k(x_0)).$$  On the other hand, assume that $s_k(x_0) = -s_k(x_0-s_k(y_0))$ and $s_k(y_0) = -s_k(y_0-s_k(x_0))$.  Then we have $s_k(x_0) = -s_k(x_2)$ since $x_2 = x_1 = x_0-s_k(y_0)$. Moreover, we then have $s_k(y_0) = -s_k(y_2)$ since $y_0-s_k(x_0) = y_1+s_k(x_2) = y_1+s_k(x_1) = y_2$. We now find
    \begin{align*}\begin{split}z_4& = (x_4,y_4) \\ & = (x_3,y_3+s_k(x_3))\\ & = (x_2-s_k(y_2),y_2+s_k(x_2-s_k(y_2))) \\ & =(x_1-s_k(y_2),y_1+s_k(x_1)+s_k(x_1+s_k(y_0)))\\ &= (x_1-s_k(y_2),y_1+s_k(x_2)+s_k(x_0))\\ & = (x_0-s_k(y_0)-s_k(y_2),y_0+s_k(x_2)+s_k(x_0))\\ & = (x_0,y_0)\end{split}\end{align*}
 which means that $(x_0,y_0)$ is a fixed point with contractible orbit. \end{proof}
 
 \begin{lemma}\label{lemma: fixed points are non-degenerate}
    For $k\in \mathbb Z_{\geq 1}$ large enough and $\varepsilon>0$ small enough, all fixed points $z = (x,y)$ with contractible orbits are non-degenerate and satisfy $|s_k(x)|< k$ and $|s_k(y)|< k$.
\end{lemma}
\begin{proof}
    We first show that any fixed point $(x,y)\in T^2$ with a contractible orbit has the property that for all intermediate points $(x_i,y_i)$ the coordinates $x_i,y_i$ lie in the part where $s_k$ is linear. Assume that this is not the case. Without loss of generality, we assume this fails, i.e. $x = x_0$ or $y=y_0$. Indeed, if this is not the case, it must hold for $x_2$ or $x_3$ (recall that $x_0 = x_3$, $x_1 = x_2$, $y_0 = y_1$, and $y_2 = y_3$). In this case, we consider the fixed point $\tilde z = \phi_{k,b}^1\circ \phi_{k,a}^1(z)$ which is given by $(x_2,y_2)$.\\
    This means that $|s_k(y)| \in [ k,k+3\varepsilon/2]$ or $|s_k(x-s_k(y))| \in [k,k+3\varepsilon/2]$ where we assume that the perturbation to a smooth approximation $s$ of the original function is small enough and done in a small enough neighborhood of non-smooth points.\\ First assume that $|s_k(y)| \in [ k,k+3\varepsilon/2]$. Since $s_k$ is periodic, we have $s_k(x-s_k(y)) \in s_k(x\pm [ 0,3\varepsilon/2])$. By Lemma~\ref{lemma: condition on contractible orbits} we must have $s_k(x) = -s_k(x-s_k(y))$ and hence we find $0<c\leq 3 \varepsilon/2$ such that $$|s_k(x)|= {1\over 2}|s_k(x)-s_k(x-s_k(y))| = {1\over 2}|s_k(x)-s_k(x\pm c)|.$$ Choosing $\varepsilon$ small enough, we have that $c>0$ is arbitrarily small. By continuity, for every $\delta>0$ we have that $|s_k(x)| = {1\over2}|s_k(x)-s_k(x\pm c)|<\delta$ when $\varepsilon>0$, and therefore also $c$, is small enough. But this contradicts $s_k(y) = -s_k(y-s_k(x))$ by choosing $\delta$ small enough. Indeed, we similarly have $$|s_k(y)| ={1\over 2}|s_k(y)-s_k(y-s_k(x))|$$ which tends to $0$ when $\varepsilon$ tends to $0$ since in this case $|s_k(x)|$ tends to $0$. This contradicts the assumption that $|s_k(y)|\geq k$. An analogous argument holds in the case $|s_k(x-s_k(y))|\in [k,k+3\varepsilon/2]$ by using $|s_k(x-s_k(y))| = |s_k(x)|$.\\
    Hence, we may assume that $
    x_i, y_i$ lie in the part of $[0,1]$ where $s_k$ is linear. A computation shows  we have \begin{align*}\begin{split}d_z\phi_k & = \begin{pmatrix}
        1& 0\\4\epsilon_4\lambda& 1
    \end{pmatrix}\begin{pmatrix}
        1& 4\epsilon_3\lambda\\0& 1
    \end{pmatrix}\begin{pmatrix}
        1& 0\\4\epsilon_2\lambda& 1
    \end{pmatrix}\begin{pmatrix}
        1& 4\epsilon_1\lambda\\ 0 & 1
    \end{pmatrix} \\ &= \begin{pmatrix}1+16\epsilon_2\epsilon_3\lambda^2
& \cdots\\ \cdots & 1+16(\epsilon_1\epsilon_2+\epsilon_3\epsilon_4+\epsilon_4\epsilon_1)\lambda^2+256\epsilon_1\epsilon_2\epsilon_3\epsilon_4\lambda^4\end{pmatrix}\end{split},\end{align*} where $\lambda=k+\varepsilon$ and $\epsilon_i\in \{\pm 1\}$ depends on $z$ and is given by $$\epsilon_1 = -\text{sign}(\dot s_k(y_0)),\quad \epsilon_2 = \text{sign}(\dot s_k(x_1)),\quad \epsilon_3 = -\text{sign}(\dot s_k(y_2)),\quad \epsilon_4 = \text{sign}(\dot s_k(x_3)).$$ Since $\det(d_z\phi_k) = 1$, $
1$ is an eigenvalue of $d_z\phi_k$ if and only if $\text{tr}(d_z\phi_k) = 2$. However, we have $$\text{tr}(d_z\phi_k) = 2+16(\epsilon_1\epsilon_2+\epsilon_2\epsilon_3+\epsilon_3\epsilon_4+\epsilon_4\epsilon_1)\lambda^2+256\epsilon_1\epsilon_2\epsilon_3\epsilon_4\lambda^4\neq2$$ for $k\geq 1$ and $\varepsilon>0$ small enough. Therefore, any fixed point with a contractible orbit is non-degenerate.
\end{proof} 
As a path of Hamiltonian diffeomorphisms from the identity to $\phi_k$ we consider the following isotopy \begin{align}\label{def: path of Hamiltonian diffoemorphisms}\{\phi_{k,b}^{t}\}\#\{\phi_{k,a}^{t}\}\#\{\phi_{k,b}^t\}\#\{\phi_{k,a}^t\}.\end{align}
We now characterize the fixed points with contractible orbits using the following data. Firstly, for every such fixed point $z$ of $\phi_k$ the \emph{winding numbers} $(m,n)\in \mathbb Z^2$ are defined by \begin{align}\label{eq: coefficients in spelling}
     m= \lfloor x-s_{k}(y)\rfloor\quad \text{and}\quad n = \lfloor y-s_{k}(x)\rfloor,
 \end{align} where $\lfloor \cdot \rfloor$ denotes the floor function. Since $|s_k|$ is bounded by $k+\varepsilon$, the winding numbers of any fixed points lie in $m,n\in [-k,k]$. Indeed, by Lemma~\ref{lemma: fixed points are non-degenerate} we have $|s_k(y)| < k$ and $x\in [0,1)$ and therefore $x-s_k(y)\in (-k,k+1)$ which means $|m|\leq k$. Likewise, we find $|n|\leq k$.  Next, we define the \emph{branching index} $\epsilon = \epsilon(z)\in \{\pm 1 \}^4$ by   $$\epsilon_1 = -\text{sign}(\dot s_k(y_0)),\quad \epsilon_2 = \text{sign}(\dot s_k(x_1)),\quad \epsilon_3 = -\text{sign}(\dot s_k(y_2)),\quad \epsilon_4 = \text{sign}(\dot s_k(x_3)),$$ where $z_i = (x_i,y_i)$ are the intermediate points. The branching index captures whether the flow locally acts as a positive or negative shear at the intermediate points. 

We now study the dynamics of fixed points $z=(x,y)$ such that $s(x)\neq 0\neq s(y)$. Let $m,n\in [-k,k]$ and fix a branching index $\epsilon$.
 Denote by $\bar z_i$ the shifted intermediate points in $[0,1]^2$ given by $$\bar z_0 = z_0\quad \bar z_1= z_1-(m,0),\quad \bar z_2 = z_2-(m,n),\quad \bar z_3 = z_3-(0,n).$$ The point $z_0$ is a fixed point of $\phi_k$ with contractible orbit and winding number $(m,n)$ if and only if it solves the following system of equations. \begin{align}\label{eq: system of equations 1}
        \begin{split}
            \overline z_1 = \overline z_0 + (-s_{k}(\overline y_0)-m,0),\quad \overline z_2 = \bar z_1 + (0,s_{k}(\overline x_1)-n)   \\
            \overline z_3 = \overline z_2 + (-s_{k}(\overline y_2)+m,0),\quad \overline z_0 = \bar z_3 + (0,s_{k}(\overline x_3)+n).
        \end{split}
    \end{align}  This can be written as \begin{align*}
        \begin{split}\overline z_1 = \begin{pmatrix}
            1& 4\epsilon_1\lambda\\ 0 & 1
        \end{pmatrix}\overline z_0 + \nu_{1}, \quad \overline z_2 = \begin{pmatrix}
            1&0\\ 4\epsilon_2\lambda& 1
        \end{pmatrix}\overline z_1+ \nu_{2}, \\ 
        \overline z_3= \begin{pmatrix}
            1& 4\epsilon_3\lambda\\ 0 & 1
        \end{pmatrix}\overline z_2 + \nu_{3}, \quad \overline z_0 = \begin{pmatrix}
            1&0\\ 4\epsilon_4\lambda& 1
        \end{pmatrix} \overline z_3 + \nu_{4},
        \end{split}
    \end{align*}
    where $\lambda = k+\varepsilon$ and $\nu_{i}$ are shift parameters given by \begin{align*}
        \begin{aligned}
            &\nu_{1} = \begin{cases}
                (-m,0), & \epsilon_1 = 1,s(y) <0\\ (-m+2 \lambda,0), & \epsilon_1 = -1\\
                (-m-4\lambda ,0), & \epsilon_1 = 1,s(y)>0
            \end{cases},&\nu_{2} = \begin{cases}
                (0,-n), & \epsilon_2 = -1,s(x) >0\\ (0,-n-2 \lambda), & \epsilon_2 = 1\\
                (0,-n+4\lambda ), & \epsilon_2 = -1,s(x)<0
            \end{cases},\\
            &\nu_{3} = \begin{cases}
                (m,0), & \epsilon_3 = 1,s(y) >0\\ (m+2 \lambda,0), & \epsilon_3 = -1\\
                (m-4\lambda ,0), & \epsilon_3 = 1, s(y)<0
            \end{cases},&\nu_{4} = \begin{cases}
                (0,n), & \epsilon_4 = -1,s(x) <0\\ (0,n-2 \lambda), & \epsilon_4 = 1\\
                (0,n+4\lambda ), & \epsilon_4 = -1, s(x)>0
            \end{cases},
        \end{aligned}
    \end{align*}
    where the cases depend on the direction of flow and slope of the speed function at the intermediate points, and we use that $s_k(x_1) =-s_k(x_0) = -s_k(x)$ and $s_k(y_2) = -s_k(y_0) =-s_k(y)$. One can formulate this with closed forms by \begin{align*}
        \begin{split}
            &\nu_{1} = (-m-\lambda[\text{sign}(s(y))(1+\epsilon_1)+2\epsilon_1],0), \, \nu_{2} = (0,-n-\lambda[\text{sign}(s(x))(1-\epsilon_2)+2\epsilon_2]),\\
            &\nu_{3} = (m+\lambda[\text{sign}(s(y))(1+\epsilon_3)-2\epsilon_3],0), \, \nu_{4} = (0,n+\lambda[\text{sign}(s(x))(1-\epsilon_4)-2\epsilon_4]).
        \end{split}
    \end{align*} 
    Any fixed point then satisfies the equation \begin{align}\label{eq: accumulated equation} (A-id)z = -\nu,\end{align} where $$A = \begin{pmatrix}
        1&0 \\  4\epsilon_4\lambda & 1
    \end{pmatrix} \begin{pmatrix}
        1& 4\epsilon_3\lambda \\ 0 & 1
    \end{pmatrix}\begin{pmatrix}
        1&0 \\  4\epsilon_2\lambda & 1
    \end{pmatrix} \begin{pmatrix}
        1& 4\epsilon_1\lambda \\ 0 & 1
    \end{pmatrix}$$ and $\nu$ is the accumulated shift parameter $$\nu = \nu_4+\begin{pmatrix}
        1 & 0 \\ 4\epsilon_4\lambda & 1
    \end{pmatrix} \left [\nu_3+\begin{pmatrix}
        1 &  4\epsilon_3\lambda\\0 & 1
    \end{pmatrix}\left [\nu_2+\begin{pmatrix}
        1 & 0\\4\epsilon_2\lambda & 1
    \end{pmatrix} \nu_1\right] \right].$$
In the proof of Lemma~\ref{lemma: fixed points are non-degenerate}, we already showed that $A$ does not have $1$ as an eigenvalue and hence $A-id$ is invertible. Hence, there is a unique solution $z$ to (\ref{eq: accumulated equation}). The solution to this equation is approximated by $$x = {1\over 2}+{\text{sign}(s(x)) (1-\epsilon_4)\over 4}-{\epsilon_4n \over 4k}+{1\over k}\left ({(\epsilon_1+\epsilon_3)\epsilon_4\over 16}\left({m\over k}+\text{sign}(s(y))\right )+{\epsilon_4\text{sign}(s(y))\over 8}\right)+O\left (1\over k^2\right)$$ $$y = {1\over 2}+{\text{sign}(s(y)) (1+\epsilon_1)\over 4}+{\epsilon_1m \over 4k}+{1\over k}\left ({(\epsilon_2+\epsilon_4)\epsilon_1\over 16}\left({n\over k}+\text{sign}(s(x))\right )-{\epsilon_1\text{sign}(s(x))\over 8}\right)+O\left (1\over k^2\right),$$ where we use that $\varepsilon<{1\over k}$ is chosen small enough for each $k$ as can be verified with a Python script. The point $(x,y)$ is a fixed point of $\phi_k$ exactly if it satisfies $(x,y)\in [0,1)^2$.

Lastly, note that for every fixed point $z$ of $\phi_k$ we must either have $$s(x)=s(y) = 0\quad \text{or}\quad s(x) \neq 0 \neq s(y).$$ Indeed, if $s(x) =0$ then Lemma~\ref{lemma: condition on contractible orbits} asserts that $s(y) = -s(y)$ and hence also must be $0$ and analogously if the role of $x$ and $y$ are swapped. We formulate the result of this discussion in the following lemma. 

\begin{lemma}\label{lemma: classify fixed points}
Any fixed point $z= (x,y)$ of $\phi_k$ with contractible orbits either satisfies $s(x) = s(y) = 0$ or solves (\ref{eq: accumulated equation}).
\end{lemma}
Let $z$ be a fixed point of $\phi_k$ and let $\bar z_i$ denote its intermediate points projected to the torus. The intermediate points pass from $\bar z_i$ to $\bar z_{i+1}$ by either an action of $\phi_{k,a}$ or $\phi_{k,b}$. At these points, the differential $d_{\bar z_{i}}(\phi_{k,a})$ or $d_{\bar z_{i}}\phi_{k,b}$ can take two forms depending on which branch of $s$ the coordinate $\bar y_i$ or $\bar x_i$ lies determined by the branching index at $\epsilon_i$. Namely, in the standard trivialization of $TT^2$ we have $$d_{\bar z_i}\phi_{k,a} = \begin{pmatrix}
    1 & \epsilon_i4\lambda\\ 0 & 1
\end{pmatrix},\quad d_{\bar z_i}\phi_{k,b} = \begin{pmatrix}
    1 & 0 \\ \epsilon_i 4\lambda & 1
\end{pmatrix}.$$ 
The following lemma is somewhat computation-heavy, and we will prove it in Appendix~\ref{section: appendix}. 
\begin{lemma}\label{lemma: index of orbits}
    Let $z$ be a fixed point of $\phi_k$ with $\epsilon = \epsilon(z) \in \{\pm 1\}^4$.
    For $k$ large enough we have $$\mu(z) = {(\epsilon_1+\epsilon_3)-(\epsilon_2+\epsilon_4)\over 2}.$$
\end{lemma}
We will estimate the stability radius of braids by action estimates. For this, the following action formula will be essential. 
\begin{lemma}\label{lemma: action}
    For every fixed point $z =(x,y)$ with contractible orbit such that $s(x)\neq 0 \neq s(y)$, with winding numbers $(m,n)$ and branching index $\epsilon = \epsilon(z)\in \{\pm 1\}^4$ the action of its orbit $\gamma$ is given by $$\mathcal A_{H_k}(\gamma) = -mn+{k\over 4}(\mu(z)-2)+{\epsilon_1+\epsilon_3\over 8}\left ({m^2\over k}-2|m|\right)-{\epsilon_2+\epsilon_4\over 8}\left ({n^2\over k}-2|n|\right)-{|m|+|n|\over 2}+ O(1).$$
\end{lemma}
\begin{proof}
    Let $z_i = (x_i,y_i)\in \mathbb R^2$ be the intermediate points of $z$ and let $\bar z_i = (\bar x_i,\bar y_i)$ denote the intermediate points projected to $T^2$. We first show the following formula for the action, and then approximate $s(x)$ and $s(y)$ \begin{align}\label{eq: Action estimate 1}
   \mathcal A_{H_k}(\gamma) = -\lambda^2 s(x)s(y)+{\lambda \over 8}\left [2\mu(z) - (\epsilon_1+\epsilon_3)s(y)^2+(\epsilon_2+\epsilon_4)s(x)^2-4\right] + O(1),     
    \end{align} where $\lambda = (k+\varepsilon)$. The action is given by $$\mathcal A_{H_k}(\gamma) = h_{k}(y_0)+ h_{k}(x_1)+ h_{k}(y_2) + h_{k}(x_3) -(x_1-x_0)(y_2-y_1),$$ where $(x_1-x_0)(y_2-y_1)$ accounts for the area of the capping, and the rest comes from the Hamiltonian part. First, note that the area term is given by $$(x_1-x_0)(y_2-y_1) = s_k(x_1)(-s_k(y_0)) = s_k(x_0)s_k(y_0)= \lambda^2s(x_0)s(y_0).$$ We have that \begin{align*}
        \bar y_0 = \begin{cases}
            (-\epsilon_1s(\bar y_0)+1-\epsilon_1)/4, & \bar y_0\in [0,1/2]\\ 
            (-\epsilon_1s(\bar y_0)+3+\epsilon_1)/4, & \bar y_0 \in [1/2,1]
        \end{cases},   \bar x_1 = \begin{cases}
            (\epsilon_2s(\bar x_1)+1+\epsilon_2)/4, & \bar x_1\in [0,1/2]\\ 
            (\epsilon_2s(\bar x_1)+3-\epsilon_2)/4, & \bar x_1 \in [1/2,1]
        \end{cases}
    \end{align*}\begin{align*}
        \bar y_2 = \begin{cases}
            (-\epsilon_3s(\bar y_2)+1-\epsilon_3)/4, &\bar y_2\in [0,1/2]\\ 
            (-\epsilon_3s(\bar y_2)+3+\epsilon_3)/4, & \bar y_2 \in [1/2,1]
        \end{cases},   \bar x_3 = \begin{cases}
            (\epsilon_4s(\bar x_3)+1+\epsilon_4)/4, & \bar x_3\in [0,1/2]\\ 
            (\epsilon_4s(\bar x_3)+3-\epsilon_4)/4, & \bar x_3 \in [1/2,1]
        \end{cases}
    \end{align*}
    We have that $$h_k(t) = \begin{cases}
        -2\lambda t^2, & t\in [0,1/4]\\ 
        \lambda(2(t-1/2)^2-1/4),&t\in  [1/4,3/4]\\
        -2\lambda (t-1)^2, & t \in [3/4,1]
    \end{cases}.$$
     Note that $$\epsilon_1 = \begin{cases}
         -1, & y_0 \in [1/4,3/4]\\ 1,& \text{otherwise}
     \end{cases}, \quad \epsilon_2 = \begin{cases}
         1, & x_1 \in [1/4,3/4]\\ -1,& \text{otherwise}
     \end{cases},$$ $$\epsilon_3 = \begin{cases}
         -1, & y_2 \in [1/4,3/4]\\ 1,& \text{otherwise}
     \end{cases}, \quad \epsilon_4 = \begin{cases}
         1, & x_3 \in [1/4,3/4]\\ -1,& \text{otherwise}
     \end{cases}.$$ Going through all cases, we find the closed form for the Hamiltonian terms\begin{align*}
         h_k(\bar y_0) = -\lambda(\epsilon_1s(\bar y_0)^2+1-\epsilon_1)/8, \quad h_{k}(x_1) = \lambda(\epsilon_2s(\bar x_1)^2-1-\epsilon_2)/8
     \end{align*}
     \begin{align*}
         h_k(\bar y_2) = -\lambda(\epsilon_3s(\bar y_2)^2+1-\epsilon_3)/8, \quad h_{k}(\bar x_3) = \lambda(\epsilon_4s(\bar x_3)^2-1-\epsilon_4)/8
     \end{align*}
     Using that $s(\bar y_2) = -s(\bar y_0)$ and $s(\bar x_0) = s(\bar x_3) = -s(\bar x_1)$ and additionally using $\mu(z) = ((\epsilon_1+\epsilon_3)-(\epsilon_2+\epsilon_4))/2$ as in Lemma~\ref{lemma: index of orbits} shows (\ref{eq: Action estimate 1}). \\ We now estimate the values of $s(x)$ and $s(y)$. By Lemma~\ref{lemma: classify fixed points}, we either have $s(x) = s(y) = 0$ or $s(x) \neq 0 \neq s(y)$. In the first case, formula (\ref{eq: Action estimate 1}) says that $$\mathcal A_{H_k}(\gamma) = {\lambda\over 8}(2\mu(z)-4) +O(1) = {k\over 4}(\mu(z)-2)+O(1),$$ where we use $\lambda = k+\varepsilon$. The winding numbers are $(m,n) = 0$ in this case and hence we arrive at the promised expression in the lemma. Now assume $s(x)\neq 0 \neq s(y)$. The fixed point solves the system of equations (\ref{eq: system of equations 1}) and hence also (\ref{eq: accumulated equation}), given by $$(A-id) z = -\nu.$$ Solving this system of equations, with a Python script, one can verify that we have $$x = {1\over 2}+{\text{sign}(s(x)) (1-\epsilon_4)\over 4}-{\epsilon_4n \over 4k}+{1\over k}\left ({(\epsilon_1+\epsilon_3)\epsilon_4\over 16}\left({m\over k}+\text{sign}(s(y))\right )+{\epsilon_4\text{sign}(s(y))\over 8}\right)+O\left (1\over k^2\right)$$ $$y = {1\over 2}+{\text{sign}(s(y)) (1+\epsilon_1)\over 4}+{\epsilon_1m \over 4k}+{1\over k}\left ({(\epsilon_2+\epsilon_4)\epsilon_1\over 16}\left({n\over k}+\text{sign}(s(x))\right )-{\epsilon_1\text{sign}(s(x))\over 8}\right)+O\left (1\over k^2\right),$$ where we use that $\varepsilon<{1\over k}$ is chosen small enough for each $k$. Plugging this back into $s(\cdot)$ we find \begin{align*}\begin{split}&s(x) = -{n\over k}+{1\over k}\left ({(\epsilon_1+\epsilon_3)\over 4}\left({m\over k}+\text{sign}(s(y))\right )+{\text{sign}(s(y))\over 2}\right)+O\left (1\over k^2\right),\\ & s(y) = -{m\over k}- {1\over k}\left ({(\epsilon_2+\epsilon_4)\over 4}\left({n\over k}+\text{sign}(s(x))\right )-{\text{sign}(s(x))\over 2}\right)+O\left (1\over k^2\right).\end{split}\end{align*} Using that $\lambda = k+O\left (1\over k\right )$ we calculate $$-\lambda^2 s(x)s(y) = -mn+{\epsilon_1+\epsilon_3\over 4}\left ({m^2\over k}-|m|\right )-{\epsilon_2+\epsilon_4\over 4}\left ({n^2\over k}-|n|\right)-{|m|+|n|\over 2}+O(1),$$ where we use that $\text{sign}(s(y))m = -|m|$ and $\text{sign}(s(x))n = -|n|$. The Hamiltonian term of the action is then given by $${\lambda \over 8}(2\mu(z)-(\epsilon_1+\epsilon_3)s(y) ^2+(\epsilon_2+\epsilon_4)s(x)^2-4)= {k\over 4}(\mu(z)-2)-{\epsilon_1+\epsilon_3\over 8}{m^2\over k}+{\epsilon_2+\epsilon_4\over 8}{n^2\over k}+O(1).$$ Using both terms and (\ref{eq: Action estimate 1}) we find $$\mathcal A_{H_k}(\gamma) = -mn+{k\over 4}(\mu(z)-2)+{\epsilon_1+\epsilon_3\over 8}\left ({m^2\over k}-2|m|\right)-{\epsilon_2+\epsilon_4\over 8}\left ({n^2\over k}-2|n|\right)-{|m|+|n|\over 2}+ O(1).$$
\end{proof}
We now consider the construction for $2k$ in place of $k$ and show that for a choice of $x_1^{(k)},x_2^{(k)}\in \mathcal P(H_{2k})$ and any suitable choice of almost complex structure $J$ we have  $$\mathcal E(H_{2k},J,x_1^{(k)},x_2^{(k)})\to \infty\quad (k\to \infty).$$ Namely, let $x_1^{(k)},x_2^{(k)}\in \mathcal P(H_{2k})$ be the contractible loops with winding numbers $(k,k)$ and $(-k,-k)$ and branching indices $\epsilon(x_1) = (1,-1,1,-1)$ and $\epsilon(x_2) = (-1,1,-1,1)$ so that $\mu(x^{(k)}_1) = 2$ and $\mu(x_2^{(k)}) =-2$. Indeed, solving (\ref{eq: accumulated equation}) in this case ensures that these fixed points exist. For better readability, we drop the superscript $(k)$. In the following, we give an obstruction to the existence of Floer isotopies with one end in $\{x_1,x_2\}$ and another pair $\{y_1,y_2\}\subset \mathcal P({H_{2k}})$ relying on the winding numbers of the orbits $y_1$ and $y_2$; this argument is purely topological. We then estimate the energy of the Floer isotopies using the action differences of the orbits $x_i$ and $y_i$ for $i =1,2$. 
\begin{lemma}\label{lemma: winding numbers}
    Let $k$ be large enough and $y_1,y_2\in \mathcal P({H_{2k}})$ contractible orbits with winding numbers $(m_i,n_i)$ such that there exists a Floer isotopy $(u_1,u_2)$, where $u_i$ connects $x^{(k)}_i$ and $y_i$ for $i=1,2$. Then we have $$m_1-m_2 = n_1-n_2= \pm 2k.$$
\end{lemma}
\begin{proof}
    We apply Lemma~\ref{lemma: conjugacy classes and spelling} to the loops $\{x_1^{(k)},x_2^{(k)}\}$ and $\{y_1,y_2\}$ and have that an isotopy can only exist if the words $\text{spell}(x_1,x_2)$ and $\text{spell}(y_1,y_2)$ are conjugate in $F_2$. The fixed points $z= x_1^{(k)}(0)$ and $w = x_2^{(k)}(0)$ of $\phi_{2k}$ and their intermediate points satisfy the following dynamics. $$z_0,w_2\in [0,1/2]\times [0,1/2]+\mathbb Z^2,\quad z_1,w_3\in [1/2,1]\times [0,1/2]+\mathbb Z^2,$$ $$ z_2,w_0\in [1/2,1]\times [1/2,1]+ \mathbb Z^2,\quad z_3,w_1\in [0,1/2]\times [1/2,1]+\mathbb Z^2.$$  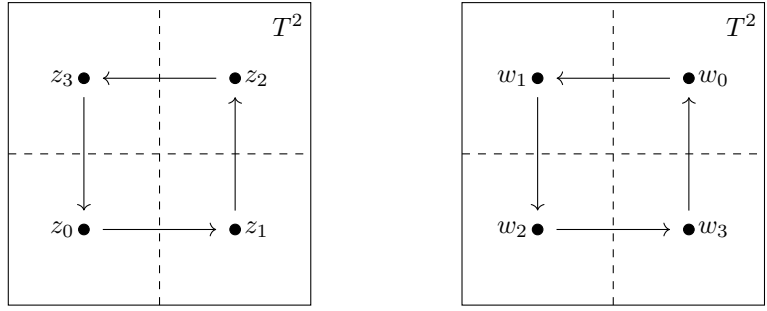
\begin{figure}[h!]
            \begin{tikzpicture}
                \begin{scope}[shift = {(-3,0)}]
                    \draw[] (-2,-2)--(2,-2)--(2,2)--(-2,2)--cycle;
                    \draw[thin, dashed] (0,-2)--(0,2);
                    \draw[thin,dashed] (-2,0)--(2,0);
                    \filldraw[black] (-1,-1) circle (2pt) node[left] {$z_0$};
                    \filldraw[black] (1,-1) circle (2pt) node[right] {$z_1$};
                    \filldraw[black] (1,1) circle (2pt) node[right] {$z_2$};
                    \filldraw[black] (-1,1) circle (2pt) node[left] {$z_3$};
                    \draw[->] (-0.75,-1)--(0.75,-1);
                    \draw[->] (1,-0.75)--(1,0.75);
                    \draw[->] (0.75,1)-- (-0.75,1);
                    \draw[->] (-1,0.75) --(-1,-0.75);
                    \node[] at (1.7,1.7) {$T^2$};
                \end{scope}
                 \begin{scope}[shift = {(3,0)}]
                    \draw[] (-2,-2)--(2,-2)--(2,2)--(-2,2)--cycle;
                    \draw[thin, dashed] (0,-2)--(0,2);
                    \draw[thin,dashed] (-2,0)--(2,0);
                    \filldraw[black] (-1,-1) circle (2pt) node[left] {$w_2$};
                    \filldraw[black] (1,-1) circle (2pt) node[right] {$w_3$};
                    \filldraw[black] (1,1) circle (2pt) node[right] {$w_0$};
                    \filldraw[black] (-1,1) circle (2pt) node[left] {$w_1$};
                    \draw[->] (-0.75,-1)--(0.75,-1);
                    \draw[->] (1,-0.75)--(1,0.75);
                    \draw[->] (0.75,1)-- (-0.75,1);
                    \draw[->] (-1,0.75) --(-1,-0.75);
                    \node[] at (1.7,1.7) {$T^2$};
                \end{scope}
            \end{tikzpicture}
            \caption{The dynamics of intermediate points $z_0,z_1,z_2,$ and $z_3$ and $w_0,w_1,w_2,$ and $w_3$.}
            \label{fig: dynamics of intermediate points}
        \end{figure} \\
        With this in mind and the fact that the winding numbers are $(\pm k,\pm k)$, we can conclude that $$\text{spell}(x_1^{(k)},x_2^{(k)}) = a^{2k+1}b^{2k+1}a^{-2k-1}b^{-2k-1}.$$ Note that for the orbits $y_1,y_2$ there exists $m,n\in \mathbb Z$ such that $$\text{spell}(y_1,y_2) = a^mb^na^{-m}b^{-n}.$$ This word is only conjugate to $a^{2k+1}b^{2k+1}a^{-2k-1}b^{-2k-1}$ if $m = n = \pm (2k+1)$ since in $F_2$ words are conjugate if and only if their cyclically reduced words agree up to cyclic permutations. Let us consider the case $m = n= 2k+1$; the other case is analogous. Let now $z= (x^z,y^z) = y_1(0)$ and $w= (x^w,y^w) = y_2(0)$ be the corresponding fixed points. Denote by $z_i = (x_i^z,y_i^z)$ and $w_i = (x_i^w,y_i^w)$ the intermediate points of $z$ and $w$. Then we have $$ m= \lfloor x_1^z-x_1^w\rfloor-\lfloor x_0^z-x_0^w\rfloor\quad \text{and}\quad n = \lfloor y_2^z-y_2^w\rfloor-\lfloor y_0^z-y_0^w\rfloor.$$ \underline{Observation: } $x_0^z<{1\over 2}<x_0^w$ and $y_0^z<{1\over 2}<y_0^w$. \begin{proof}
           First assume that $s_{2k}(x_0^z)\geq 0$. We find the following contradiction \begin{align*}
               \begin{split}
               n= 2k+1 & =  \lfloor y_2^z-y_2^w\rfloor-\lfloor y_0^z-y_0^w\rfloor\\ & =  \lfloor y_0^z-s_{2k}(x_0^z)-(y_0^w-s_{2k}(x_0^w))\rfloor-\lfloor y_0^z-y_0^w\rfloor
               \\ & \leq  \lfloor y_0^z-y_0^w+s_{2k}(x_0^w))\rfloor-\lfloor y_0^z-y_0^w\rfloor\\ & \leq \lfloor s_{2k}(x_0^w)\rfloor+1\\& \leq 2k,
               \end{split}
           \end{align*}
           where we use Lemma~\ref{lemma: fixed points are non-degenerate} for the last inequality. This shows that $x_0^z<{1\over 2}$. Assuming that $s_{2k}(x_0^w)\leq 0$, a similar argument shows that $x_0^w>{1\over 2}$. The proof that $y_0^z<{1\over 2}<y_0^w$ is analogous. This concludes the proof of the observation.
        \end{proof} This means that $ m= \lfloor x_1^z-x_1^w\rfloor+1$ and $n = \lfloor y_2^z-y_2^w\rfloor+1.$ Moreover, the points $z$ and $w$ follow the same dynamics as in Figure~(\ref{fig: dynamics of intermediate points}) which means that $$x_1^z \in [m_z+1/2,m_z+1)\quad \text{and}\quad x_1^w \in [m_w,m_w+1/2)$$ $$y_2^z \in [n_z+1/2,n_z+1)\quad \text{and}\quad y_2^w \in [n_w,n_w+1/2),$$ where $(m_z,n_z)$ and $(m_w,n_w)$ are the winding numbers of $z$ and $w$. This means that $x_1^z-x_1^w\in(m_z-m_w,m_z-m_w+1)$ and $y_2^z-y_2^w\in (n_z-n_w,n_z-n_w+1)$ and hence $2k+1=m = m_z-m_w+1$ and $2k+1=n = n_z-n_w+1$.
    This concludes the proof of the Lemma.
\end{proof}
We now estimate the energy of Floer isotopies with ends in $\{x_1,x_2\}$. At this point, the fact that $\mathcal E(H,J,x)$ relies on the energy of Floer isotopies rather than on the energy of the single strands is essential since there are choices of orbits $y_1,y_2\in \mathcal P(H_{2k})$ that are not obstructed by Lemma~\ref{lemma: winding numbers} where one of the action difference $\mathcal A_{H_{2k}}(x_1)-\mathcal A_{H_{2k}}(y_1)$ or $\mathcal A_{H_{2k}}(y_2)-\mathcal A_{H_{2k}}(x_2)$ is small but the other is of order $k$.  \begin{lemma}\label{lemma: stability radius computation}
   For $k$ large enough and any generic almost complex structure $J$ we have $$\mathcal E(H_{2k},J,x_1^{(k)},x_{2}^{(k)})\geq {1\over 8}k+O(1) $$
\end{lemma}
\begin{proof}
Suppose that $y_1,y_2\in \mathcal P(H_{2k})$ are such that there exists a Floer isotopy $(u_1,u_2)$ with $\mu_{\operatorname{max}}(u_1,u_2) = 1$ connecting $x_i$ and $y_i$ for $i=1,2$. Purely for topological reasons, by Lemma~\ref{lemma: winding numbers}, we must have that $y_1$ and $y_2$ are contractible loops with winding numbers $(k+p,k+q)$ and $(-k+p,-k+q)$ or $(-k+p,-k+q)$ and $(k+p,k+q)$ where $p,q\in \mathbb Z$ are integers with $|p|,|q| \leq k$. Note that an orbit $y$ with fixed index $\mu(y)$ has the same action when the winding number is $(k+p,k+q)$ as when it is $(-k-p,-k-q)$. Since we consider all integer values of $p,q\in [-k,k]$, we may and do only consider the case when the winding numbers of $y_1$ and $ y_2$ are $(k+p,k+q)$ and $(-k+p,-k+q)$ since we obtain the other case by inverting $p$ and $q$. There are two cases to consider: $(\mu(y_1),\mu(y_2)) = (1,-2)$ or $(\mu(y_1),\mu(y_2)) = (2,-1)$. Note that we need not consider the case  $(\mu(y_1),\mu(y_2)) = (1,-1)$ as there cannot exist a Floer isotopy between $\{x_1,x_2\}$ and $\{y_1,y_2\}$; since the index does not increase along Floer cylinders, we must either have $\mu(x_i)\leq \mu(y_i)$ or $\mu(y_i)\leq \mu(x_i)$ for $i=1,2$ respectively. Hence, in any Floer isotopy, $u_1$ or $u_2$ is constant. \\ $ $\\
\underline{Case 1:} $\mu(y_1) = 2$ and  $\mu(y_2) = -1$. \\
    Since $u_1$ is constant, the first strand remains unchanged, hence $p = q = 0$. We have that $$E(u_1,u_2) = E(u_2) \geq\left(\mathcal A_{H_{2k}}(y_2)-\mathcal A_{H_{2k}}(x_2)\right)>0.$$ Note that we only need to consider the orbits $y_2$ where this difference is positive. Using Lemma~\ref{lemma: action} a short computation shows that $$\mathcal A_{H_{2k}}(x_2) = -k^2-2k+{\mu(x_2)\over 8}k+O(1).$$ Moreover, we have \begin{align*}\begin{split}\mathcal A_{H_{2k}}(y_2) &=-k^2+{k\over 2}(\mu(y_2)-2)+{\epsilon_1+\epsilon_3\over 8}\left({k^2\over 2k}-2k\right) -{\epsilon_2+\epsilon_4\over 8}\left({k^2\over 2k}-2k\right)- k+O(1)\\ & =-k^2-2k+{\mu(y_2)\over 8}k+O(1).\end{split}\end{align*} Hence, we find the estimate $$
            \mathcal A_{H_{2k}}(y_2)-\mathcal A_{H_{2k}}(x_2) ={\mu(y_2)-\mu(x_2)\over8}k +O(1) = {1\over 8}k + O(1).
 $$Hence, we find $E(u_1,u_2) \geq {1\over 8}k+O(1)$. \\$ $
    \\
    \underline{Case 2:} $\mu(y_1) = 1$ and $\mu(y_2) = -2$. \\ This case is very similar to the first one; however, now $u_2$ is constant and we have $p=q=0$. We again find $$E(u_1,u_2) = E(u_1) \geq \left ( \mathcal A_{H_{2k}}(x_1)-\mathcal A_{H_{2k}}(y_1)\right)>0$$ and only need to consider those $y_1$ where the action difference is positive. Again using Lemma~\ref{lemma: action} a short computation shows $$\mathcal A_{H_{2k}}(x_1) = -k^2-2k+{\mu(x_1)\over 8}k+O(1).$$ Similarly as in the first case we find \begin{align*}\begin{split}\mathcal A_{H_{2k}}(y_1) &= -k^2-2k+{\mu(y_1) \over 8}k+O(1).\end{split}\end{align*} We find $$\mathcal A_{H_{2k}}(x_1)-\mathcal A_{H_{2k}}(y_1) = {\mu(x_1)-\mu(y_1)\over 8}k+ O(1) = {1\over 8}k+O(1).$$ Hence, we find $E(u_1,u_2)\geq {1\over 8}k+O(1)$.\end{proof}
    Putting everything together, we can now prove Theorem~\ref{thm:entropy}. \begin{proof}[Proof of Theorem~\ref{thm:entropy}]
        Consider the diffeomorphism $\phi_{2k}$ and let $x^{(k)}_1,x^{(k)}_2\in \mathcal P({H_{2k}})$ be the contractible orbits as in Lemma~\ref{lemma: winding numbers}. In the proof of Lemma~\ref{lemma: winding numbers} we saw that $$\text{spell}(x_1^{(k)},x_2^{(k)}) = a^{2k+1}b^{2k+1}a^{-2k-1}b^{-2k-1}$$ and hence as discussed in Section~\ref{subsection: torus braids} we have $$\mathcal B(H_{2k},x_1^{(k)},x_2^{(k)}) = \beta_{2k+1,2k+1}.$$ By Lemma~\ref{lemma: entropy of the braids} the entropy of the braid is bounded by $$h(\beta_{2k+1,2k+1})\geq {2\over 5}\log(2k+1).$$ In Lemma~\ref{lemma: stability radius computation} we saw that $\mathcal E(H_{2k},J,x_1^{(k)},x_2^{(k)})\to \infty$ as  $k\to \infty$ for every almost complex structure $J$. Since the framing of the braid $B(H_{2k},x_1^{(k)},x_2^{(k)})$ is $$f_{2k}: x_1^{(k)}\mapsto \mu(x_1^{(k)}) = 2,\quad  x_2^{(k)}\mapsto \mu(x_2^{(k)}) = -2$$ we have that $\mathcal E^{\operatorname{int}}(H_{2k},J,x_1^{(k)},x_2^{(k)}) = \infty$ since there are no index $1$ Floer cylinders connecting the orbits. Using the definition of the stability radius (\ref{eq: stability radius injective framing}) for injective framings we find $${1\over 16}k+O(1) \leq \mathcal S\left (H_{2k},x_1^{(k)},x_2^{(k)}\right )\to \infty \quad (k\to \infty).$$
        For a non-degenerate Hamiltonian diffeomorphism $\psi$ such that $$d_{\mathrm{H}}(\phi_{2k},\psi)< \mathcal S (H_{2k},x_1^{(k)},x_2^{(k)} ),$$ there exists a Hamiltonian $G$ generating $\psi$ such that $E(H-G)< \mathcal S (H_{2k},x_1^{(k)},x_2^{(k)} )$. After the braid stability Theorem~\ref{thm: braid stability} we have that $\dim \operatorname{BCF}_{\beta_{2k+1,2k+1},f_{2k}}(G)\geq 1$ which means that there are periodic orbits $y_1,y_2\in \mathcal P(G)$ such that $\mathcal B(G,y_1,y_2) = \beta_{2k+1,2k+1}$. It follows that $${1\over 5}\log(2k+1)< h(\beta_{2k+1,2k+1}) \leq h_{\operatorname{top}}(\psi).$$ Now, more generally, if $\psi$ is possibly a degenerate Hamiltonian diffeomorphism, there are arbitrarily small $C^\infty$ perturbations of $\psi$ yielding a non-degenerate Hamiltonian diffeomorphism $\tilde \psi$ such that still $d_{\mathrm{H}}(\phi_{2k},\tilde \psi)<\mathcal S (H_{2k},x_1^{(k)},x_2^{(k)} )$. For $\tilde \psi$ the entropy estimate holds. Since topological entropy is $ C^\infty$-upper-semicontinuous \cite{Yo87,Ne89}, the estimate holds for $\psi$ as well. This means that for $\alpha_k = {1\over 5}\log(2k+1)$ we have $${1\over 16}k+O(1) \leq d_{\mathrm{H}}(\operatorname{Ent}_{\leq \alpha_k}(T^2,\omega),\phi_{2k}) ,$$ for $k$ large enough.
    \end{proof}
\section{Appendix}\label{section: appendix}
Next, we prove Lemma~\ref{lemma: index of orbits}. We use the Maslov index as defined by Robbin and Salamon in \cite{RS93}. This index assigns to a path of Lagrangian subspaces $\Lambda:[0,1]\to \mathcal L(n)\subset \operatorname{Gr}(n,\mathbb R^{2n})$ and a reference subspace $V\in \mathcal L(n)$ a half-integer which is invariant under homotopy of this path relative to the endpoints. The index counts the intersections of $\Lambda$ and $V$ as follows. Namely, if we assume that the crossings are regular, for every $t_0\in [0,1]$ there is a well-defined quadratic form $$Q(\Lambda,t_0):\Lambda(t_0)\to \mathbb R, \quad v\mapsto {d\over dt}|_{t=t_0}\omega(v,w(t)),$$ where we fix a Lagrangian complement $W$ of $\Lambda(t_0)$ and choose $w(t)\in W $ such that $v+w(t)\in \Lambda(t)$ for small $t$. Let $V$ be a further Lagrangian subspace. A $t\in [0,1]$ such that $\Lambda(t) \cap V$ is non-trivial is called \emph{crossing time} and define the associated \emph{crossing form} by $$\Gamma(\Lambda,V,t_0) = Q(\Lambda,t_0)\,|\,_{\Lambda(t_0)\cap V}.$$ The Maslov index of a path of Lagrangians is then given by $$\mu(\Lambda,V) = {1\over 2}\text{sgn}(\Gamma(\Lambda,V,0))+\sum_{0<t<1}\text{sgn}(\Gamma(\Lambda,V,t)) + {1\over 2}\text{sgn}(\Gamma(\Lambda,V,1)).$$ This index is, in particular, useful since it is well-defined even when $\Lambda (1) = V$, which allows us to consider paths of symplectic matrices which end in a symplectic matrix with eigenvalue $1$. The connection to the Conley--Zehnder index is the following. Since we consider Hamiltonian diffeomorphisms in $T^2 = \mathbb R^2/\mathbb Z^2$, the standard trivialization of the tangent bundle induces a representation of the differentials $\{d(\phi_k^t)_{\phi^t_k(x_0)}\}$ as a path of symplectic matrices. For a periodic orbit $x\in \mathcal P(H)$ we let $\Psi_x:[0,1]\to Sp(2,\mathbb R)$ denote this path. Furthermore, let for every path $A:[0,1]\to Sp(2,\mathbb R)$  $$\mathrm{Graph}(A(t)) = \{(v,A(t)v)\,|\,v\in \mathbb R^2\}$$ denote the path of the graphs. In our case, the Conley--Zehnder index is then given by paths of Lagrangians in $\mathbb R^4$ $$\mu_{CZ}(x) = \mu(\mathrm{Graph}(\Psi_x),\Delta_{\mathbb R^2}),$$ where we choose $\omega \oplus -\omega$ on $\mathbb R^4$. For more details, see \cite{RS93}. Note that for every $(0,0)\neq (v,v)\in \mathrm{Graph}(A(t))\cap \Delta_{\mathbb R^2}$  we have $A(t)v = v$ and $A(t)$ has the  eigenvalue $1$. Hence, the index $\mu(\mathrm{Graph}(A(t)),\Delta_{\mathbb R^2})$ sums the intersection form over all times $t$ where $A(t)$ has the eigenvalue $1$ (with half counts on the boundary). This index is invariant under so-called \emph{stratum homotopies} (see \cite{RS93}). In particular, this means the following. Given a homotopy $A_s(t)$ of paths of matrices for $s,t\in [0,1]$, such that $A_s(0) = id$ and none of the matrices $A_{s}(1)$ has $1$ as an eigenvalue, then $\mu(\mathrm{Graph}(A_s(t)),\Delta_{\mathbb R^2}) = \mu(\mathrm{Graph}(A_{0}(t)),\Delta_{\mathbb R^2})$. We will make use of this shortly.

\begin{proof}[Proof of Lemma~\ref{lemma: index of orbits}]

Let us for now fix a large positive integer $k\in \mathbb Z_{\geq 1}$ and define $\lambda = 4(k+\varepsilon)$. For every fixed point $z =z_0 = (x_0,y_0)$ with intermediate points $(x_i,y_i)$ let $$\epsilon_1 = -\text{sign}(\dot s_{k}(y_0)),\quad \epsilon_2 = \text{sign}(\dot s_{k}(x_1)),\quad \epsilon_3 = -\text{sign}(\dot s_{k}(y_2)),\quad \epsilon_4 = \text{sign}(\dot s_{k}(x_3)).$$
For such a fixed point, using the standard trivialization of $TT^2$, the path $(d\phi_k^t)_{\phi_k^t(z_0)}$ represented as matrices with respect to the standard basis is homotopic through symplectic matrices to the path \begin{align*}\begin{split} P_\epsilon(t) &=\begin{pmatrix}
    1& 0\\\lambda\epsilon_4t & 1
\end{pmatrix} \begin{pmatrix}
    1 & \lambda\epsilon_3t\\ 0 & 1
\end{pmatrix}\begin{pmatrix}
    1& 0 \\ \lambda\epsilon_2t&1
\end{pmatrix}\begin{pmatrix}
    1 & \lambda\epsilon_1t\\0&1
\end{pmatrix}\\ & = \begin{pmatrix}
    1+t^2\lambda^2\epsilon_2\epsilon_3 &t\lambda(\epsilon_1+\epsilon_3+t^2\lambda^2\epsilon_1\epsilon_2\epsilon_3)\\t\lambda(\epsilon_2+\epsilon_4+t^2\lambda^2\epsilon_2\epsilon_3\epsilon_4)& 1 + t^2\lambda^2(\epsilon_1\epsilon_2+\epsilon_3\epsilon_4+\epsilon_4\epsilon_1)+t^4\lambda^4\epsilon_1\epsilon_2\epsilon_3\epsilon_4
\end{pmatrix}\end{split}\end{align*}
We will distinguish cases and compute the index in each case. Unfortunately, for 12 of the 16 choices, the path $P_\epsilon(t)$ gives rise to degenerate intersection forms. Hence, we need to perturb the path.  Let us first discuss the four paths where we do not need to perturb $P_\epsilon(t)$. For every $\epsilon\in \{\pm 1\}^4$ we will show that $$\mu(z) = {(\epsilon_1+\epsilon_3)-(\epsilon_2 + \epsilon_4)\over 2}.$$
The matrix $P_\epsilon(t)$ has 1 as an eigenvalue for $t\in [0,1]$ if and only if $$\text{tr}(P_\epsilon(t)) = 2 + t^2\lambda^2(\epsilon_1+\epsilon_3)(\epsilon_2+\epsilon_4) + t^4\lambda^4\epsilon_1\epsilon_2\epsilon_3\epsilon_4 = 2$$ which means that $t=0$ or $$t = \sqrt{-{(\epsilon_1+\epsilon_3)(\epsilon_2+\epsilon_4)\over \lambda^2\epsilon_1\epsilon_2\epsilon_3\epsilon_4}}.$$
We here compute the intersection form of the crossing time $t=0$ in generality, as it will appear in any of the cases below. Suppose that $$\dot  P_\epsilon(0) = \begin{pmatrix}
    0 & p \\ q & 0
\end{pmatrix}.$$
Since $P_\epsilon(0) = id$ we have $\mathrm{Graph}(P_\epsilon(0)) = \Delta\subset \mathbb R^4$. Let us choose the Lagrangian complement $W = 0\times \mathbb R\times \mathbb R\times 0$ of $\Delta_{\mathbb R^4}$. For every $t\in \mathbb R$ close to 0 and $v=(x,y,x,y) \in \Delta_{\mathbb R^4}$ let $w(t) = (0,\eta(t),\xi(t),0)\in W$ such that $v+w(t) \in \mathrm{Graph}(P_\epsilon(t))$. This means that $$\begin{pmatrix}
    x+\xi(t)\\ y
\end{pmatrix} = P_\epsilon(t)\begin{pmatrix}
    x\\y+\eta(t)
\end{pmatrix}.$$ Differentiating at $t=0$ and using that $\eta(0)= 0$ as well as $P_\epsilon(0) = id$ we find $$\begin{pmatrix}
    \dot \xi(0)\\ 0
\end{pmatrix} = \dot P_\epsilon(0)\begin{pmatrix}
    x\\ y
\end{pmatrix}+\begin{pmatrix}
    0\\\dot \eta(0)
\end{pmatrix}.$$ Hence, we find $$\dot \xi(0) = py\quad \text{and}\quad \dot \eta(0) = -qx.$$
We now have that \begin{align*}
    Q(x,y) &= {d\over dt}\,\big|\,_{t=0} \omega(v,w(t)) \\ & =\dot \eta(0)x+\dot \xi(0)y\\ & = -qx^2+py^2.
\end{align*}
\underline{Case:} $\epsilon=\pm(1,1,1,1)$.\\
Then the only $t\in [0,1]$ such that $1$ is an eigenvalue of $P_\epsilon(t)$ is $t=0$. We calculate that $$\dot P_\epsilon(0) = \begin{pmatrix}
    0& \pm 2\lambda\\ \pm2\lambda& 0
\end{pmatrix}$$ and find \begin{align*}
    Q(x,y) = \mp 2\lambda x^2\pm2\lambda y^2
\end{align*}
which is a quadratic form of signature $0$. Thus, the fixed point $z$ has index $\mu(z) = 0$, which agrees with the given formula. \\
\underline{Case:} $\epsilon = \pm (1,-1,1,-1)$.\\
 In this case we find that $$\dot P_\epsilon(0) = \begin{pmatrix}
    0 & \pm2\lambda\\\mp2\lambda& 0
\end{pmatrix}$$ and a similar computation as above shows that the crossing form at $t = 0$ is given by $$Q(x,y) = \pm 2\lambda x^2\pm2\lambda y^2$$ which is a quadratic form of signature $\pm2$. Hence, the crossing time $t=0$ contributes $\pm1$ to the index. However, there is another crossing time $$t_1 = {2\over \lambda},$$ which lies in $[0,1]$ when $k$ is large enough. We have that $$P_\epsilon(t) = \begin{pmatrix}
    1-t^2\lambda^2 & \pm 2t\lambda\mp t^3\lambda^3\\\mp2t\lambda\pm t^3\lambda^3 & 1 - 3t^2\lambda^2 + t^4\lambda ^4
\end{pmatrix}$$ and hence $$P_\epsilon(t_1) = \begin{pmatrix}
    -3 & \mp4\\ \pm 4& 5
\end{pmatrix}.$$ Note that again $W = 0\times \mathbb R\times \mathbb R \times 0$ is a Lagrangian complement to $\mathrm{Graph}(P_\epsilon(t_1))$. We find that the eigenvector of $P_\epsilon(t_1)$ with eigenvalue $1$ is given by $(\mp 1,1)$ and the intersection $\mathrm{Graph}(P_\epsilon(t_1)) \cap \Delta $ is spanned by the vector $v=(\mp 1,1,\mp 1,1)$. For $t\in \mathbb R$ close to $t_1$ we let $w(t) = (0,\eta(t),\xi(t),0)$ such that $v +w(t) \in \mathrm{Graph}(P_\epsilon(t))$. Similar to above, this means that $$\begin{pmatrix}
    x+\xi(t)\\ y
\end{pmatrix} = P_\epsilon(t)\begin{pmatrix}
    x\\y+\eta(t)
\end{pmatrix}.$$ Differentiation at $t_1$ and using $\eta(t_1)=0$ that$$\begin{pmatrix}
    \dot \xi(t_1) \\ 0 
\end{pmatrix} = \begin{pmatrix}
    -4\lambda &\mp10\lambda\\ \pm10\lambda & 20\lambda
\end{pmatrix}\begin{pmatrix}
    \mp 1\\1
\end{pmatrix}+\begin{pmatrix}
    -3&\mp 4\\\pm4 & 5
\end{pmatrix}\begin{pmatrix}
    0 \\ \dot \eta(t_1)
\end{pmatrix}.$$ We find that $$ \dot \eta(t_1) = -2\lambda\quad \text{and}\quad  \pm \dot \xi(t_1)=2\lambda.$$ Now, we have that \begin{align*}\begin{split}
    \Gamma(\mathrm{Graph}(P_\epsilon(t)),\Delta,t_1)(s) & = Q(\mathrm{Graph}(P_\epsilon(t_1)))|_{\mathrm{Graph}(P_\epsilon(t_1))\cap \Delta}(s) \\ & = s^2{d\over dt}\big|_{t = t_1}\omega(v,w(t))\\ & =  s^2(\mp \dot \eta(t_1)+\dot \xi(t_1) )\\ & = \pm s^24\lambda.
\end{split}\end{align*}
We therefore find that at crossing time $t= t_1$, the crossing form contributes $\pm 1 $ to the index. Finally, we find $\mu(z)= \pm 2$, which agrees with the given formula in the statement of the lemma.
\\ 
Now consider the remaining cases where $\epsilon_1+\epsilon_3 = 0$ or $\epsilon_2+\epsilon_4 = 0$. Unfortunately, in these cases, the crossing form of $P_\epsilon$ at $t=0$ is degenerate. Since $$\text{tr}(P_\epsilon(t)) = 2 + t^4\lambda^4\epsilon_1\epsilon_2\epsilon_3\epsilon_4$$ we have that $\text{tr}(P_\epsilon(t)) = 2$ if and only if $
t=0$. Hence, $P_\epsilon(t)$ has $1$ as an eigenvalue if and only if $t=0$, which is the only (degenerate) crossing time of the Lagrangian path associated to $P_\epsilon(t)$. We perturb the path $P_\epsilon(t)$ such that this remains true and $t =0$ is the only crossing time but is now regular. \\
Let $\alpha,\beta\in \mathbb R$ be small enough in absolute value and to be fixed later. We consider the perturbed path of symplectic matrices $$ P_{\epsilon,\alpha,\beta}(t) = \begin{pmatrix}
    1 & 0 \\\alpha t & 1\end{pmatrix}P_\epsilon(t)\begin{pmatrix}
        1 & \beta t\\ 0 & 1
    \end{pmatrix}.$$ Note that $$\text{tr}(P_{\epsilon,\alpha,\beta}(1)) = \text{tr}(P_\epsilon(1)) + \alpha (P_\epsilon(1))_{12}+\beta (P_\epsilon(1))_{21}+\alpha \beta (P_\epsilon(1))_{11}.$$
    Since $$\text{tr}(P_\epsilon(1))= 2 + \lambda^2(\epsilon_1+\epsilon_3)(\epsilon_2+\epsilon_4) + \lambda^4\epsilon_1\epsilon_2\epsilon_3\epsilon_4\neq 2$$ when $k$ is large enough, we have that $\text{tr}(P_{\epsilon,\alpha,\beta}(1))\neq 2$ when $|\alpha|,|\beta|$ are small enough. Note that $ P_{\epsilon,\alpha,\beta}(0) = id$. It follows that $P_\epsilon$ and $P_{\epsilon,\alpha,\beta}$ are stratum homotopic via the homotopy given by $$I:[0,1]\times [0,1] \to Sp(2,\mathbb R),\quad (s,t)\mapsto  \begin{pmatrix}
    1 & 0 \\\alpha st & 1\end{pmatrix}P_\epsilon(t)\begin{pmatrix}
        1 & \beta st\\ 0 & 1
    \end{pmatrix}. $$
    A short computation shows \begin{align*}\begin{split}\text{tr}(P_{\epsilon,\alpha,\beta}(t)) =& 2+t^4\lambda^4\epsilon_1\epsilon_2\epsilon_3\epsilon_4+ \\ &+ t^2  \left [\alpha \beta +\alpha\lambda(\epsilon_1+\epsilon_3)+ \beta\lambda(\epsilon_2+\epsilon_4)
    \right ]\\ & +t^4\left [\alpha \lambda^3\epsilon_1\epsilon_2\epsilon_3 + \beta \lambda^3\epsilon_2\epsilon_3\epsilon_4+\alpha\beta\lambda^2\epsilon_2\epsilon_3\right],\end{split}\end{align*} where we use that $\epsilon_1+\epsilon_3 = 0$ or $\epsilon_2 + \epsilon_4 =0$.  Hence, we have $\text{tr}( P_{\epsilon,\alpha,\beta}(t)) = 2$  if and only if $t = 0$ or $t = \sqrt{D_{\epsilon,\alpha,\beta}}$ and $t\in [0,1]$ where $$D_{\epsilon,\alpha,\beta} = {-{\alpha \beta +\alpha \lambda(\epsilon_1+\epsilon_3)+\beta\lambda(\epsilon_2+\epsilon_4)\over \lambda^4\epsilon_1\epsilon_2\epsilon_3\epsilon_4+\left [\alpha\lambda^3\epsilon_1\epsilon_2\epsilon_3 + \beta \lambda^3\epsilon_2\epsilon_3\epsilon_4+\alpha\beta\lambda^2\epsilon_2\epsilon_3\right]}}.$$ $ $\\
    \underline{Case:} $\epsilon_1 + \epsilon_3 = 0$ and $\epsilon_2+\epsilon_4 \neq 0$.\\ 
   Choose $\delta>0$ small enough and set $\alpha = 0$ and $\beta = \delta(\epsilon_2+\epsilon_4)\epsilon_1\epsilon_2\epsilon_3\epsilon_4$. Note that in this case we have $\epsilon_1\epsilon_3 = -1$ and $\epsilon_2 \epsilon_4 = 1$ and hence $\beta = -\delta (\epsilon_2+\epsilon_4)$. When $\delta$ is small enough, $D_{\epsilon,\alpha,\beta}<0$, so $ t=0$ is the only crossing time. We have that $$\dot P_{\epsilon,\alpha,\beta} (0) = \dot P_\epsilon(0) + \begin{pmatrix}
       0 & \beta \\ 0 & 0
   \end{pmatrix} = \begin{pmatrix}0 & -\delta(\epsilon_2+\epsilon_4) \\ 
       \lambda(\epsilon_2+\epsilon_4) & 0
   \end{pmatrix}.$$ As above, we find that the associated intersection form at $t = 0$ is given by $$Q(x,y) = -(\epsilon_2+\epsilon_4)\left[\lambda x^2+\delta y^2\right]$$ which has signature $ -(\epsilon_2+\epsilon_4)$ since $\delta,\lambda>0$. This contributes $-(\epsilon_2+\epsilon_4)/2$, and since $t=0$ is the only crossing time, the index is $-(\epsilon_2+\epsilon_4)/2$, which agrees with the formula. \\
   \underline{Case:} $\epsilon_1 + \epsilon_3 \neq 0$ and $\epsilon_2+\epsilon_4 = 0$.\\
    This case is analogous to the last case, but we set $\beta = 0$ and $\alpha = -\delta (\epsilon_1+\epsilon_3)$ for small enough $\delta>0$. We find that $D_{\epsilon,\alpha,\beta}<0$ and hence that $t = 0$ is the only crossing time. We now have $$\dot P_{\epsilon,\alpha,\beta}(0) = \begin{pmatrix}
        0 & \lambda(\epsilon_1+\epsilon_3)\\ -\delta (\epsilon_1+\epsilon_3) & 0
    \end{pmatrix}.$$ The intersection form at $t = 0$ is then given by $$Q(x,y) = (\epsilon_1+\epsilon_3)\left[ \delta x^2 +\lambda y^2\right ],$$ which has signature $(\epsilon_1+\epsilon_3)$ since $\delta,\lambda>0$. Hence, the index is given by $(\epsilon_1+\epsilon_3)/2$, which agrees with the formula stated in the lemma.\\ $ $\\
    \underline{Case:} $\epsilon_1+\epsilon_3= \epsilon_2+\epsilon_4 =0$.
    \\ In this case we choose $\delta >0$ small enough and set $\alpha = \beta = \delta$. We have that $\epsilon_1\epsilon_2\epsilon_3\epsilon_4 = 1$, and hence, for $\delta>0$ sufficiently small, we find $D_{\epsilon,\alpha,\beta}<0$ and that $t = 0$ is the only crossing time. Moreover, we have $$\dot P_{\epsilon,\alpha,\beta}(0) = \begin{pmatrix}
        0 & \delta \\ \delta & 0
    \end{pmatrix}$$ and the intersection form at $t =0$ is given by $$Q(x,y) = -\delta x^2+\delta y^2$$ which has signature $0$. Hence, the index is $0$ in this case, which agrees with the formula given in the lemma.
    \end{proof}
\bibliographystyle{alpha}
\bibliography{biblio}
\end{document}